\documentclass[12pt]{article}
\usepackage[left=2cm,top=2cm,right=2cm,bottom=2cm]{geometry}

\usepackage[cp1251]{inputenc}
\usepackage[english]{babel}

\usepackage{amsfonts,epsfig,epstopdf}
\usepackage{amssymb,amsmath,enumerate,float,amsthm}
\usepackage{graphicx,bbm,xcolor,latexsym}

\usepackage{chngcntr}
\counterwithin{figure}{section}

\usepackage{hyperref}
\usepackage{enumitem}
\usepackage{mathrsfs}
\usepackage{mathtools}
\usepackage{amsthm}

\numberwithin{equation}{section}
\pdfoutput=1

\ifpdf
\DeclareGraphicsExtensions{.eps,.pdf,.png,.jpg}
\else
\DeclareGraphicsExtensions{.eps}
\fi

\usepackage{chngcntr}
\counterwithin{figure}{section}

\newtheorem{definition}{Definition}[section]
\newtheorem{lemma}[definition]{Lemma}
\newtheorem{theorem}[definition]{Theorem}
\newtheorem{proposition}[definition]{Proposition}

\newtheorem{remark}[definition]{Remark}

\DeclareMathAlphabet\mathbit
    \encodingdefault\rmdefault\bfdefault\itdefault
\DeclareOldFontCommand{\bi}{\normalfont\bfseries\itshape}{\mathbit}

\newcommand{\be}{\begin{equation}}
\newcommand{\ee}{\end{equation}}

\def\fakebold#1{\relax\ifvmode\leavevmode\fi%
\ifmmode%
\setbox0=\hbox{$#1$}%
\else%
\setbox0=\hbox{#1}%
\fi%
\kern-.02em\copy0 \kern-\wd0%
\kern .04em\copy0 \kern-\wd0%
\kern-.0125em\raise.02em\box0%
}%

\renewcommand{\geq}{\geqslant}
\renewcommand{\leq}{\leqslant}
\newcommand{\mathd}{\mathrm{d}}

\newcommand{\tmem}[1]{{\em #1\/}}
\newcommand{\tmmathbf}[1]{\ensuremath{\boldsymbol{#1}}}
\newcommand{\tmop}[1]{\ensuremath{\operatorname{#1}}}

\newcommand{\tmtextit}[1]{{\itshape{#1}}}

\newcommand{\tmtextrm}[1]{{\rmfamily{#1}}}
\newcommand{\ptl}{\partial}

\definecolor{myred}{RGB}{160,0,0}
\definecolor{mygreen}{RGB}{0,160,0}
\definecolor{myblue}{RGB}{0,0,160}
\hypersetup{
    colorlinks = true,
    linkcolor = {myred},
    citecolor = {mygreen},
    urlcolor  = {myblue},
}
\graphicspath{{pictures/}}
\numberwithin{equation}{section}
\numberwithin{figure}{section}
\numberwithin{table}{section}
\definecolor{darkgreen}{rgb}{0,0.5,0}
\newcommand{\RED}{} % Without red comments

\definecolor{manchester}{rgb}{.42,.17,.58}
\newcommand*\link[1]{\hspace*{0em plus 1fill}\makebox{#1}}
\renewcommand{\qedsymbol}{$\blacksquare$}

\title{Vertex Green's functions of a quarter-plane. \protect\\ Links between the functional equation, additive crossing and Lam{\'e} functions}

\author{Rapha\"{e}l C. Assier$^{*}$ and Andrey V. Shanin$^{\dagger}$\\
	\footnotesize{$^{*}$ Department of Mathematics, University of Manchester, Oxford Road, Manchester, {\rm M13 9PL}, UK}\\
	\footnotesize{$^{\dagger}$ Department of Physics (Acoustics Division), Moscow State University, Leninskie Gory, {\rm 119992}, Moscow, Russia}
}

\begin{document}

\maketitle

%\newpage

%\hrule
\begin{abstract}
  In our previous work (Assier \& Shanin, QJMAM, 2019), we gave a new spectral
  formulation in two complex variables associated with the problem of \RED{plane-wave}
  diffraction by a quarter-plane. In particular, we showed that the unknown
  spectral function satisfies a condition of additive crossing about its
  branch set. In this paper, we study a very similar class of spectral
  problem, and show how the additive crossing can be exploited in order to
  express its solution in terms of Lam{\'e} functions. The solution\RED{s} obtained can be thought of as tailored vertex Green's function\RED{s} whose behaviour\RED{s} in the near-field \RED{are} directly related to the eigenvalues of the Laplace--Beltrami operator. This is important since the correct near-field behaviour at the tip of the quarter-plane had so far never been obtained via a \RED{multivariable complex analysis} approach. %The present work is also an emphasis of the importance of the additive crossing property.

%\smallskip
%\noindent \textbf{Keywords.} Acoustics, Diffraction, Quarter-plane, Lam{\'e} functions, Laplace-Beltrami
\end{abstract}
%\bigskip
%\hrule

%%%%%%%%%%%%%%%%%%%%%%%%%%%%%%%%%%%%%%%%%%%%%%%%%%%%%%%%%%%%%%%%%%%%%%%%
%%%%%%%%%%%%%%%%%%%%%%%%%%%%%%%%%%%%%%%%%%%%%%%%%%%%%%%%%%%%%%%%%%%%%%%%

\section{Introduction}

\RED{The long-term motivation behind the present work is} the problem of 
\RED{plane-wave} diffraction by 
a quarter-plane (an obstacle having 
zero thickness and the shape of a plane angular sector with opening angle equal to~$\pi /2$). 

%This is a three-dimensional scalar and stationary scattering problem governed by the Helmholtz equation. The scatterer is a quarter-plane, i.e. it is an obstacle having 
%zero thickness and the shape of a plane angular sector with opening angle equal to~$\pi /2$. Dirichlet boundary conditions are imposed on the two faces of the scatterer. 

A \RED{comprehensive} literature review dedicated to this \RED{important unsolved canonical} problem can be found in \RED{one of} our previous paper\RED{s} \cite{Assier2018a}. One of the ways to tackle this problem, which we will continue to develop in the present work, is to reformulate it as a two-complex-variables functional equation of the Wiener--Hopf type. Various (mostly unsuccessful) attempts (see e.g.\ \cite{Radlow1965})  to solve this functional equation were reviewed in \cite{Assier2018a}. 
%One of the approaches applied to the 
%problem is the derivation of 2D Wiener--Hopf problem (which is remarkably simple,
%see e.g.\ \cite{Radlow1965}) and the attempts to solve it. Here the authors try to develop  
%this direction. 

%\GRE{One should note that the two-complex-variables Wiener--Hopf (2DWH) problem differs strongly from its 
%one-complex-variable analogue (1DWH). The unknown functions for the 1DWH problem are some functions analytic 
%in the upper or lower half-plane of a single complex spectral argument. It is known 
%that inverse Fourier transforms of such functions are equal to zero on the positive or negative 
%half-axis of the real physical coordinate variable. Such unknown functions can hence be referred to as 1/2-based. 

%By analogy, for the 2DWH problem one obtains two unknown functions (of two complex spectral arguments), one of which being 1/4-based, and another being 3/4-based. The 1/4-basedness means that 
%a 2D inverse Fourier transform of such a function is not zero only on a single quadrant of the 
%physical coordinate plane. A 2D inverse Fourier transform of a 3/4-based function should be zero on 
%one quadrant of the plane and non-zero on the remaining three quadrants.

%While the criterion for 1/4-basedness is well-known (the function should be analytic in a product of two half-planes), there is no known 
%criterion for 3/4-basedness.} 
\RED{In this same paper, we also showed that one of the two unknown spectral functions of the two-complex-variables Wiener--Hopf problem (2DWH) had to satisfy the property of \textit{additive crossing} about its branch lines (two-complex-variables analog to branch points). It allowed us to derive a new spectral formulation for the plane-wave-incidence quarter-plane problem. }
%introduced the concept 
%of {\em additive crossing\/} of branch lines and showed that it was deeply connected with the 3/4-basedness
%of the unknown function. 
However, at this stage, it is not clear how the additive crossing property can be used in practice. With the present work, our aim is to bridge this gap and give a concrete example of usage of the 
additive crossing property for a simpler but related diffraction problem. \RED{In doing so, we also aspire to develop a sophisticated mathematical machinery on which our future investigations can be built.}

\RED{As mentioned previously,} in \cite{Assier2018a} we considered the problem of diffraction of a plane wave 
by a quarter-plane. In the present paper we have in mind a slightly different wave propagation problem.
\RED{We keep the same} scatterer (a Dirichlet quarter-plane), but there is no incident field. Instead, the vertex condition imposed on the field is weakened: the field can grow as any power 
of the radius. Such field behaviour corresponds to an arbitrary configuration of sources located at the \RED{vertex} of the quarter-plane.
This problem is considered in the spectral domain \RED{and} a corresponding functional problem is formulated. \RED{We call it} the \textit{simplified functional problem} \RED{(\textbf{SFP})} 
to separate it from the \textit{functional problem} \RED{(\textbf{FP})} derived \RED{previously} for the problem of plane-wave diffraction.     

For the simplified problem \RED{\textbf{SFP}}, we show that the usage of the additive crossing property leads to \RED{a discrete set of explicit solutions, written} in the form of a product of Lam{\'e} \RED{and Hankel} functions \RED{and indexed by the eigenvalues of the Laplace--Beltrami operator.} \RED{In order to do so, we make a connection between functional equations and ordinary differential equations (ODEs) of Fuchsian type, and provide a concrete way of indenting two-dimensional integration contours in $\mathbb{C}^2$ using the \textit{bridge and arrow} notation introduced in Appendix \ref{app:appA}.} \RED{The solutions obtained 
agree} perfectly with what can be expected from the point of view of the \RED{sphero-conal} separation of variables method \cite{satterwhite,kraus}\RED{, but we obtain it solely by means of complex analysis. This point is emphasised on the diagram of figure \ref{fig:high-lev-diag}.}

\begin{figure}[h!]
	\centering{ \includegraphics[width=0.9\textwidth]{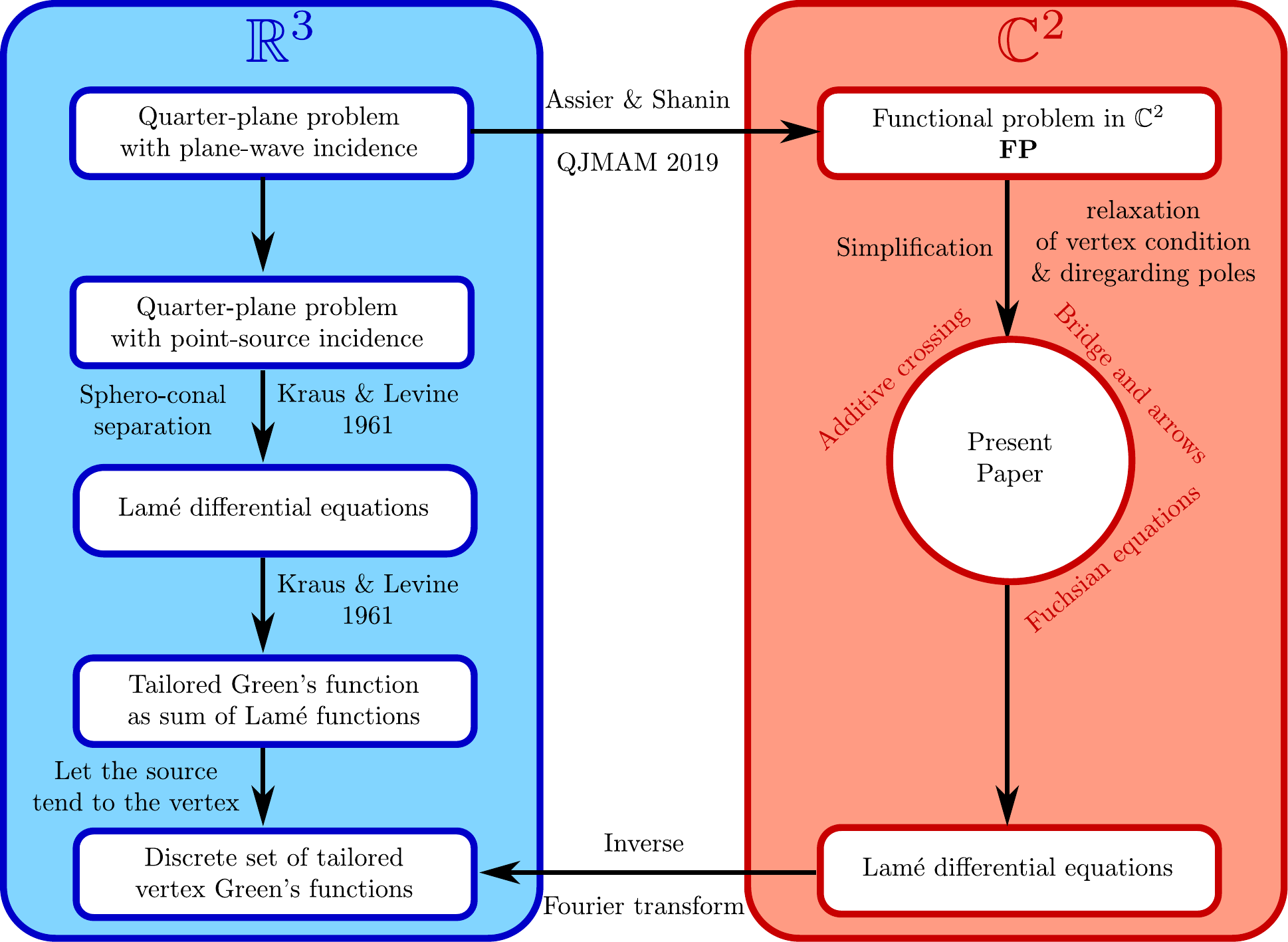}}
	\caption{High-level flow chart diagram of the present paper and links with existing literature}
	\label{fig:high-lev-diag}
\end{figure}
 
The rest of the paper is organised as follows. We formulate the simplified \RED{functional} problem \RED{(\textbf{SFP})} in Section \ref{sec:ProblemFormulation}, while Section \ref{sec:solfuncprob} is dedicated to solving this problem. More precisely, in Section \ref{sec:formulationangularalpha} we reformulate \RED{\textbf{SFP}} using complex angular coordinates \RED{to obtain the so-called \textit{angular simplified formulation} (\textbf{ASF})}. Using additive crossing we show that it can be recast as a stencil equation for functions of two complex variables in Section \ref{sec:stencilformulation}\RED{, leading to the \textit{stencil formulation} (\textbf{StF})}. The stencil equation is solved by separation of (complex) variables in Section \ref{sec:seprarablestencil}. As a result, a one-dimensional stencil equation \RED{(\textbf{1DSt})} is obtained. It is reduced to an ODE in Section \ref{sec:stencil2odetext}, and we show in Section \ref{sec:Lame} that this ODE can be reduced to the Lam{\'e} equation. Finally, in Section \ref{sec:wavefield}, the solution to \textbf{SFP} is transformed into a wave field 
via Inverse Fourier transform \RED{and we give an explicit form in terms of Lam{\'e} and Hankel functions. The content of the present paper, its successive spectral formulations and their connections are summarised by the diagram of figure \ref{fig:paper-diag}.}

\begin{figure}[h!]
	\centering{ \includegraphics[width=0.65\textwidth]{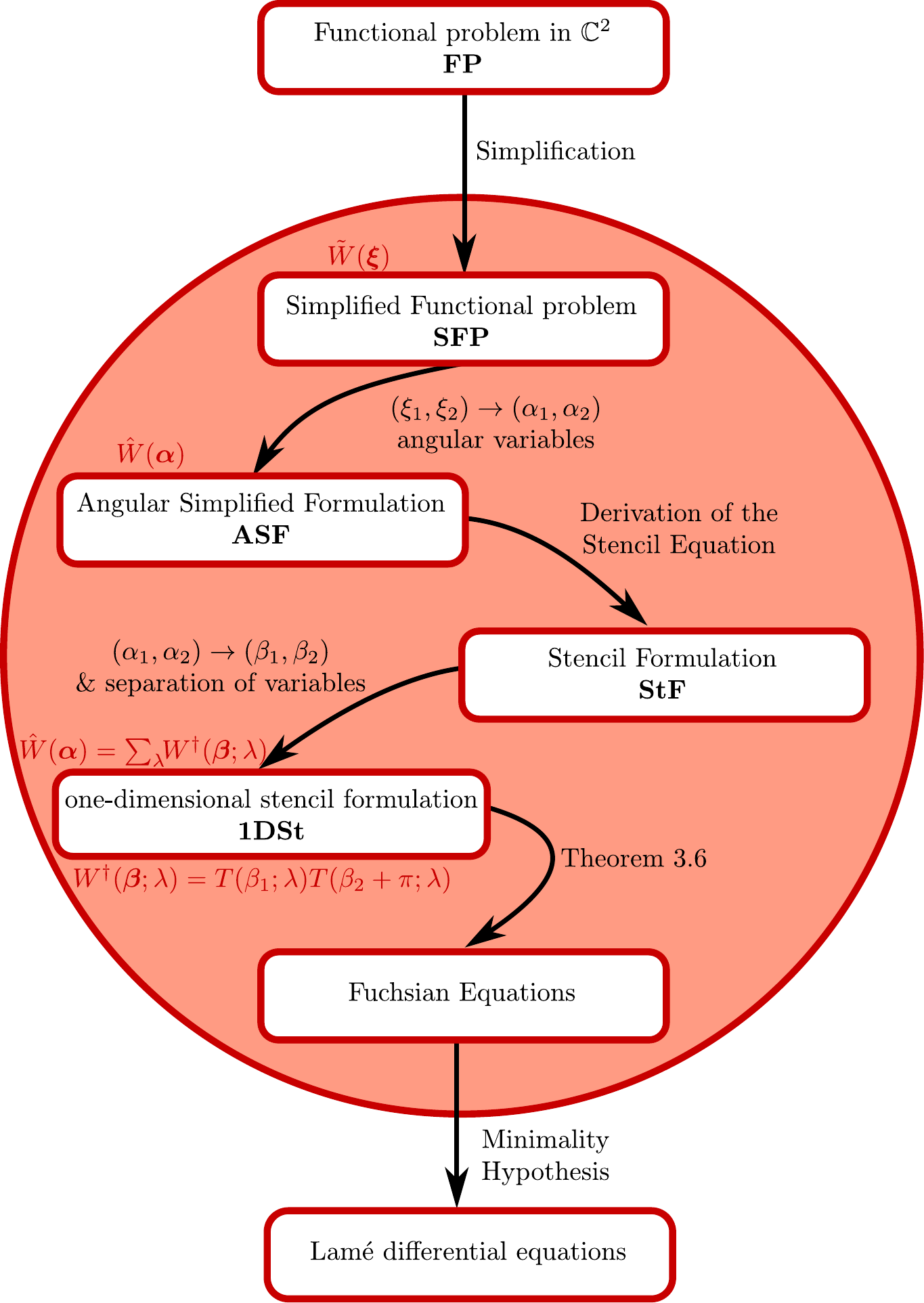}}
	\caption{Diagrammatic description of the present paper}
	\label{fig:paper-diag}
\end{figure}

%At the moment, it is unclear how this functional problem can be solved. In the
%present paper, we propose to consider a slightly simplified but very similar
%spectral formulation that will be given in Section
%\ref{sec:simplifiedformulation}. Using an appropriate change of variable
%(Section \ref{sec:formulationangularalpha}), we reduce this formulation to a
%{\tmem{stencil equation}} in the spectral space (Section
%\ref{sec:stencilformulation}). In Section \ref{sec:solstencil}, we show that
%the solution of the stencil equation is separable in a special spectral
%coordinate system (Section \ref{sec:seprarablestencil}), and that its solution
%can be expressed in terms of solutions to Lam{\'e}'s equation (Section
%\ref{sec:stencil2odetext}). The consequences are shown in Section
%\ref{sec:backtowaves} to be that the physical wave field resulting from the
%simplified spectral formulation can be written in terms of a product of Hankel
%functions of the distance $r$ to the vertex of the quarter-plane and Lam{\'e}
%functions of the sphero-conal angular coordinates (which in our case
%degenerate to the spherical coordinates $(\theta, \varphi)$), which are the
%natural separated variables for the quarter-plane problem as shown in
%{\\cite{satterwhite}}.

%%%%%%%%%%%%%%%%%%%%%%%%%%%%%%%%%%%%%%%%%%%%%%%%%%%%%%%%%%%%%%%%%%
%%%%%%%%%%%%%%%%%%%%%%%%%%%%%%%%%%%%%%%%%%%%%%%%%%%%%%%%%%%%%%%%%%

\section{Problem formulation} \label{sec:ProblemFormulation}

\subsection{Functional problem for the quarter-plane diffraction problem}

As hinted in introduction, the present work is motivated by the canonical problem of diffraction of an incident plane wave $u^{\rm in}$ by a quarter-plane, which we will re-formulate here for completeness. The total field $u^{\rm t}$ satisfies the Helmholtz equation  
\begin{equation}
\Delta u^{\rm t} + k^2 u^{\rm t} = 0,
\label{Helmholtz}
\end{equation}
in the three-dimensional space $(x_1, x_2, x_3) \in \mathbbm{R}^3$, where $\Delta$ is the three-dimensional Laplacian. The wavenumber parameter $k$ is assumed to have a non-zero positive real part and a vanishingly small positive imaginary part. The imaginary part of $k$ can be interpreted as the absorption of the medium. The scatterer is the quarter-plane 
$\tmop{QP} \equiv \left\{ (x_1, x_2, x_3) ,
x_{1, 2} > 0 \text{ and } x_3 = 0 \right\}$.
The total field $u^{\rm t}$ obeys the Dirichlet boundary conditions $u^{\rm t} =0$ on the faces of the quarter-plane. 
The geometry of the problem is illustrated in figure \ref{fig:geom}. 

\begin{figure}[h!]
\centering{ \includegraphics[width=0.3\textwidth]{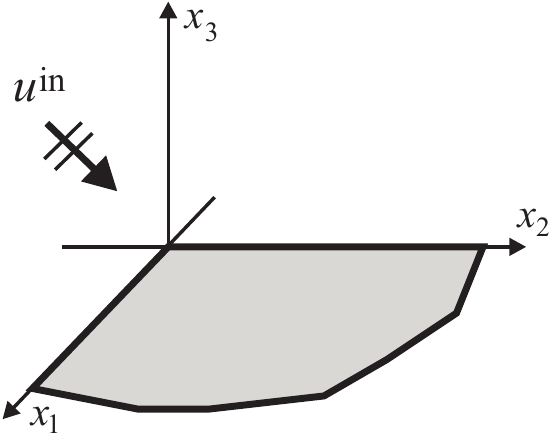}}
   \caption{The quarter-plane problem geometry}
\label{fig:geom}
\end{figure}
The total field $u^{\rm t}$
is a sum of the incident field $u^{\rm in}$ and the scattered field $u$: 
\[
u^{\rm t} = u^{\rm in} + u ,
\] 
and the incident field is a plane wave that can be expressed as 
\[
u^{\rm in} = \exp \left\{ 
i (k_1 x_1 + k_2 x_2 - (k^2 - k_1^2 - k_2^2)^{1/2} x_3) \right\}.
\] 
We assume that the wavenumber components of the incident wave are such that ${\rm Re}[k_{1,2}] > 0$ and ${\rm Im}[k_{1,2}] > 0$. 

For the problem to be well-posed, the scattered field $u$ should obey: 
\begin{itemize}
\item
the Helmholtz equation 
(\ref{Helmholtz}) in the free space, 

\item
the inhomogeneous 
Dirichlet condition $u = - u^{\rm in}$ on QP, 

\item
the radiation condition that can be formulated in the form of the limiting absorption 
principle, 
  
\item
the edge conditions at the two edges of the quarter-plane: 
$ x_1 = x_3 = 0, x_2 > 0 $ and 
$ x_2 = x_3 = 0, x_1 > 0 $, 

\item  
the vertex condition at the tip of the quarter-plane.   
  
\end{itemize}

The edge and vertex conditions take the form of Meixner conditions. They are equivalent to say that the energy-like combination $|\nabla u^{\rm t}|^2  + |u^{\rm t}|^2$ should be locally integrable near the edges and the vertex.

%For an extensive review of previous attempts and
%various techniques used to tackle this problem, as well as of its significance
%to diffraction in general, we kindly refer the reader to the introduction of
%{\cite{Assier2018a}}.

As often
for diffraction problems, it is convenient to work in the Fourier space.
We will consider $x_{1, 2, 3} \in \mathbbm{R}$ and
$\xi_{1, 2} \in \mathbbm{C}$ and will denote $\tmmathbf{\xi}= (\xi_1, \xi_2) \in
\mathbbm{C}^2$ and $\tmmathbf{x}= (x_1, x_2) \in \mathbbm{R}^2$. 
Introduce the double Fourier transform $\mathfrak{F}$ and its inverse
$\mathfrak{F}^{- 1}$ defined by
\begin{eqnarray*}
  \mathfrak{F} [\phi] (\tmmathbf{\xi}, x_3) = \int_{- \infty}^{\infty} \int_{-
  \infty}^{\infty} \phi (\tmmathbf{x}, x_3) e^{i\tmmathbf{\xi} \cdot
  \tmmathbf{x}} \mathd \tmmathbf{x} & \tmop{and} & \mathfrak{F}^{- 1}
  [\tilde{\Phi}] (\tmmathbf{x}) = \frac{1}{4 \pi^2} \int_{- \infty}^{\infty}
  \int_{- \infty}^{\infty} \tilde{\Phi} (\tmmathbf{\xi}) e^{- i\tmmathbf{\xi}
  \cdot \tmmathbf{x}} \mathd \tmmathbf{\xi}
\end{eqnarray*}
for any suitable physical function $\phi (\tmmathbf{x}, x_3)$ and spectral
function $\tilde{\Phi} (\tmmathbf{\xi})$. 

In {\cite{Assier2018a}}, we 
formulated a functional problem, which can be treated as a 2DWH problem. We introduced the 
unknown spectral functions 
$\tilde U$ and $\tilde W$ as 
\begin{equation}
\tilde{U} (\tmmathbf{\xi}) =\mathfrak{F} [u] (\tmmathbf{\xi}, 0^+), 
\qquad 
\tilde{W} (\tmmathbf{\xi}) =\mathfrak{F} \left[ \frac{\partial u}{\partial
  x_3} \right] (\tmmathbf{\xi}, 0^+),
\label{eq:SpecFun}  
\end{equation}  
and showed that these functions obey the {\em functional equation\/} 
\begin{eqnarray}
  \tilde{K} (\tmmathbf{\xi}) \tilde{W} (\tmmathbf{\xi}) & = & i \tilde{U}
  (\tmmathbf{\xi}),  \label{eq:functionaleq}
\end{eqnarray}
where the {\tmem{kernel}} $\tilde{K}$ is
defined by 
\[
  \tilde{K} (\tmmathbf{\xi}) = (k^2 - \xi_1^2 - \xi_2^2)^{- 1 / 2} .
\]

Upon denoting $\tmmathbf{k}= (k_1, k_2)$, the incident plane wave takes the
form $u^{\text{in}} (\tmmathbf{x}, x_3) = e^{i (\tmmathbf{k} \cdot
\tmmathbf{x}- x_3 / \tilde{K} (\tmmathbf{k}))}$. The unknown functions $\tilde U(\tmmathbf{\xi})$ and $\tilde W(\tmmathbf{\xi})$
are initially defined for real $\tmmathbf{\xi}$, but they can be analytically continued 
into much wider domains. This analytical continuation was the main subject of \cite{Assier2018a}.
Thus, below, $\tmmathbf{\xi}$ is considered as a pair of {\em complex\/} variables.

Such functions of two complex variables can have singularities of polar and branching types. 
Such singularities are not isolated points like for one complex variable functions, but are located on \textit{analytic sets}, which are surfaces of real dimension~2 embedded in the space $(\xi_1, \xi_2)\in\mathbb{C}^2$ having real dimension~4. For convenience, in \cite{Assier2018a}, such sets have been called \textit{polar 2-lines} and \textit{branch 2-lines}. We will continue to use the same terminology here. 

Let us introduce the domains
$\widehat{H}^{+}$ and $\widehat{H}^{-}$ as the upper and lower half-planes of the complex plane. Consider also the domains $H^{\pm}$ that are the upper and lower half-planes cut along the cuts $h^{\pm}$ as illustrated in figure \ref{fig:geomandHplus}.
The cuts $h^{\pm}$ are the images of the real axis under the mappings 
$\xi \to \pm \sqrt{k^2 - \xi^2}$.

\begin{figure}[h]
  \centering{\includegraphics{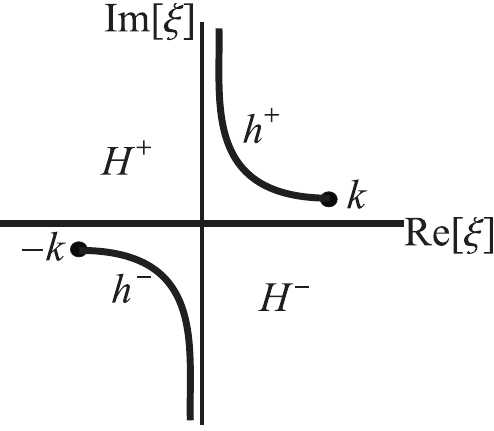}}
  \caption{The sets $H^{\pm}$ and the cuts $h^{\pm}$ within the $\xi$ complex plane}
\label{fig:geomandHplus}
\end{figure}

In {\cite{Assier2018a}} we introduced the important notion 
of {\tmem{additive crossing}}. The simplest possible definition of 
additive crossing is the following. 
Let $D_1$ and $D_2$ be some domains in the $\xi_1$ and $\xi_2$ complex planes respectively (see figure~\ref{fig:addcross}). 
Assume that these domains are cut along the cuts $\chi_{1,2}$ starting at the points 
$d_{1,2}$ and denote the shores of the cuts by the symbols $r$ (right) and 
$\ell$ (left).

\begin{figure}[h]
\centering{\includegraphics{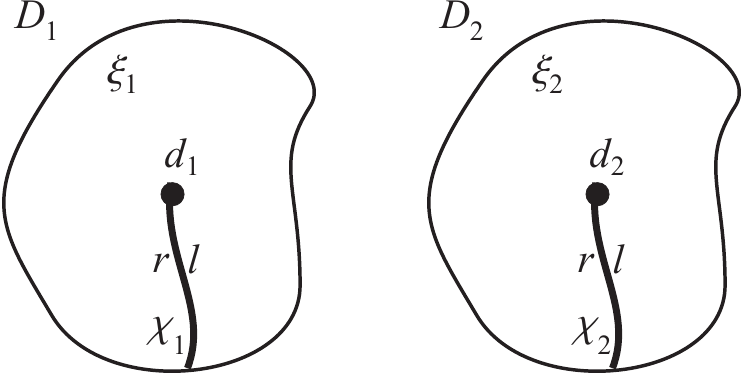}}
   \caption{Geometrical illustration of the sets used in defining additive crossing}
\label{fig:addcross}
\end{figure}

Consider $\tilde{\Phi}(\xi_1 , \xi_2)$ to be a function holomorphic in the domain 
$(D_1 \setminus \chi_1) \times (D_2 \setminus \chi_2)$ and one-sided continuous on the shores of the cuts (here by cuts, we mean the sets $\chi_1 \times D_2$ and $D_1 \times \chi_2$). Let $\xi_1 \in \chi_1$ and $\xi_2 \in (D_2 \setminus \chi_2)$ and denote by $\tilde{\Phi}(\xi_1^{\ell} , \xi_2)$ and $\tilde{\Phi}(\xi_1^r , \xi_2)$ the values of $\tilde{\Phi}$ on different shores of $\chi_1$. If these values are not equal then 
$d_1 \times D_2 $ is a branch 2-line of $\tilde{\Phi}$. Similarly, we can define $\tilde{\Phi}(\xi_1 , \xi_2^{\ell})$ and $\tilde{\Phi}(\xi_1 , \xi_2^r)$ for $\xi_1\in(D_1 \setminus \chi_1)$ and $\xi_2\in\chi_2$, and if these two quantities are not equal, then $D_1 \times d_2$ is also a branch 2-line of~$\tilde{\Phi}$. By continuity, it is hence possible to define the quantities $\tilde{\Phi} (\xi_1^{r,\ell}, \xi_2^{r,\ell})$.

\begin{definition}
We say that such a function $\tilde{\Phi}$ has the additive crossing property about the branch 2-lines  $d_1 \times D_2$
and $D_1 \times d_2$ if
  \begin{align*}
    \tilde{\Phi} (\xi_1^{\ell}, \xi_2^{\ell}) + \tilde{\Phi} (\xi_1^r, \xi_2^r) & =  \tilde{\Phi}
    (\xi_1^{\ell}, \xi_2^r) + \tilde{\Phi} (\xi_1^r, \xi_2^{\ell}).
  \end{align*}  
\end{definition}

Note that the definition introduced above does actually not require the concept of a branch 2-line. In fact, it does admit some generalisations, but this is not needed for the present work. Using this concept of additive crossing, one can formulate the main theorem proven in \cite{Assier2018a}:

\begin{theorem}
  \label{th:part1} Let $k_{1,2}$ be such that 
  $\tmop{Re}[k_{1, 2}] > 0$ and $\tmop{Im}[k_{1, 2}] > 0$. 
  For any function $\tilde{W} (\tmmathbf{\xi})$, consider
  the two associated functions $\tilde{U} (\tmmathbf{\xi})$ and $\tilde{U}'
  (\tmmathbf{\xi})$ defined by
  \begin{eqnarray*}
    \tilde{U} (\tmmathbf{\xi}) = - i \tilde{K} (\tmmathbf{\xi}) \tilde{W}
    (\tmmathbf{\xi}) & \tmop{and} & \tilde{U}' (\tmmathbf{\xi}) = \tilde{U}
    (\tmmathbf{\xi}) - (\xi_1 + k_1)^{- 1} (\xi_2 + k_2)^{- 1}.
  \end{eqnarray*}
  If the function $\tilde{W} (\tmmathbf{\xi})$ 
  and the associated function $\tilde{U}' (\tmmathbf{\xi})$  
  have the following properties:
  \begin{enumerate}[label=\textup{FP\arabic*}]
    \item \label{itm:1th2}$\tilde{W}$ is holomorphic in the domain $(\hat{H}^+
    \times (\hat{H}^+ \cup H^- \backslash \{ - k_2 \})) \cup ((\hat{H}^+ \cup
    H^- \backslash \{ - k_1 \}) \times \hat{H}^+)$
    
    \item \label{itm:2th2}$\tilde{W} (\tmmathbf{\xi})$ has poles (2-lines) 
    at $\xi_1 = - k_1$ and $\xi_2 = - k_2$ with known residues
    
    \item \label{itm:3th2} The associated function $\tilde{U}'
    (\tmmathbf{\xi})$ is holomorphic in the domain $(H^- \backslash \{ - k_1 \})
    \times (H^- \backslash \{ - k_2 \})$

    \item \label{itm:4th2}$\tilde{U}' (\tmmathbf{\xi})$ has the additive crossing property for
    the 2-lines $\xi_1 = - k$ and $\xi_2 = - k$ with cuts $h^-$
    (it means that $D_{1,2} = H^- \setminus \{ -k_{1,2} \}$,
    $d_{1,2}= - k$ and $\chi_{1,2} =h^-$)

    \item \label{itm:5th2}There exist some functions $E_1 (\xi_1)$ and $E_2
    (\xi_2)$, defined for complex $\xi_1$ and $\xi_2$, such that
    \begin{eqnarray*}
      | \tilde{W} (\xi_1, \xi_2) | < E_1 (\xi_1) | \xi_2 |^{- 1 / 2} &
      \tmop{as} & | \xi_2 | \rightarrow \infty, \tmop{Im} [\xi_2] > 0\\
      | \tilde{W} (\xi_1, \xi_2) | < E_2 (\xi_2) | \xi_1 |^{- 1 / 2} &
      \tmop{as} & | \xi_1 | \rightarrow \infty, \tmop{Im} [\xi_1] > 0
    \end{eqnarray*}
        
    \item \label{itm:6th2}There exists a function $C (\beta, \psi_1, \psi_2)$,
    defined for $0 < \beta < \pi / 2$ and $0 < \psi_{1, 2} < \pi$ such that
    for real $\Lambda$
    \begin{eqnarray*}
      | \tilde{W} (\xi_1, \xi_2) | < C (\beta, \psi_1, \psi_2) \Lambda^{- 1 -
      \mu} & \tmop{for} \tmop{some} & \mu > - 1 / 2,
    \end{eqnarray*}
    where $\xi_{1, 2}$ are parametrised as follows for large real $\Lambda$
    \begin{eqnarray*}
      \xi_1 = \Lambda e^{i \psi_1} \cos (\beta) & \tmop{and} & \xi_2 = \Lambda
      e^{i \psi_2} \sin (\beta)
    \end{eqnarray*}
    
  \end{enumerate}
  Then the field $u (\tmmathbf{x}, x_3)$ defined by $u (\tmmathbf{x}, x_3) = -
  i\mathfrak{F}^{- 1} [\tilde{K} \tilde{W} e^{i | x_3 | / \tilde{K}}]
  (\tmmathbf{x})$ is the sought-after solution to the quarter-plane problem.
\end{theorem}

From point \ref{itm:1th2} it follows that the function $\tilde W$ is 1/4-based (its inverse Fourier transform is non-zero only for $x_1 >0$ and $x_2 >0$). Point \ref{itm:2th2} is responsible for the 
incident plane wave. Points \ref{itm:3th2} and~\ref{itm:4th2} are nontrivial. From them it follows that 
$\tilde U'$ is 3/4-based, i.e.\ its inverse Fourier transform is equal to zero 
for $x_1 >0$ and $x_2 > 0$. Point \ref{itm:5th2} is related to the edge conditions, while 
point \ref{itm:6th2} is responsible for the vertex condition. The radiation condition should be fulfilled by construction.  

The conditions of  Theorem~\ref{th:part1} form what we will refer to as the \textit{functional problem} (\textbf{FP})
for~$\tilde W (\tmmathbf{\xi})$. By this we mean that the functional problem for $\tilde W$
is the following:
\begin{align}
\textbf{FP} &: \text{Find a function $\tilde W(\tmmathbf{\xi})$ 
obeying the points \ref{itm:1th2}--\ref{itm:6th2}}
\label{eq:FPproblem}
\end{align}

%%%%%%%%%%%%%%%%%%%%%%%%%%%%%%%%%%%%%%%%%%%%%%%%%%%%%%%%%%%%%%%%%%
\subsection{A simplified functional problem}
\label{sec:simplifiedformulation}

The purpose of the present work is to illustrate the significance and practical implications of the additive crossing property. Hence, for simplicity, let us now consider a modified version of the spectral formulation (\textbf{FP}) by making the following simplifications. We will disregard the polar singularities due to $k_{1, 2}$ corresponding to the
incident wave (affecting points \ref{itm:1th2}, \ref{itm:2th2} and
\ref{itm:3th2}, making the conditions imposed on $\tilde{W}$ {\tmem{stronger}}. In addition, we will weaken the
vertex growth conditions (affecting point \ref{itm:6th2} of
Theorem \ref{th:part1}) imposed on $\tilde{W}$, allowing an arbitrary 
power growth (the parameter $\mu$ should now be considered as arbitrary). Finally, let us abandon the edge growth conditions (point \ref{itm:5th2} of
Theorem \ref{th:part1}). Surprisingly, we will see that with the weakened 
vertex condition, the edge conditions need to be formulated in a slightly 
different form. This is why we do not consider the edge conditions now, and will return to them later, ultimately using them for selecting the right solutions in Section \ref{sec:stencil2odetext}.  

This results in the following {\em simplified functional problem} (\textbf{SFP}):
\begin{align}
\textbf{SFP} &: \text{Find a function $\tilde W(\tmmathbf{\xi})$ 
obeying the points \ref{item:SFP1}--\ref{item:SFP4}},
\label{eq:SFPproblem}
\end{align}
where the four properties are
{\em
  \begin{enumerate}[label=\textup{SFP\arabic*}]
    \item \label{item:SFP1}$\tilde{W}$ is analytic in the domain $(\hat{H}^+ \times (\hat{H}^+
    \cup H^-)) \cup ((\hat{H}^+ \cup H^-) \times \hat{H}^+)$
    
    \item \label{item:SFP2}The function $\tilde{U} (\tmmathbf{\xi}) =-i \tilde{K}
    (\tmmathbf{\xi}) \tilde{W} (\tmmathbf{\xi})$ is analytic in the domain
    $H^- \times H^-$
    
    \item \label{item:SFP3}The function $\tilde{U} (\tmmathbf{\xi})$ has the additive crossing
    property for the 2-lines $\xi_1 = - k$ and $\xi_2 = - k$ with associated
    cuts $h^-$
    
   \item \label{item:SFP4}There exists a real parameter $\mu$
   such that for any $\xi_{1,0}$, $\xi_{2,0}$  
   \[
   |\tilde W (\Lambda \xi_{1,0} , \Lambda \xi_{2,0})| < \Lambda^{-1-\mu}
   \]
   for large enough $\Lambda>0$, where the points $\xi_{1,0}$ and $\xi_{2,0}$ are chosen such that $(\Lambda \xi_{1,0} , \Lambda \xi_{2,0})$ remains within the domain of analyticity of $\tilde W$.   %{\color{red} FOR WHICH $\xi_{1,2,0}$??}
     \end{enumerate}
}
The present paper is dedicated to the resolution of this simplified functional problem \textbf{SFP}. 
Below we show that solutions of this functional problem correspond 
to wave fields generated by some source configurations at the vertex of the quarter-plane. The connection between the solution of the simplified functional problem \textbf{FP} and the 
simplified functional problem \textbf{SFP} is not completely clear at this stage and is beyond the scope of this work. \RED{Note that due to the presence of branch 2-lines, the unknown functions $\tilde{W}$ and $\tilde{U}$ are effectively multivalued functions of two complex variables. For this work, we do not need to (and we do not aim to) construct an associated Riemann manifold over $\mathbb{C}^2$ (generalisation of Riemann surface in 1D complex analysis), on which these functions are analytic. We have considered such questions in a different physical context in \cite{analyticalcontifields}.}

%%%%%%%%%%%%%%%%%%%%%%%%%%%%%%%%%%%%%%%%%%%%%%%%%%%%%%%%%%%%%%%%%%
\subsection{On the function $\tilde W(\tmmathbf{\xi})$ and its associated wave field}

\label{sec:wavefield}

A solution $\tilde W(\tmmathbf{\xi})$ of 
the simplified functional problem \textbf{SFP} corresponds to a wave field 
$u(\tmmathbf{x},x_3)$
defined by 
\begin{eqnarray}
  u (\tmmathbf{x}, x_3) & = & - \frac{i}{4 \pi^2} 
  \int_{\Gamma_{\xi}}
  \int_{\Gamma_{\xi}}
    \tilde{K} (\xi_1, \xi_2) \tilde{W} (\xi_1, \xi_2) 
    e^{i x_3 \tilde{K}^{-1}   (\xi_1, \xi_2)} 
    e^{- i (\xi_1 x_1 + \xi_2 x_2)} \mathd \xi_1  \mathd
  \xi_2,  
  \label{eq:physicalfieldintegralxi}
\end{eqnarray}
for $x_3 > 0$, where $\Gamma_{\xi}$ is just the real segment $(-\infty,\infty)$. 
This representation is inherited from the definition (\ref{eq:SpecFun}).
The properties of the integral (\ref{eq:physicalfieldintegralxi}) should be 
considered carefully. 

The integrand has no singularities on the surface of integration, i.e.\ on the 
real plane, since $k$ has a small positive imaginary part. Formally, the convergence of the integral can be questionable, since $\tilde W$
can grow at infinity as an arbitrary power of $|\tmmathbf{\xi}|$. However, note that if $x_3 > 0$ then the integral converges exponentially since $\text{Im}[\tilde{K}^{-1}]>0$. When $x_3=0$, in order for it to remain exponentially convergent, it is necessary to regularise the integral by deforming the contour ever so slightly. This point is addressed in Appendix \ref{app:SPM}. 

The exponential convergence of the regularised integral (\ref{eq:physicalfieldintegralxi}) enables one to differentiate it with respect to the variables $x_{1,2,3}$, which play the role of parameters. One can easily show that $u(x_1, x_2, x_3)$ obeys the Helmholtz equation (\ref{Helmholtz})
for $x_3 > 0$. This is supported by the elementary observation that (\ref{eq:physicalfieldintegralxi})
has the structure of a plane wave decomposition. 

Using the methods presented in \cite{Assier2018a}, one can prove that   
$u(x_1, x_2, 0) = 0$ for  $x_1 > 0$ and $x_2>0$ . This follows from
point \ref{item:SFP1} of the simplified functional problem (implying the 1/4-basedness of $\tilde{W}$). Besides, using point \ref{item:SFP2}, one can prove that
$\tfrac{\ptl u}{\ptl x_3} (x_1, x_2 , 0)  = 0$ for $x_1 < 0$ or $x_2 < 0$. 

By applying the multidimensional saddle-point method \cite{Bleistein2012} (which is not elementary in this case)
one can prove that 
$u(\tmmathbf{x},x_3)$ obeys the radiation condition for $x_3 > 0$. Intuitively this is clear, 
since the plane wave decomposition (\ref{eq:physicalfieldintegralxi})
contains only waves that are outgoing and decaying for $x_3 \to \infty$. Also it is possible to show that in the area $x_3 = 0$, $\sqrt{x_1^2 + x_2^2} \to \infty$ the field $u$ contains only outgoing waves. 

Since we ultimately wish to let ${\rm Im}[k] \to 0$, one should be careful to \textit{indent} the contour of integration of (\ref{eq:physicalfieldintegralxi}) properly in order to bypass the singularity defined by the equation $\xi_1^2 + \xi_2^2 = k^2$. Since the contour is a two-dimensional surface embedded in a four-dimensional space, this is non-trivial. In order to facilitate this task, we will make use of the \textit{bridge and arrow} notation, to which the Appendix \ref{app:appA} is dedicated. 

A consequence of the application of the multidimensional saddle-point method is that the values of $\tilde W(\xi_1, \xi_2)$ 
for the circle $\xi_1^2 + \xi_2^2 = k^2$ play an important role: the directivity pattern of the 
field (also known as diffraction coefficient) is proportional to the function $\tilde W$ taken at these values (see \cite{surprise}, where a similar result is used).

We shall see later that the function $\tilde W$ is multivalued. The values corresponding
to the directivity (i.e.\ belonging to the surface of integration within this circle)
will be referred to as belonging to the {\em physical sheet} of the Riemann manifold of $\tilde W$.  

Finally, according to general properties of the Fourier transform, the condition \ref{item:SFP4} guarantees that the field has a singularity at the origin no stronger than of power type. 

We can hence conclude that, if $\tilde W$ 
satisfies the simplified functional problem \textbf{SFP} formulated above, 
it corresponds to a mixed homogeneous boundary-value problem on a quarter-plane. The solution has a weakened vertex condition (comparatively to the classical diffraction problem) and abandoned edge conditions (we plan to impose the edge conditions later). Such solution contains only outgoing wave components. 
  
%%%%%%%%%%%%%%%%%%%%%%%%%%%%%%%%%%%%%%%%%%%%%%%%%%%%%%%%%%%%%%%%%%
%%%%%%%%%%%%%%%%%%%%%%%%%%%%%%%%%%%%%%%%%%%%%%%%%%%%%%%%%%%%%%%%%%

\section{Solution of the simplified functional problem} \label{sec:solfuncprob}

\subsection{The simplified functional problem in the angular coordinates}
\label{sec:formulationangularalpha}

Let us introduce the so-called angular coordinates $\alpha_{1, 2}$, linked to
the Fourier coordinates $\xi_{1, 2}$ by
\begin{equation}
  \alpha_{1, 2} = \arcsin (\xi_{1, 2} / k) 
  \quad  \tmop{ and } \quad \xi_{1, 2} = k \sin
  (\alpha_{1, 2})  .
  \label{eq:mappingccord}
\end{equation}
Here of course, $\alpha_{1, 2}$ are understood to be complex. Let us list the properties of the mapping $\xi \to \alpha$. The points 
$\pm k$ located on the physical sheet are mapped to the points $\pm \pi /2$. 
%Here the term {\em on the physical sheet\/} means that these points belong to the 
%integration surface of (\ref{eq:physicalfieldintegralxi}) in the limit 
%${\rm Im}(k) \to 0$. 
Indeed, the points with affixes (coordinates) $\xi = \pm k$
located on other sheets of the Riemann surface of the analytical continuation of $\tilde W$ are mapped to $\pm \pi/2 + \pi n$. 

The shores of the cut $h^+$ are mapped to the curve $m^+$ shown in 
figure~\ref{fig:fromxitoalphacontour}. The left shore of the cut 
is mapped to the part of $m^+$ with ${\rm Im}[\alpha] >0$, while the right shore is mapped 
to the part of $m^+$ with ${\rm Im}[\alpha] < 0$.
Similarly, the shores of $h^-$ are mapped to the curve $m^-$. This mapping from $m^\pm$ to $h^\pm$ is illustrated on figure~\ref{fig:fromxitoalphacontour} and summarised in table~\ref{table:mapping-summary}.

\begin{table}[h]
  \centering\begin{tabular}{l}
    \begin{tabular}{|c|c|c|c|c|}
      \hline
      set in $\alpha$ plane & lower part of $m^-$ & upper part of $m^-$ &
      lower part of $m^+$ & upper part of $m^+$\\
      \hline
      set in $\xi$ plane & left shore of $h^-$ & right shore of $h^-$ & right
      shore of $h^+$ & left shore of $h^+$\\
      \hline
    \end{tabular}
  \end{tabular}
  \caption{Summary of which part of $h^\pm$ are mapped to which part of $m^\pm$ by (\ref{eq:mappingccord})}
\label{table:mapping-summary}
\end{table}

One should also note that the sets $m^{\pm}$ have the inherent symmetry property that the upper part of $m^{\pm}$ (such that $\text{Im}[\alpha] \geq 0$) and the lower part of $m^\pm$ (such that $\text{Im}[\alpha] \leq 0$) are the image of each other via the mapping $\alpha\rightarrow\pm\pi-\alpha$, as summarised in table \ref{table:symalphaplane}.

\begin{table}[h]
  \centering\begin{tabular}{l}
    \begin{tabular}{|c|c|c|}
      \hline
      set in $\alpha$ plane & $m^-$ & $m^+$\\
      \hline
      symmetry mapping linking upper and lower part & $- \pi - \alpha$
      & $\pi - \alpha$\\
      \hline
    \end{tabular}
  \end{tabular}
  \caption{Inherent symmetry of the sets $m^{\pm}$}
\label{table:symalphaplane}
\end{table}

As is also shown in figure~\ref{fig:fromxitoalphacontour}, we introduce the contour $\Gamma_\alpha$ as the image of the contour $\Gamma_\xi$ by the mapping (\ref{eq:mappingccord}). It is clear that we have $\Gamma_{\alpha}=m^-+\tfrac{\pi}{2}=m^+-\tfrac{\pi}{2}$.
%a set of points in the $\alpha$-plane corresponding to the values $\arcsin(r/k)$, where $r$ runs over $\mathbb{R}$. 
Both $\Gamma_{\xi}$ and $\Gamma_{\alpha}$ are shown by a red line in the figure. We can also introduce the curved strips $M^\pm $ as it is 
shown in the figure. Following from the properties of conformal mappings, 
it is clear that the domains $H^{\pm}$ are mapped to $M^{\pm}$ by (\ref{eq:mappingccord}). 

In what follows we will use notations like $\pi - M^+$. This 
particular notation correspond to 
the set of values of all $\alpha = \pi - \alpha'$ where $\alpha'\in M^+$.   

\begin{figure}[h]
\centering\includegraphics[width=0.48\textwidth]{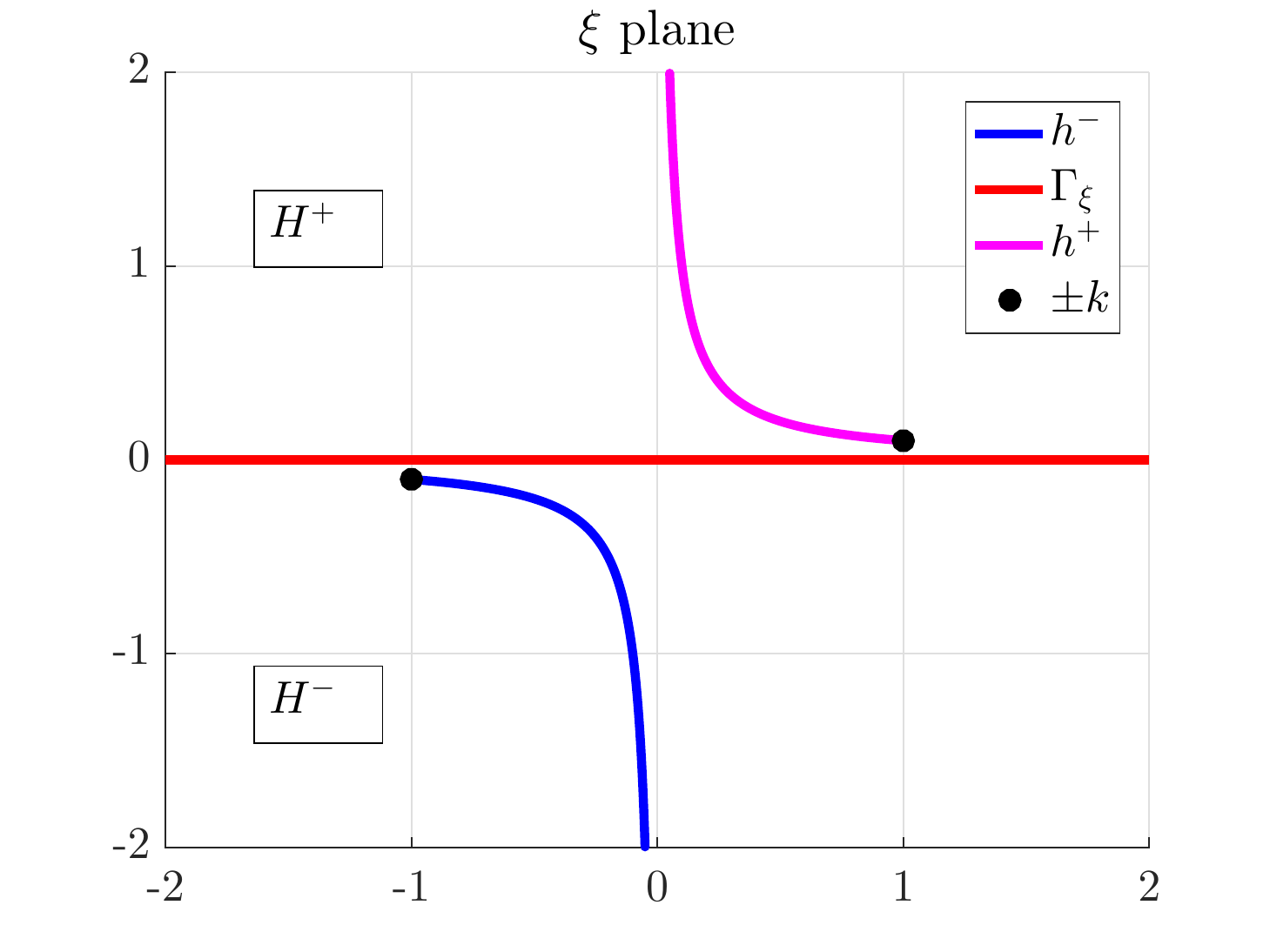}
\includegraphics[width=0.48\textwidth]{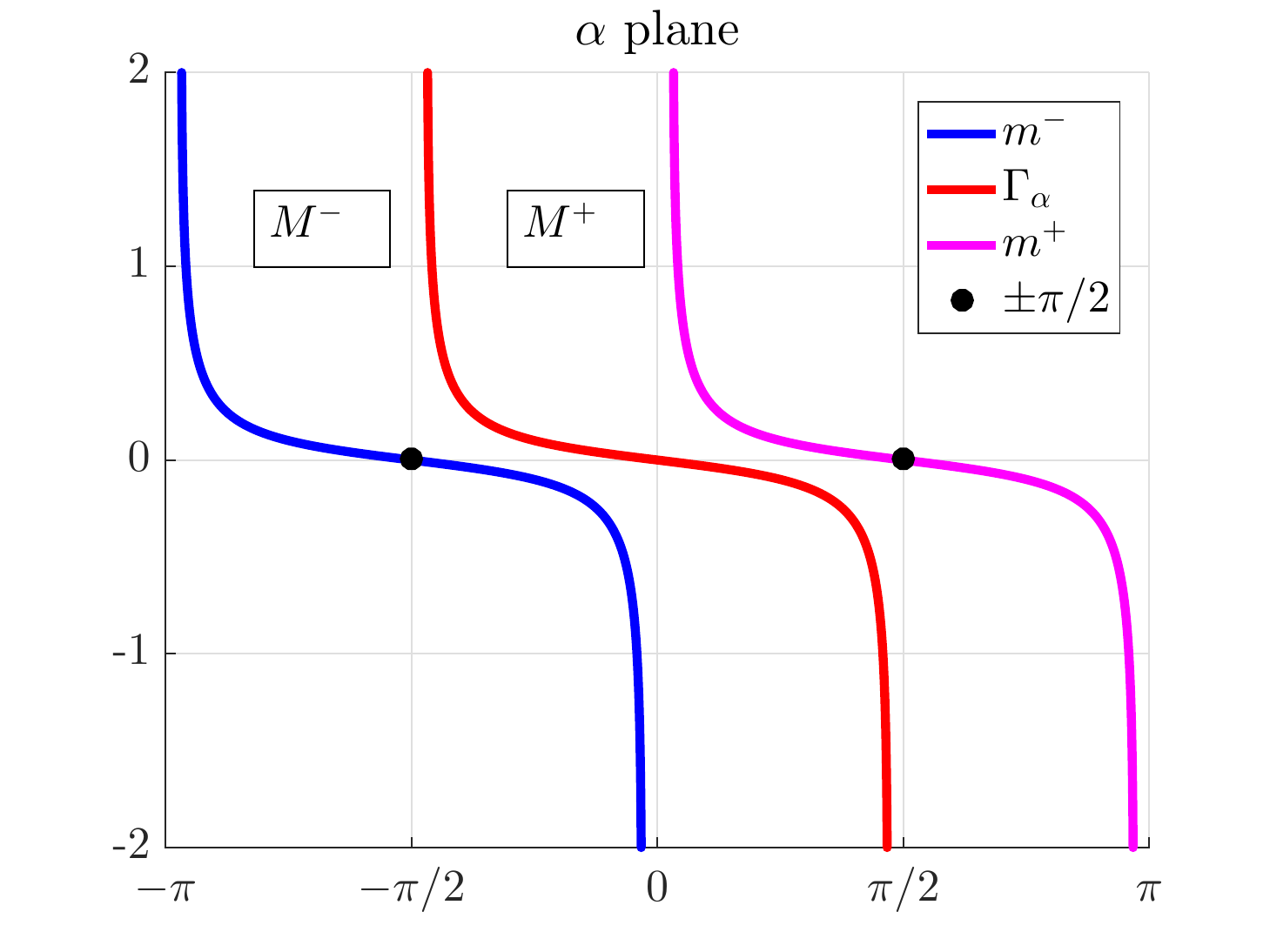}
  \caption{Transformation from the $\xi$ plane to the $\alpha$ plane}
\label{fig:fromxitoalphacontour}
\end{figure}

For any function $\tilde{\Phi} (\xi_1, \xi_2)$, we can define a new
function $\hat{\Phi} (\alpha_1, \alpha_2)$ by
\begin{eqnarray*}
  \hat{\Phi} (\alpha_1, \alpha_2) & = & \tilde{\Phi} (k \sin (\alpha_1), k
  \sin (\alpha_2)).
\end{eqnarray*}
Using this rule, we define the functions $\hat{W}$, $\hat{K}$ and $\hat{U}$ 
of $\boldsymbol{\alpha}=(\alpha_1, \alpha_2)$.
We can reformulate the simplified functional problem \textbf{SFP} using the angular coordinates. This leads to a new problem referred to as the \emph{angular simplified formulation} (\textbf{ASF}), defined as follows:
\begin{align}
\textbf{ASF} &: \text{Find a function $\hat{W}(\boldsymbol{\alpha})$ obeying the points \ref{itm:1prop4}--\ref{itm:4prop4}}
\end{align}
The following proposition lists the conditions of this new problem.

\begin{proposition}
  \label{prop:angularmodel}Let $\tilde{W} (\tmmathbf{\xi})$ be a function that
  solves the simplified functional problem \textup{\textbf{SFP}}. 
  Then its associated function $\hat{W}
  (\tmmathbf{\alpha})$ has the following properties (they compose the angular 
  simplified formulation \textup{\textbf{ASF}}):
  \begin{enumerate}[label=\textup{ASF\arabic*}]
    \item \label{itm:1prop4}$\hat{W} (\tmmathbf{\alpha})$ is analytic in the
    domain $(M^+ \times (M^+ \cup M^-)) \cup ((M^+ \cup M^-) \times M^+)$
    
    \item \label{itm:1primeprop4}We have $\hat{W} (\alpha_1, \alpha_2) =
    \hat{W} (\pi - \alpha_1, \alpha_2)$ and $\hat{W} (\alpha_1, \alpha_2) =
    \hat{W} (\alpha_1, \pi - \alpha_2)$ on $m^+ \times m^+$.
    
    \item \label{itm:2prop4} The function $\hat{U} (\tmmathbf{\alpha}) = -i
    \hat{K} (\tmmathbf{\alpha}) \hat{W} (\tmmathbf{\alpha})$ is analytic on
    $M^- \times M^-$
    
    \item \label{itm:3prop4} The function $\hat{U} (\tmmathbf{\alpha})$
    satisfies the following equation on $m^- \times m^-$
    \begin{eqnarray}
      \hat{U} (\alpha_1, \alpha_2) + \hat{U} (- \pi - \alpha_1, - \pi -
      \alpha_2) & = & \hat{U} (\alpha_1, - \pi - \alpha_2) + \hat{U} (- \pi -
      \alpha_1, \alpha_2)  \label{eq:prop4pt3}
    \end{eqnarray}
    
  \item \label{itm:4prop4}
    There exists a constant $\mu$ such that
    for real $\alpha_{1,0}$, $\alpha_{2,0}$, $\alpha_{1,0}'$, $\alpha_{2,0}'$  
    \[	
	\hat W (\alpha_{1,0} + \Lambda \alpha'_{1,0},\alpha_{2,0} + \Lambda \alpha'_{2,0})
	< \exp \{ \mu \Lambda \}
	\]
    for large enough $\Lambda$. Parameters  
    $\alpha_{1,0}$, $\alpha_{2,0}$, $\alpha_{1,0}'$, $\alpha_{2,0}'$  
    are chosen such that the points $(\alpha_{1,0} + \Lambda \alpha'_{1,0},\alpha_{2,0} +     \Lambda \alpha'_{2,0})$
    fall into the domains of analyticity defined in points \ref{itm:1prop4} and \ref{itm:2prop4}, and 
    $(\alpha_{1,0}')^2 + (\alpha_{2,0}')^2 = 1$. 
  
  \end{enumerate}

  Reciprocally, if $\hat{W} (\tmmathbf{\alpha})$ satisfies these conditions,
  then the associated function $\tilde{W} (\tmmathbf{\xi})$ satisfies the
  simplified functional formulation \textup{\textbf{SFP}}.
\end{proposition}

\begin{proof}
  The points \ref{itm:1prop4}, \ref{itm:2prop4} and \ref{itm:4prop4}
  are straightforwardly equivalent to the points \ref{item:SFP1}, \ref{item:SFP2} and \ref{item:SFP4} respectively,  since the sets $H^{\pm}$ and the cuts $h^{\pm}$ are sent
  to $M^{\pm}$ and $m^{\pm}$ via the change of variables (\ref{eq:mappingccord}). The point
  \ref{itm:3prop4} comes from the additive property \ref{item:SFP3} satisfied by $\tilde{U}
  (\tmmathbf{\xi})$ and from the mapping and symmetry results summarised in tables \ref{table:mapping-summary}--\ref{table:symalphaplane}.
Indeed, let us consider $\xi^{\ell}$ to belong to the left shore of $h^-$, and let $\alpha$ be its image by the mapping (\ref{eq:mappingccord}) then it is clear that $\alpha$
  belongs to the lower part of $m^-$. Now the symmetry of $m^-$ implies that
  $- \alpha - \pi$ belongs to the upper part of $m^-$ and is hence mapped back by (\ref{eq:mappingccord}) to the point $\xi^r$ belonging to the right shore of $h^-$. Hence $\xi^{\ell}_{1,2} \leftrightarrow \alpha_{1,2}$ and $\xi^r_{1,2}
  \leftrightarrow - \pi - \alpha_{1,2}$ and the additive crossing property \ref{item:SFP3} of
  $\tilde{U} (\tmmathbf{\xi})$ is equivalent to the point \ref{itm:3prop4} for $\hat{U}
  (\tmmathbf{\alpha})$. 

The point \ref{itm:1primeprop4} can be proved
  similarly from the analyticity property \ref{item:SFP1} of $\tilde{W}$. Indeed, let us consider $(\xi_1,\xi_2)\in h^+\times \hat{H}^+$ and refer to $\xi_1^{\ell}$ and $\xi_1^{r}$ as being on the left and right shores of $h^+$ respectively. The function $\tilde{W}$ satisfies $\tilde{W} (\xi_1^{\ell}, \xi_2) = \tilde{W} (\xi_1^r, \xi_2)$, since it
  is analytic on $h^+$ (as a function of $\xi_1$). Let $\alpha_1\in m^+$ and $\alpha_2\in M^+$ be the image of $\xi_1^\ell$ and $\xi_2$ by the mapping (\ref{eq:mappingccord}), it is clear from tables \ref{table:mapping-summary}--\ref{table:symalphaplane} that $\alpha_1$ belongs to the upper part of $m^+$, and that $\pi-\alpha_1$ is mapped back to $\xi_1^r$. We hence have that $\hat{W} (\alpha_1, \alpha_2) = \hat{W} (\pi
  - \alpha_1, \alpha_2)$ on $m^+ \times M^+$. Very Similarly we obtain that $\hat{W} (\alpha_1, \alpha_2) = \hat{W} (\alpha_1, \pi-\alpha_2)$ on $M^+ \times m^+$. And hence by continuity, both equalities are simultaneously valid on $m^+\times m^+$, as required.
%{\link{\qedsymbol}}
\end{proof}

As discussed in {\cite{Assier2018a}}, and as could be anticipated from the
definition of $\tilde{K} (\tmmathbf{\xi})$ for which it is a singular set, the
set
\begin{eqnarray*}
  \gamma^{\star} & = & \left\{ \tmmathbf{\xi} \in \mathbbm{C}^2 \text{ such
  that } \xi_1^2 + \xi_2^2 = k^2 \right\},
\end{eqnarray*}
that we will refer to as the {\tmem{complexified circle}} is very important to
us. Its image by the mapping (\ref{eq:mappingccord}), and hence the singular
set of $\hat{K} (\tmmathbf{\alpha})$, is the set of 2-lines
 $\sigma_n^{\pm}$ defined for any $n \in \mathbbm{Z}$ by
\begin{eqnarray}
  \sigma_n^{\pm} & = & \left\{ \tmmathbf{\alpha} \in \mathbbm{C}^2 \text{ such
  that } \alpha_1 \pm \alpha_2 = \pi / 2 + n \pi \right\}, 
\end{eqnarray}
corresponding to the $\alpha_{1, 2}$ such that $\cos^2 (\alpha_1) = \sin^2
(\alpha_2)$. 
%Note that these complexified lines are in fact analytic manifolds
%of dimension 2. 
In order to help with the visualisation of such sets, we will
often look at their {\tmem{real trace}}. The real trace of a singular set
$\gamma$ is a set of real curves describing $\gamma \cap \mathbbm{R}^2$. In
Figure~\ref{fig:realtraces}, we show the real traces of $\gamma^{\star}$
(left) and of some of the $\sigma_n^{\pm}$ (right) for real $k$, i.e.\
this is the limiting case ${\rm Im}[k] \to 0$. 

\begin{figure}[h]
  \centering\includegraphics[width=0.4\textwidth]{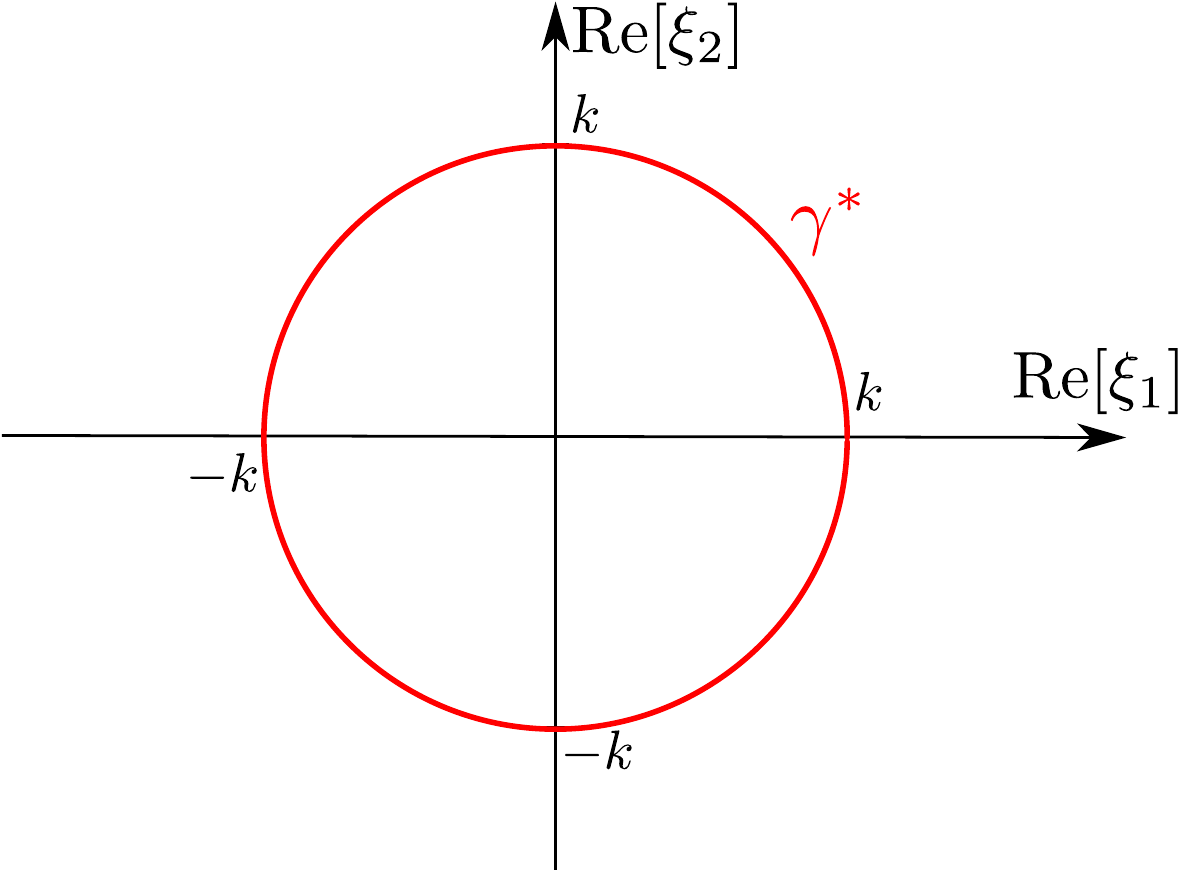}\quad\includegraphics[width=0.4\textwidth]{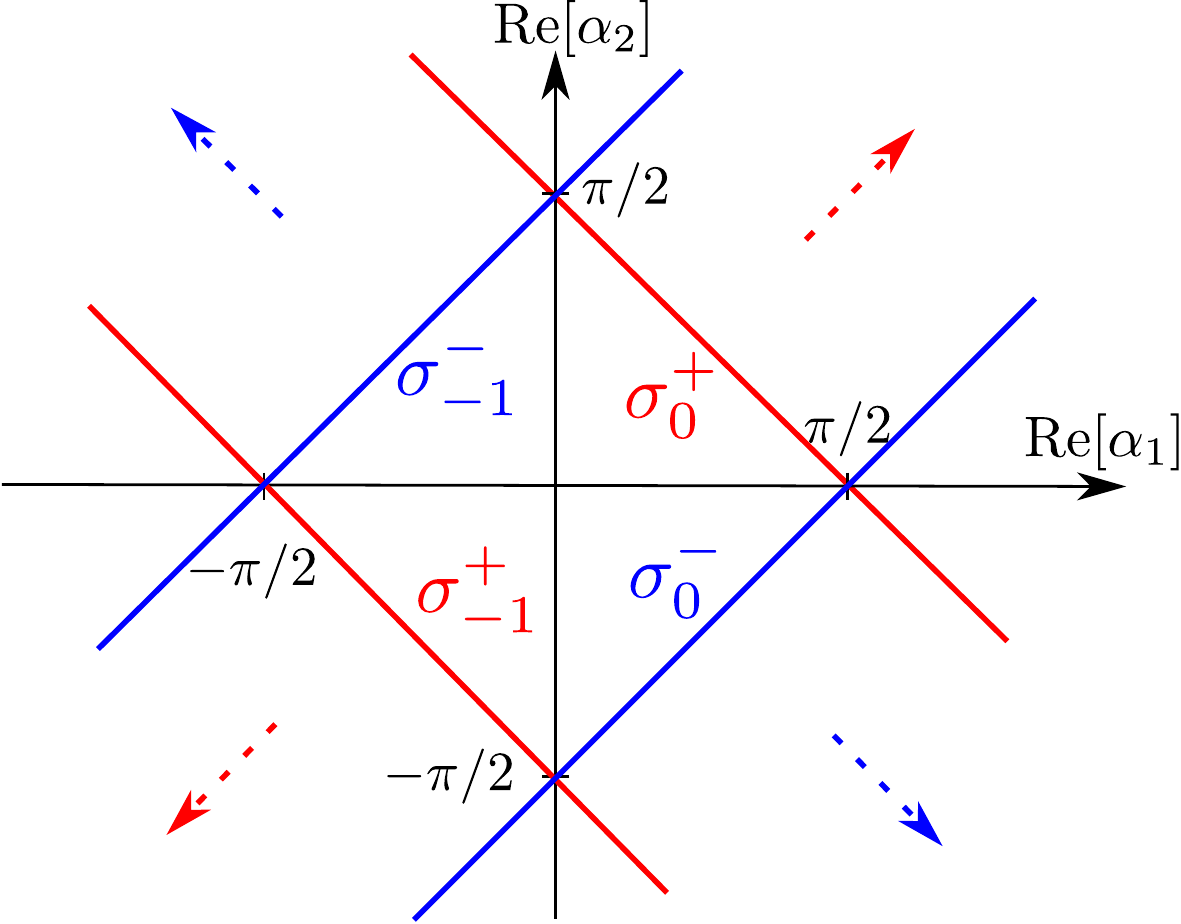}
  \caption{Real traces of $\gamma^{\star}$ (left) and $\sigma_n^{\pm}$ (right)
  }
\label{fig:realtraces}
\end{figure}

Note that the region $\xi_1^2 + \xi_2^2 < k^2$ encircled by the real trace of $\gamma^*$, corresponding to the directivity of 
the field $u$, is mapped onto the square $|\alpha_1| + |\alpha_2| < \pi/2$ shown in figure~\ref{fig:realtraces}, right. 

The point \ref{itm:1primeprop4} of Proposition \ref{prop:angularmodel} is
important since it will allow us to analytically continue the function
$\hat{W} (\tmmathbf{\alpha})$ to a wider domain, as summarised in Proposition
\ref{prop:sqaureroots} below.

\begin{proposition}
  \label{prop:sqaureroots}If $\hat{W} (\tmmathbf{\alpha})$ satisfies the angular simplified
  formulation \textup{\textbf{ASF}} of Proposition \ref{prop:angularmodel}, it can be analytically
  continued to the domain $\Omega \cup \Omega_{+ -} \cup \Omega_{- +} \cup
  \Omega_{- -}$, where
  \begin{eqnarray*}
    \Omega = (M^+ \cup M^-) \times (M^+ \cup M^-) & \tmop{and} & \Omega_{+ -}
    = (M^+ \cup M^-) \times (\pi - (M^+ \cup M^-))\\
    \Omega_{- +} = (\pi - (M^+ \cup M^-)) \times (M^+ \cup M^-) & \tmop{and} &
    \Omega_{- -} = (\pi - (M^+ \cup M^-)) \times (\pi - (M^+ \cup M^-))
  \end{eqnarray*}
  Note that $\hat{W} (\tmmathbf{\alpha})$ will have branch lines within this
  domain at $\sigma_{- 1}^+$, $\sigma_2^+$, $\sigma_1^-$ and $\sigma_{- 2}^-$,
  where it can be represented as a regular function multiplied by a square
  root. In particular, we can show that the function
  \begin{eqnarray*}
    & [(\alpha_1 + \alpha_2 + \pi / 2) (\alpha_1 + \alpha_2 - 5 \pi / 2)
    (\alpha_1 - \alpha_2 - 3 \pi / 2) (\alpha_1 - \alpha_2 + 3 \pi / 2)]^{- 1
    / 2} \hat{W} (\alpha_1, \alpha_2) & 
  \end{eqnarray*}
  is analytic on $\Omega \cup \Omega_{+ -} \cup \Omega_{- +} \cup \Omega_{-
  -}$.
\end{proposition}

\begin{proof}
  Points~\ref{itm:1prop4} and~\ref{itm:2prop4} of 
  Proposition~\ref{prop:angularmodel} provide some information on the behaviour of 
  $\hat W$ in the domain $\Omega$. Namely, \ref{itm:1prop4} states that 
  $\hat W$ is holomorphic in 
  $(M^+ \times (M^+ \cup M^-)) \cup ((M^+ \cup M^-) \times M^+ )$, 
  while \ref{itm:2prop4} states that $\hat W$ has a branch 2-line
  in $M^- \times M^-$, and this 2-line is $\sigma^+_{-1}$. The reasons for this being that on $M^-\times M^-$, we have $\hat{W}=i\hat{U}/\hat{K}$, $\hat{U}$ being analytic, and $\sigma^+_{-1}$ being the only part of the singularity set of $\hat{K}(\boldsymbol{\alpha})$ that belongs to $M^-\times M^-$.

Another way to see this is that according to \ref{itm:2prop4} and the definition of $\hat{K}$, the product 
  \[
  (\alpha_1 + \alpha_2 + \pi /2)^{-1/2} \hat W (\alpha_1, \alpha_2)
  \]
  is holomorphic in $M^- \times M^-$. Hence, allowing for this branch 2-line, we can analytically continue $\hat{W}$ onto $M^-\times M^-$. This means that overall we have an analytic continuation on 
\[
(M^+ \times (M^+ \cup M^-)) \cup ((M^+ \cup M^-) \times M^+ ) \cup (M^-\times M^-),
\]
which can easily be seen to be equal to $\Omega$.

The point~\ref{itm:1primeprop4} can now be used to continue $\hat W$ into the domain $\Omega_{-+}$. In order to do so just pick a point in $\Omega_{-+}$, and construct the analytic continuation of $\hat{W}$ to this point as follows. Start from within $\Omega$, at $(0,0)$ say, move to $m^+\times m^+$, and use the first formula to analytically continue $\hat W$ past the $m^+$ border of the $\alpha_1$ plane until reaching the sought-after point in $\pi-(M^+\cup M^-)$. Analytical continuation to $\Omega_{+-}$ can be done similarly by using the second formula of \ref{itm:1primeprop4}, while for the analytical continuation to $\Omega_{--}$, both formulae in \ref{itm:1primeprop4} should be used simultaneously. The corresponding reflections of the 
  branch 2-line $\sigma^+_{-1}$ are $\sigma^-_1$, $\sigma^-_{-2}$, and $\sigma^+_2$. They are also branch 2-lines of the analytical continuation of $\hat{W}$ and are of the same type as $\sigma^+_{-1}$.  
\end{proof}

%Before moving on further, it is important at this point to remember that the
%sought-after physical field is given by an integral by
%(\ref{eq:physicalfieldintegralxi}) in the $(\xi_1, \xi_2)$ complex planes.
%When expressed in the $(\alpha_1, \alpha_2)$ coordinates, this integral
%becomes
%\begin{eqnarray}
%  u (\tmmathbf{x}, x_3) & = & \frac{- ik^2}{4 \pi^2} \int_{\hat{\Gamma}}
%  \hat{K} (\alpha_1, \alpha_2) \hat{W} (\alpha_1, \alpha_2)
%  e^{\frac{ix_3}{\hat{K} (\alpha_1, \alpha_2)}} e^{- ik (\sin (\alpha_1) x_1 +
%  \sin (\alpha_2) x_2)} \cos (\alpha_1) \cos (\alpha_2) \mathd \alpha_1 \wedge
%  \mathd \alpha_2,  \label{eq:physicalfieldintegralalpha}
%\end{eqnarray}
%where the manifold of integration $\hat{\Gamma}$ is the image of
%$\tilde{\Gamma}$ by the change of variable (\ref{eq:mappingccord}). Since
%$\tilde{\Gamma}$ is {\tmem{almost equal}} to $\mathbbm{R}^2$, it is clear that
%a fragment of $\hat{\Gamma}$ is {\tmem{almost equal}} to the real square
%$\hat{\Gamma}' = \left[ \frac{- \pi}{2}, \frac{\pi}{2} \right] \times \left[
%\frac{- \pi}{2}, \frac{\pi}{2} \right]$.

%In the next section, we will provide another equivalent formulation and
%introduce the {\tmem{stencil equation}}.

%%%%%%%%%%%%%%%%%%%%%%%%%%%%%%%%%%%%%%%%%%%%%%%%%%%%%%%%%%%%%%%%%%

\subsection{The stencil equation}
\label{sec:stencilformulation}

%Using the angular formulation of Proposition \ref{prop:angularmodel} and the
%analytical continuation result of Proposition \ref{prop:sqaureroots}, we can
%provide an important reformulation of our problem.

The function $\hat K (\alpha_1 , \alpha_2)$ is branching in the domain of 
two complex variables $(\alpha_1, \alpha_2)\in\mathbb{C}^2$, moreover, its branch 2-lines 
$\sigma_n^{\pm}$ have one dimensional real traces on the real plane $(\alpha_1, \alpha_2)\in\mathbb{R}^2$ as can be seen in figure \ref{fig:realtraces}.
Let us fix the value of this function on the real plane. In order to do this, we will use the {\em bridge and arrow\/} notations defined and developed in Appendix~\ref{app:appA}.

The starting point of fixing the value of $\hat K$ is to find its value at 
$(\alpha_1 , \alpha_2) = (0,0)$, or, the same, the value
$\tilde K(0,0)$. As it is clear from the form of   
(\ref{eq:physicalfieldintegralxi}), one should choose 
\begin{equation}
\hat K(0,0) = k^{-1}
\label{eq:sheetfixing1}
\end{equation}
to get the plane waves outgoing for $x_3 \to \infty$. 

The integration surface in the field reconstruction formula 
(\ref{eq:physicalfieldintegralxi}) is the real plane $(\xi_1, \xi_2)\in\mathbb{R}^2$.
For ${\rm Im}[k] \to 0$ the real square $\boldsymbol{\xi}\in[-k,k]^2$ is mapped onto the 
square $\boldsymbol{\alpha}\in[-\tfrac{\pi}{2},\tfrac{\pi}{2}]^2$ by (\ref{eq:mappingccord}). This square in the real $(\alpha_1,\alpha_2)$ plane is denoted by $\boldsymbol{\Gamma'}$ and shown in grey in figures~\ref{fig:stencilalpha}~and~\ref{fig04}. Let us fix the way in which the integration surface bypasses the singularities. The proper bypasses are shown in figure~\ref{fig04}.

The bypass symbol introduced in this way enables one to define the value 
of $\hat K(-\tfrac{\pi}{2},-\tfrac{\pi}{2})$. Namely, this value is equal to $-i k^{-1}$.
Hence, remembering that $(-\tfrac{\pi}{2},-\tfrac{\pi}{2})\in m^- \times m^-$, the values of $\hat K$ can be found on the whole set 
$m^- \times m^-$ by continuity. Thus, the link $\hat U = \hat K \hat W$ becomes clarified on 
$m^- \times m^-$. 

Because of the inherent symmetry of $m^-$ given in table \ref{table:symalphaplane} and the definition of $\hat{K}$, the following identities are valid on $m^- \times m^-$:
\begin{equation}
\hat K(\alpha_1 , \alpha_2) = \hat K (-\pi - \alpha_1 , \alpha_2)
= \hat K (\alpha_1 , - \pi - \alpha_2) = \hat K (-\pi - \alpha_1 , -\pi - \alpha_2).
\label{eq:sheetfixing2}
\end{equation}
Thus, from point~\ref{itm:3prop4} of Proposition~\ref{prop:angularmodel} and the definition of $\hat{U}$,
it follows that, on $m^-\times m^-$, we have 
\begin{equation}
\hat{W} (\alpha_1, \alpha_2) + 
\hat{W} (- \pi - \alpha_1, - \pi - \alpha_2) 
 =  
\hat{W} (\alpha_1, - \pi - \alpha_2) + 
\hat{W} (- \pi -\alpha_1, \alpha_2).  
\label{eq:sheetfixing3}
\end{equation}
 
Now, remembering that \ref{itm:1primeprop4} was used in order to continue the function $\hat W$ analytically, 
thus it remains valid for the continued function $\hat W$. Namely, the condition 
\begin{equation}
\hat W (\alpha_1 , \alpha_2) = \hat W(\pi - \alpha_1, \alpha_2)
\label{eq:sheetfixing4}
\end{equation}
links the values of $\hat W$ on $m^- \times m^-$ with the values 
on $(2\pi + m^-) \times m^-$. Similarly, 
the condition  
\begin{equation}
\hat W (\alpha_1 , \alpha_2) = \hat W(\alpha_1, \pi - \alpha_2)
\label{eq:sheetfixing5}
\end{equation}
links the values  
on $m^- \times m^-$ with the values on  
$m^- \times (2\pi + m^-)$. 
Finally, using these two relations, one can obtain the relation
\begin{equation}
\hat W (\alpha_1 , \alpha_2) = \hat W(\pi -\alpha_1, \pi - \alpha_2)
\label{eq:sheetfixing6}
\end{equation}
linking $m^- \times m^-$ with $(2\pi +m^-) \times (2\pi + m^-)$. However, these three links should be clarified, since the function
$\hat W$ is branching. 

Note that the function $\hat W$ is single-valued on the sets 
$m^- \times m^-$,  
$(2\pi + m^-) \times m^-$,
$m^- \times (2\pi + m^-)$,  
$(2\pi +m^-) \times (2\pi + m^-)$, since these do not intersect the branch 2-lines. Thus, one should find four reference points linked by the aforementioned relations, and then continue the relations by continuity. 
These points can be easily found by taking into account that
the point $(\tfrac{\pi}{2} , \tfrac{\pi}{2})$ belongs to the physical sheet of $\hat W$. They are the points 
$(-\tfrac{\pi}{2}, -\tfrac{\pi}{2})$, $(- \tfrac{\pi}{2}, \tfrac{3\pi}{2})$, $(\tfrac{3\pi}{2}, -\tfrac{\pi}{2})$,
$(\tfrac{3\pi}{2} , \tfrac{3\pi}{2})$ shown in figure~\ref{fig:stencilalpha} and belonging respectively to $m^- \times m^-$, $(2\pi + m^-) \times m^-$, $m^- \times (2\pi + m^-)$ and $(2\pi +m^-) \times (2\pi + m^-)$. 
The figure~\ref{fig:stencilalpha} shows the paths along which one can reach the
reference points from the point $(\tfrac{\pi}{2} , \tfrac{\pi}{2})$ belonging to the 
physical sheet.  

\begin{figure}[h]
  \centering\includegraphics[width=0.6\textwidth]{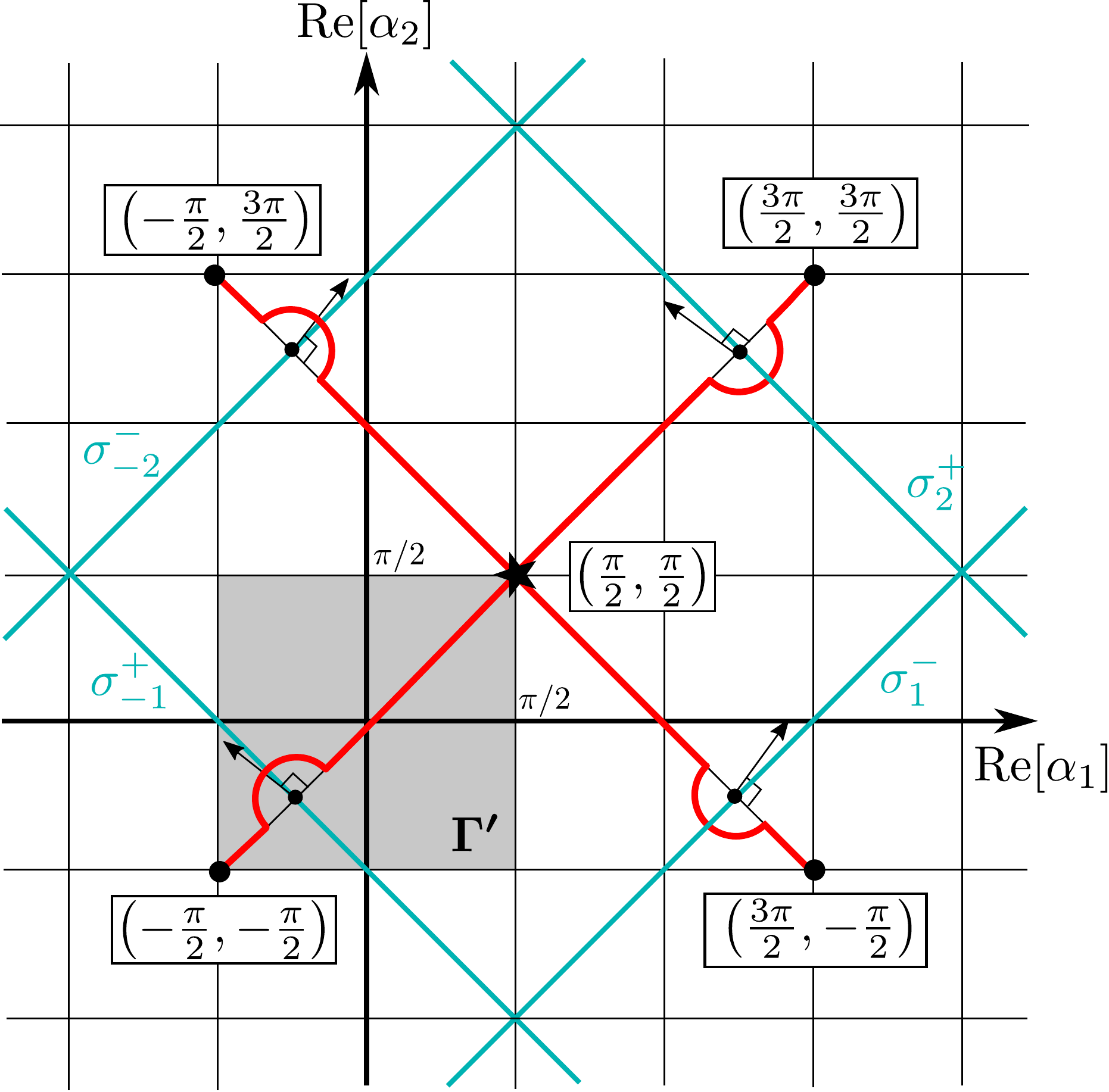}
  \caption{Reference points used for the derivation of the stencil equation
  (\ref{eq:stencil})}
\label{fig:stencilalpha}
\end{figure} 

Let us now combine the additive branching relation (\ref{eq:sheetfixing3})
with the analyticity relations (\ref{eq:sheetfixing4}),
(\ref{eq:sheetfixing5}), (\ref{eq:sheetfixing6}).
As a result, we obtain the so-called \emph{stencil equation} 
\begin{equation}
      \hat{W} (\alpha_1, \alpha_2) + \hat{W} (2 \pi + \alpha_1, 2 \pi +
      \alpha_2) - \hat{W} (\alpha_1, 2 \pi + \alpha_2) - \hat{W} (2 \pi +
      \alpha_1, \alpha_2)  =  0,  \label{eq:stencil}
\end{equation} 
valid for $(\alpha_1 , \alpha_2) \in m^- \times m^-$. 

The values of $\hat W$ in (\ref{eq:stencil}) are chosen by continuity from the reference points shown in figure~\ref{fig:stencilalpha}. Therefore, we obtain the \emph{stencil formulation} (\textbf{StF}) of the angular functional formulation \textbf{ASF}, as summarised in following proposition:

\begin{proposition}
  \label{prop:stencilform}If $\hat{W} (\tmmathbf{\alpha})$ is a function that
  satisfies the angular formulation \textup{\textbf{ASF}} of Proposition \ref{prop:angularmodel},
  then it has the following properties (they compose the stencil formulation \textup{\textbf{StF}}):
  \begin{enumerate}[label=\textup{StF\arabic*}]
    \item \label{itm:1prop6}$[(\alpha_1 + \alpha_2 + \pi / 2) (\alpha_1 +
    \alpha_2 - 5 \pi / 2) (\alpha_1 - \alpha_2 - 3 \pi / 2) (\alpha_1 -
    \alpha_2 + 3 \pi / 2)]^{- 1 / 2} \hat{W} (\alpha_1, \alpha_2)$ is analytic
    on $\Omega \cup \Omega_{+ -} \cup \Omega_{- +} \cup \Omega_{- -}$.
    
    \item \label{itm:2prop6}$\hat{W} (\tmmathbf{\alpha})$ obeys the stencil equation (\ref{eq:stencil}) on $m^- \times m^-$.
    
    \item \label{itm:3prop6}$\hat{W} (\tmmathbf{\alpha})$ obeys the relations
    (\ref{eq:sheetfixing4}),
    (\ref{eq:sheetfixing5}), (\ref{eq:sheetfixing6}) for 
    $(\alpha_1 , \alpha_2) \in m^- \times m^-$.    
        
	\item \label{itm:4prop6} $\hat{W}(\boldsymbol{\alpha})$ satisfies the growth condition \ref{itm:4prop4}.
  \end{enumerate}
  Reciprocally, if a function satisfies the stencil formulation \textup{\textbf{StF}}, then it satisfies
  the angular formulation \textup{\textbf{ASF}}.
\end{proposition}

%%%%%%%%%%%%%%%%%%%%%%%%%%%%%%%%%%%%%%%%%%%%%%%%%%%%%%%%%%%%%%%%%%
\subsection{Separation of variables for the stencil equation}
\label{sec:seprarablestencil}

Let us now focus on finding some solutions obeying the stencil functional problem \textbf{StF} formulated
above in Proposition~\ref{prop:stencilform}. 
Note that we are not trying to build a general solution of the stencil functional problem. Instead, we are building a basis of partial solutions, and later on, 
by comparison with the standard separation of variables for a quarter-plane, the way of constructing a general solution will become clear.

We will
start by concentrating solely on items \ref{itm:1prop6} and \ref{itm:2prop6},
that is the domain of analyticity of $\hat{W}$ and the stencil equation
(\ref{eq:stencil}). Then we will prove that the obtained solutions obey
{\tmem{naturally}} the symmetry conditions of the third item.

Let us take a somewhat similar approach to that of separation of variables.
Since the singularities of $\hat{W}$ are located along the lines
$\sigma^{\pm}_n$, it seems natural to change the variables from
$\tmmathbf{\alpha}= (\alpha_1, \alpha_2)$ to $\tmmathbf{\beta}= (\beta_1,
\beta_2)$ defined by
\begin{eqnarray}
  \beta_1 \equiv \alpha_1 + \alpha_2 & \tmop{and} & \beta_2 \equiv \alpha_1 -
  \alpha_2 .  \label{eq7008}
\end{eqnarray}
The possible singularities are now given by the lines
\begin{eqnarray*}
  \sigma^+_n : \beta_1 = \pi / 2 + \pi n & \tmop{and} & \sigma^-_n : \beta_2 =
  \pi / 2 + \pi n,
\end{eqnarray*}
and we can consider the new unknown function $W^{\dagger} (\tmmathbf{\beta})$
defined by
\[ W^{\dag} (\beta_1, \beta_2) = \hat{W} ((\beta_1 + \beta_2) / 2, (\beta_1 -
   \beta_2) / 2) . \]
Using this change of variables, the stencil equation (\ref{eq:stencil})
becomes 
\begin{align}
  W^{\dag}  (\beta_1, \beta_2) + W^{\dag}  (\beta_1 + 4 \pi, \beta_2) -
  W^{\dag}  (\beta_1 + 2 \pi, \beta_2 - 2 \pi) - W^{\dag}  (\beta_1 + 2 \pi,
  \beta_2 + 2 \pi) & =  0,  \label{eq7009}
\end{align}
where $\tmmathbf{\beta}$ belongs to the image of $m^- \times m^-$ by
(\ref{eq7008}) and the point connections and singularities bypasses can be
visualised in a similar way as in figure~\ref{fig:stencilalpha}, but this time
in the real $\beta_1, \beta_2$ plane, as can be seen in figure~\ref{fig:stencilbeta}. An important step of our consideration will be to
assume that (\ref{eq7009}) can be continued to all $(\beta_1, \beta_2)$
belonging to the Riemann manifold over $\mathbb{C}^2$ associated to
$W^{\dagger}$.

\begin{figure}[h]
  \centering\includegraphics[width=0.6\textwidth]{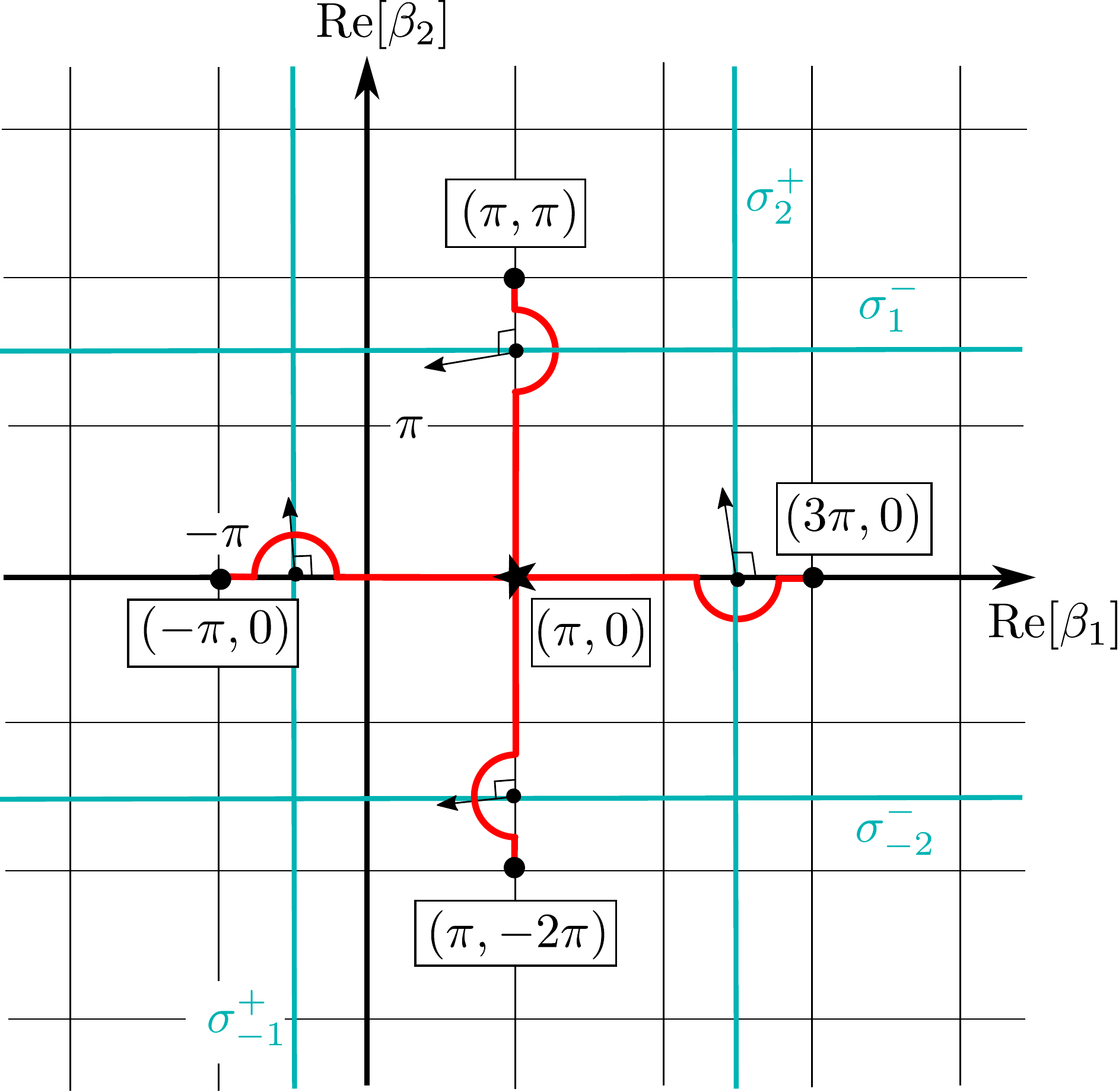}
  \caption{Points connection and singularities bypasses for the
  $\tmmathbf{\beta}$ stencil equation (\ref{eq7009})}
\label{fig:stencilbeta}
\end{figure}

Now, upon dividing (\ref{eq7009}) through by $W^{\dagger} (\beta_1 + 2 \pi,
\beta_2)$, we get
\begin{align}
  \frac{W^{\dag}  (\beta_1, \beta_2) + W^{\dag}  (\beta_1 + 4 \pi,
  \beta_2)}{W^{\dagger} (\beta_1 + 2 \pi, \beta_2)} & =  \frac{W^{\dag} 
  (\beta_1 + 2 \pi, \beta_2 - 2 \pi) + W^{\dag}  (\beta_1 + 2 \pi, \beta_2 + 2
  \pi)}{W^{\dagger} (\beta_1 + 2 \pi, \beta_2)},  \label{eq7010}
\end{align}
and looking for separated solutions of the form
\begin{align}
  W^{\dagger} (\beta_1, \beta_2) & =  \Theta (\beta_1) \Psi (\beta_2), 
  \label{eq7011}
\end{align}
a standard separation of variables argument implies that 
\begin{align}
  \Theta (\beta_1) + \Theta (\beta_1 + 4 \pi) - \lambda \Theta (\beta_1 + 2
  \pi) & =  0,  \label{eq7012}\\
  \Psi (\beta_2 - 2 \pi) + \Psi (\beta_2 + 2 \pi) - \lambda \Psi (\beta_2) & =
   0,  \label{eq:7013}
\end{align}
where $\lambda$ is a separation constant. We can hence look for solutions $W^{\dagger}(\boldsymbol{\beta};\lambda)$ of the form
$W^{\dagger} (\beta_1, \beta_2 ; \lambda) = \Theta (\beta_1 ; \lambda) \Psi
(\beta_2 ; \lambda)$, where $\Theta (\beta_1 ; \lambda)$ and $\Psi (\beta_2,
\lambda)$ satisfy (\ref{eq7012}) and (\ref{eq:7013}) respectively \RED{and are such that $\Theta$ (resp. $\Psi$) is analytic on the strip $-\pi \leq \text{Re}[\beta_1] \leq 3\pi$ (resp. $-\pi \leq \text{Re}[\beta_2] \leq 3\pi$), except for two branch points at $\beta_1=-\pi/2,5\pi/2$ (resp. $\beta_2=\pm3\pi/2$).}

Remember that this has been obtained solely from considering the items
\ref{itm:1prop6} and \ref{itm:2prop6} of the stencil equation formulation of
Proposition \ref{prop:stencilform}. Let us now focus on the item
\ref{itm:3prop6}. Under the $\boldsymbol{\beta}$ formulation, this item can be rewritten
as
\begin{align}
  W^{\dagger} (\pi - \beta_2, \pi - \beta_1 ; \lambda) & =  W^{\dagger}
  (\beta_1, \beta_2 ; \lambda),  \label{eq:symeq1beta}\\
  W^{\dagger} (\beta_2 + \pi, \beta_1 - \pi ; \lambda) & =  W^{\dagger}
  (\beta_1, \beta_2 ; \lambda) ,  \label{eq:symeq2beta}
\end{align}
and should be valid on the image of $m^-\times m^-$ by (\ref{eq7008}). Conveniently, these symmetry conditions arise \textit{naturally} from the separated form of our solution when making certain assumptions summarised in the following lemma.
\begin{lemma}
  \label{lem:lem7}The symmetry equations (\ref{eq:symeq1beta}) and
  (\ref{eq:symeq2beta}) are automatically satisfied provided that
  \begin{align}
    \Theta (\beta + \pi ; \lambda) & =  \Psi (\beta ; \lambda), 
    \label{eq:cond1lem7}\\
    \Psi (- \beta ; \lambda) & =  \Psi (\beta ; \lambda) . 
    \label{eq:cond2lem7}
  \end{align}
  Moreover, these two conditions imply that
  \begin{align}
    \Theta (\beta ; \lambda) & =  \Theta (2 \pi - \beta ; \lambda) . 
    \label{eq:cond3lem7}
  \end{align}
\end{lemma}

\begin{proof}
  In this proof, for brevity, we will drop the $\lambda$ dependency as it will
  be assumed throughout. Let us start by proving that $\left(
  \ref{eq:symeq1beta} \right)$ is satisfied, by first noticing that according
  to (\ref{eq7011}), we have
  \begin{align}
    W^{\dagger} (\beta_1, \beta_2) & =  \Theta (\beta_1) \Psi (\beta_2) , 
    \label{eq:lem71}\\
    W^{\dagger} (\pi - \beta_2, \pi - \beta_1) & =  \Theta (\pi - \beta_2)
    \Psi (\pi - \beta_1) . \label{eq:lem72}
  \end{align}
  All we need to do now is to prove that the RHS of (\ref{eq:lem71}) and
  (\ref{eq:lem72}) are equal by using the conditions (\ref{eq:cond1lem7}) and
  (\ref{eq:cond2lem7}):
  \begin{eqnarray*}
    & \Theta (\pi - \beta_2) \Psi (\pi - \beta_1) \underset{\left(
    \ref{eq:cond1lem7} \right)}{=} \Psi (- \beta_2) \Psi (\pi - \beta_1)
    \underset{\left( \ref{eq:cond2lem7} \right)}{=} \Psi (\beta_2) \Psi
    (\beta_1 - \pi) \underset{\left( \ref{eq:cond1lem7} \right)}{=} \Psi
    (\beta_2) \Theta (\beta_1) , & 
  \end{eqnarray*}
  as required. In order to get (\ref{eq:symeq2beta}), notice that by
  (\ref{eq7011}), we have
  \begin{align}
    W^{\dagger} (\beta_2 + \pi, \beta_1 - \pi) & =  \Theta (\beta_2 + \pi)
    \Psi (\beta_1 - \pi),  \label{eq:lem73}
  \end{align}
  hence all we need to do is to show that the RHS of (\ref{eq:lem71}) and
  (\ref{eq:lem73}) are equal to each other by using to conditions
  (\ref{eq:cond1lem7}) and (\ref{eq:cond2lem7}):
  \begin{eqnarray*}
    & \Theta (\beta_2 + \pi) \Psi (\beta_1 - \pi) \underset{\left(
    \ref{eq:cond1lem7} \right)}{=} \Psi (\beta_2) \Theta (\beta_1), & 
  \end{eqnarray*}
  as required. The resulting condition (\ref{eq:cond3lem7}) at the end of the
  lemma can also be obtained as follows:
  \begin{eqnarray*}
    & \Theta (\beta) \underset{\left( \ref{eq:cond1lem7} \right)}{=} \Psi
    (\beta - \pi) \underset{\left( \ref{eq:cond2lem7} \right)}{=} \Psi (\pi -
    \beta) \underset{\left( \ref{eq:cond1lem7} \right)}{=} \Theta (2 \pi -
    \beta), & 
  \end{eqnarray*}
  as required.
%{\link{\qedsymbol}}
\end{proof}

Using the application of the separation of variables performed above and the previous lemma, we can hence make further progress in constructing some solutions to the stencil formulation problem \textbf{StF} as summarised in the following proposition. This constitutes a new functional problem, the one-dimensional stencil formulation \textbf{1DSt}.

\begin{proposition} 
\label{prop:stencilsepvar}
Let us assume that for some constant $\lambda$ there exists a function $T(\beta) = T(\beta; \lambda)$ satisfying the following properties: 
\begin{enumerate}[label=\textup{1DSt\arabic*}]
   \item \label{hyp:1th9} the combination
    \begin{eqnarray}
      & (\beta + \pi / 2)^{- 1 / 2}  (\beta - 5 \pi / 2)^{- 1 / 2} T (\beta)
      &  \label{eq:squarerootbehaviour}
    \end{eqnarray}
    is analytic in the strip $- \pi \le \mathrm{Re} [\beta] \le 3 \pi$;
    
    \item \label{hyp:2th9}$T (\beta)$ obeys the functional equation 
    \begin{eqnarray}
      T (\beta) + T (\beta + 4 \pi) - \lambda T (\beta + 2 \pi) & = & 0; 
      \label{eq7031}
    \end{eqnarray}
 
   \item \label{hyp:3th9}$T (\beta)$ obeys a growth restriction of the form
    \begin{eqnarray}
      |T (\beta) | & < & C \exp \{\kappa \hspace{0.17em} | \mathrm{Im} [\beta]
      |\}  \label{eq7032}
    \end{eqnarray}
    for some real constants $C$ and $\kappa$;
   \item \label{hyp:4th9}$T$ obeys the symmetry condition 
    \begin{equation}
    T(\beta) = T(2\pi - \beta).
    \label{eq7032b}
    \end{equation}   
\end{enumerate}
Then the function $\hat{W}(\boldsymbol{\alpha})$ defined by 
\begin{equation}
\hat W(\alpha_1, \alpha_2) = 
T(\alpha_1 + \alpha_2; \lambda) T(\alpha_1 - \alpha_2 + \pi; \lambda)
\label{eq7032a}
\end{equation}
obeys the stencil formulation \textup{\textbf{StF}} of Proposition~\ref{prop:stencilform}. 

Any finite linear combination of such $\hat W(\alpha_1, \alpha_2) $
taken for different $\lambda$ also obeys the stencil formulation \textup{\textbf{StF}}.

\end{proposition}

Here $T(\beta ; \lambda)$ is to be understood as playing the role of $\Theta(\beta; \lambda) = \Psi(\beta- \pi; \lambda)$
introduced above. The stencil equation for $T$, the positions of branch points, 
and the symmetry conditions are coming directly from those for~$\Theta$.

A stencil equation for a branching function should be clarified by indicating the positions of the points linked by this equation. To make this clarification, we 
show in figure~\ref{fig:pathbetaplane} 
the reference points $(-\pi , \pi , 3 \pi)$ and the paths connecting them in the 
$\beta$-plane. The other triplets $(\beta, \beta+ 2\pi , \beta+ 4\pi)$
can be obtained from the reference triplet by continuity. 
\begin{figure}[h]
  \centering\includegraphics[width=0.4\textwidth]{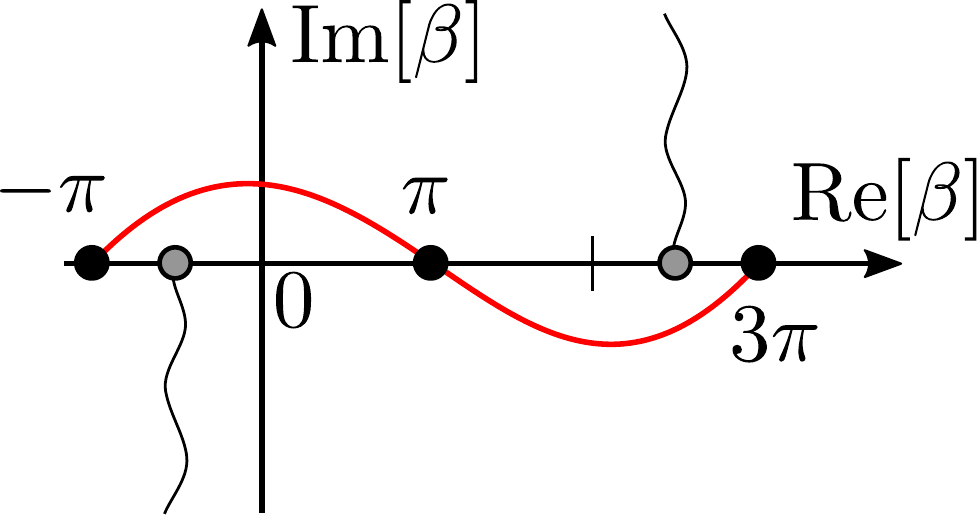}
  \caption{Connection of the points for the separated stencil equation
  (\ref{eq7012}) and illustration of the branch points at $-\tfrac{\pi}{2}$ and $\tfrac{5\pi}{2}$ and their associated cuts.}
\label{fig:pathbetaplane}
\end{figure}

Naturally, the growth condition (\ref{eq7032}) guarantees the fulfillment of the condition~\ref{itm:4prop4}
of Proposition~\ref{prop:stencilform}.

%%%%%%%%%%%%%%%%%%%%%%%%%%%%%%%%%%%%%%%%%%%%%%%%%%%%%%%%%%%%%%%%%%

\subsection{Solving one-dimensional stencil equations with ODEs}
\label{sec:stencil2odetext}

Solutions of the functional problem \textbf{1DSt} formulated in Proposition~\ref{prop:stencilsepvar}
can be obtained by solving an ordinary differential equation 
(ODE) of Fuchsian type. This link is summarised in the following theorem, the proof of which can be found in Appendix~\ref{app:C1}.  

\begin{theorem}
  \label{th:stencil2ODE}  
	  
  Let $T(\beta) = T(\beta ; \lambda)$ be a function obeying the conditions \ref{hyp:1th9}-\ref{hyp:4th9}
   of Proposition~\ref{prop:stencilsepvar}.    
   Then $T (\beta)$ is a solution of an ODE of the form
  \begin{equation}
    \left[ \frac{d^2}{d \beta^2} + f (\beta) \frac{d}{d \beta} + g (\beta)
    \right] T (\beta)  =  0 ,
    \label{eq7033}
  \end{equation}
  where $f (\beta)$ and $g (\beta)$ are rational functions (i.e.\ ratios of
  polynomials) of $\mathfrak{r} = e^{i \beta}$. 
  The coefficients $f$ and $g$ obey the symmetry relations 
  \begin{equation}   
  f (- \beta) = - f (\beta),\qquad g (- \beta) = g (\beta),  
 \label{eq7033a}  
  \end{equation}  
  and are bounded as $|{\rm Im} [\beta]| \to \infty$.

  %More precisely, we can write
  %\begin{eqnarray}
  %  f (\beta) = \frac{- 1}{\cos (\beta)}  \frac{p_{0, 2} (e^{i \beta})}{p_{0,
  %  1} (e^{i \beta})} & \tmop{and} & g (\beta) = \frac{1}{\cos (\beta)} 
  %  \frac{p_{1, 2} (e^{i \beta})}{p_{0, 1} (e^{i \beta})}, 
  %  \label{eq:fandgth9}
  %\end{eqnarray}
  %for some polynomials $p_{m, n}$ such that $p_{m, n} (\pm i) \neq 0$. 
\end{theorem}

% {\bf Raphael, are you sure about the last formula?}

Theorem~\ref{th:stencil2ODE} embeds the 1D stencil equations 
(\ref{eq7012}) and (\ref{eq:7013}) into the rich context of Fuchsian equations. Fuchsian ODEs are linear ODEs whose coefficients are rational functions and whose singular points (including infinity) are all regular singular points.
Indeed, using the change of variables $\beta \to \mathfrak{r} = e^{i \beta}$, the equation (\ref{eq7033}) becomes a Fuchsian ODE of the second order. We are not planning to use this equation, but, for completeness, we write its form here: 
\begin{equation}
\left[
\frac{d^2}{d\mathfrak{r}^2} + f^\sharp (\mathfrak{r}) \frac{d}{d\mathfrak{r}} + g^\sharp(\mathfrak{r})
\right] T^\sharp (\mathfrak{r}) = 0, 
\label{eq:fuchsian1}
\end{equation}
\[
T^\sharp(\mathfrak{r}) = T(\beta(\mathfrak{r})),
\qquad 
f^\sharp(\mathfrak{r}) = \frac{1-i f(\beta(\mathfrak{r}))}{\mathfrak{r}},
\qquad
g^{\sharp} (\mathfrak{r}) = -\frac{g(\beta (\mathfrak{r}))}{\mathfrak{r}^2},
\qquad
\beta (\mathfrak{r}) = -i \log(\mathfrak{r}). 
\]
   
Such a Fuchsian equation possesses (regular) {\em singular points} that are the (polar) singularities 
of the coefficients of (\ref{eq:fuchsian1}). These singular points can be divided into two sorts: 
they may be {\em false\/} or {\em strong\/} (see~\cite{Falsepoints}).  
A strong singular point is a point at which not only the coefficients 
$f^\sharp$ and $g^\sharp$ have singularities, but also the solutions
(at least one of the two linearly independent solutions) have singularities. At false singular points, conversely, 
both linearly independent solutions are regular. 
The removing of false singular points by isomonodromic 
transformation is the subject of~\cite{Falsepoints}. 

The only possible strong singular points of equation (\ref{eq:fuchsian1})
can be found from the behaviour of the functions $T(\beta)$ 
and $R(\beta) = T(\beta + 2\pi)$. They 
are $\mathfrak{r}=i, -i, 0, \infty$. As expected, all singular points are regular (in the usual 
sense). The pairs of exponents 
at the points $\pm i$ are $(0, 1/2)$ (see the condition~\ref{hyp:1th9} of Proposition~\ref{prop:stencilsepvar}). 
The exponents of the points $0$ and $\infty$ are unknown {\em a priori}.

Let us call the equation (\ref{eq:fuchsian1}) {\em minimal\/} 
if it has only strong singular points and let us also call the equation (\ref{eq7033}) minimal 
if the equation (\ref{eq:fuchsian1}) resulting from it is minimal. 
Below we only study the properties of minimal equations (\ref{eq7033}).
In particular, we will prove that the wave field $u$ obtained from a minimal equation 
by the procedure 
\[
T(\beta) 
\rightarrow 
\hat W(\alpha_1 , \alpha_2)
\rightarrow
\tilde W(\xi_1 , \xi_2)
\rightarrow
u(x_1, x_2, x_3)
\]
obeys the edge conditions. Moreover, a more detailed study, which falls
beyond the scope of the present work, shows that if the equation is not minimal then the resulting wave field $u$ does not obey the edge conditions. 

Note that an ODE with four strong regular singular points is Heun's equation. Unfortunately, no analytical representation for its solution or at least its monodromy matrix is known. 

The form of a minimal equation (\ref{eq7033}) is given by the following proposition, the proof of which can be found in Appendix~\ref{appB_2}. 

\begin{proposition}
\label{prop:minimal}
For the equation (\ref{eq7033}) to be minimal and for its solutions to have the correct behaviour at the points $\beta = \pm \pi/2$, the coefficients $f$ and $g$ have to be
\begin{equation}
f(\beta) = -\frac{1}{2}\tan \beta,
\qquad 
g(\beta) = \frac{a}{\cos \beta} + b
\label{eq:minimal}
\end{equation}
for some constants $a$ and $b$.
\end{proposition}

Hence if the equation (\ref{eq7033}) is minimal, it takes the form 
\begin{equation}
\left[
\frac{d^2}{d \beta^2} - \frac{\tan \beta}{2}  \frac{d}{d\beta}
+ \frac{a}{\cos \beta} + b
\right]T(\beta) = 0,
\label{eq7033min}
\end{equation}
for some arbitrary constant parameters $a$ and $b$ that should be found from some additional conditions. 

Let us now consider the points $\beta=-\pi/ 2$, $\pi / 2$, $3\pi /2$ in more details. All these points are regular singular points of (\ref{eq7033min})
with a pair of exponents $(0,1/2)$. From the theory of ODEs 
(see e.g.\ \cite{Olver}), it is known that at each point there exists 
a basis of two linearly independent solutions of (\ref{eq7033min}), such that the 
first component of this basis is regular at the corresponding singular point, 
and the second solution is a regular function multiplied by a square root 
singularity. Let us write these basis functions in the form
\[
B_{-\pi/2}(\beta) = 
\left( \begin{array}{c}
\psi_{1,1} (\beta + \pi/2) \\
(\beta + \pi/2)^{1/2} \psi_{2,1} (\beta + \pi/2)
\end{array} \right),
\]
\[ 
B_{\pi/2}(\beta) = 
\left( \begin{array}{c}
\psi_{1,2} (\beta - \pi/2) \\
(\beta - \pi/2)^{1/2} \psi_{2,2} (\beta - \pi/2)
\end{array} \right),
\]
\[
B_{3\pi/2}(\beta) = 
\left( \begin{array}{c}
\psi_{1,3} (\beta - 3\pi/2) \\
(\beta - 3\pi/2)^{1/2} \psi_{2,3} (\beta - 3\pi/2)
\end{array} \right),
\]
where $\psi_{m,n}$ are functions holomorphic in some neighbourhood of zero.  
Since there can only be two linearly independent solutions 
of (\ref{eq7033min}), the bases can be linearly expressed in terms of each other
by {\em connection matrices}:
\begin{equation}
B_{\pi / 2} =\mathsf{M} B_{-\pi/2}%=\left( \begin{array}{cc}
%m_{1,1} & m_{1,2} \\
%m_{2,1} & m_{2,2}
%\end{array} \right) B_{-\pi/2} 
= \mathsf{N} B_{3\pi/2} %=
%\left( \begin{array}{cc}
%n_{1,1} & n_{1,2} \\
%n_{2,1} & n_{2,2}
%\end{array} \right) B_{3\pi/2}
\label{eq:minimal_1}
\end{equation}
for some constant $2\times 2$ matrices $\mathsf{M}=(\mathsf{m}_{j,\ell})$ and $\mathsf{N}=(\mathsf{n}_{j,\ell})$.

The following proposition, the proof of which can be found in Appendix~\ref{appB_3}, formulates the restrictions imposed on the solutions of (\ref{eq7033min}). These conditions should be satisfied by choosing appropriate
values of $a$ and $b$.

\begin{proposition}
\label{prop:monodromy}
Let the constants $a$ and $b$ introduced in (\ref{eq:minimal}) be chosen in such a way that
the connection matrices defined in (\ref{eq:minimal_1})
for the equation (\ref{eq7033min}) have the following properties: 
\[
\mathsf{m}_{1,1} = 0 \text{ and } \mathsf{n}_{1,2} =0.
\]
Then (\ref{eq7033min}) has a solution $T (\beta)$ obeying the conditions 
of Proposition~\ref{prop:stencilsepvar}.
\end{proposition}

%%%%%%%%%%%%%%%%%

\subsection{Link between (\ref{eq7033min}) and the Lam{\'e} equation}
\label{sec:Lame}

The main aim of this section is to highlight the strong link between the stencil equation (that is a result of the additive crossing property) and Lam\'e equation. The aim being to make a connection between the separated solution of the physical field (treated for example in \cite{kraus}) and our separation technique in the Fourier space. 
In order to do so, let us introduce the change of variable 
\begin{equation}
\chi = \chi(\beta) = \arccos (\sqrt{2} \cos (\beta / 2)),
\label{eq:chi01}
\end{equation}
which maps the segment $[\pi / 2, 3 \pi /2]$ of $\beta$ onto the segment $[0 , \pi]$ of~$\chi$. Upon introducing the function $T^L (\chi)$ defined by 
\begin{equation}
T^L(\chi(\beta)) = T(\beta),
\label{eq:chi02}
\end{equation}
the ODE (\ref{eq7033min}) can be written in terms of the $\chi$ variable and becomes
\begin{equation}
\left[
\frac{d^2}{d \chi^2} + \frac{\cos \chi \, \sin \chi}{2 - \cos^2 \chi}
\frac{d}{d\chi}
+
\frac{\nu(\nu +1) \sin^2 \chi}{2 - \cos^2 \chi} 
-
\frac{4a}{2 - \cos ^2 \chi}   
\right] T^L(\chi) = 0,
\label{eq:chi03}
\end{equation}
where $\nu$ is a constant parameter linked to $b$ by the relations
\begin{align}
\nu (\nu + 1) &= 4 b, & 
\nu &= \frac{-1+ \sqrt{1 + 16 b}}{2}.
\label{eq:chi04}
\end{align}
%{\color{red}I think the equation giving $\nu$ was wrong, the sqrt should also be divided by 2. Moreover, even writing this somehow makes some assumptions on the sign of $1+16b$. So maybe scrap this altogether?}
In a very similar way, let us now introduce the change of variable 
\begin{equation}
\tau = \tau(\beta) = \arccos (\sqrt{2} \sin (\beta / 2)).
\label{eq:chi05}
\end{equation}
mapping the segment $[-\pi/2 , \pi/2]$ of $\beta$ onto the segment 
$[0,\pi]$ of $\tau$.
Upon introducing the function
$T^R (\tau)$ defined by 
\begin{equation}
T^R(\tau(\beta)) = T(\beta),
\label{eq:chi06}
\end{equation}
the equation (\ref{eq7033min}) transforms into 
\begin{equation}
\left[
\frac{d^2}{d \tau^2} + \frac{\cos \tau \, \sin \tau}{2 - \cos^2 \tau}
\frac{d}{d\tau}
+
\frac{\nu (\nu +1) \sin^2 \tau}{2 - \cos^2 \tau} 
+
\frac{4a}{2 - \cos ^2 \tau}   
\right] T^R(\tau) = 0.
\label{eq:chi07}
\end{equation}

It is, at this stage, important to realise that equations (\ref{eq:chi03}) and (\ref{eq:chi07}) are both Lam{\'e} differential equations written in their trigonometric forms (see  \cite{Bateman3}).

The following proposition is valid:

\begin{proposition}
\label{prop:lame}
If the values of $a$ and $b$ are chosen such that the condition of Proposition~\ref{prop:monodromy}
are satisfied, then there exist solutions $T^L(\chi)$ and $T^R (\tau)$ of equations 
(\ref{eq:chi03}) and (\ref{eq:chi07}), respectively, such that 
\begin{equation}
T^L(\chi) = T^L(-\chi),
\qquad 
T^L(\chi) = T^L(2\pi -\chi),
\label{eq:chi08}
\end{equation} 
\begin{equation}
T^R(\tau) = T^R(-\tau),
\qquad 
T^R(\tau) = -T^R(2\pi -\tau). 
\label{eq:chi09}
\end{equation} 
The inverse statement is also valid. 
\end{proposition}

\begin{proof}
Let $\chi_0\in(0,\pi)$ and consider a direct path $P_{\chi_0}$ in the $\chi$ complex plane joining $\chi_0$ to $-\chi_0$. Then the image of this path in the $\beta$ plane, $2\arccos(\tfrac{1}{\sqrt{2}}\cos(P_{\chi_0}))$, is a closed path starting and ending at $\beta(\chi_0)$ and encircling the point $\beta=\tfrac{\pi}{2}$ once. Since this point is not a branch point of $T(\beta)$, the value of $T$ is the same at the start and the end of this path, implying that $T^L(\chi_0)=T^L(-\chi_0)$. 
Similarly, the image in the $\beta$ plane of a path joining $\chi_0$ to $2\pi-\chi_0$ is a closed path encircling $\beta=\tfrac{3\pi}{2}$, which is not a branch point of $T(\beta)$, implying that $T^L(\chi_0)=T^L(2\pi-\chi_0)$. 
The second equation can be proven similarly by considering $\tau_0\in(0,\pi)$, and note that the image in the $\beta$ plane of a path joining $\tau_0$ to $-\tau_0$ is a closed path encircling $\beta=\tfrac{\pi}{2}$, which is not a branch point (hence $T^R(\tau_0)=T^R(-\tau_0)$), while the image of a path joining $\tau_0$ to $2\pi-\tau_0$ encircles the point $\beta=\tfrac{-\pi}{2}$, which is a branch point of $T(\beta)$ with square root behaviour. Hence $T(\beta)$ changes its sign along this path, which leads to $T^R(\tau_0)=-T^R(2\pi-\tau_0)$.
%{\link{\qedsymbol}}
\end{proof}

It hence transpires that, by comparison\footnote{In particular see Eq. (3.5) and (3.6) of \cite{kraus} when, in their notations, $k=k'=1/\sqrt{2}$.} to \cite{kraus}, the values $a$, $b$ and the functions $T^L$, $T^R$
obey the  Sturm--Liouville problem derived for a flat Dirichlet cone in the
sphero-conal coordinates.

%%%%%%%%%%%%%%%%%%%%%%%%%%%%%%%%%%%%%%%%%%%%%%%%%%%%%%%%%%%%%%%%%%
%%%%%%%%%%%%%%%%%%%%%%%%%%%%%%%%%%%%%%%%%%%%%%%%%%%%%%%%%%%%%%%%%%

\section{Interpretation of the solution} \label{sec:wavefield}

\subsection{Imposing the edge conditions}

Let us reconstruct the wave field using the 
integral representation (\ref{eq:physicalfieldintegralxi}),
and consider the singularities of the field on the edges of the quarter-plane, namely the 
lines $x_1 = 0 , x_2 >0 , x_3 = 0$ and $x_1 > 0, x_2=0 , x_3 = 0$. A usual formulation of a diffraction problem 
includes Meixner conditions at the edges. 
These conditions have been skipped so far, and we now return to them. 

In the case of a half-plane, the 
Meixner condition is equivalent to continuity of the field near the edge. If $\rho$ is the distance between the edge and the observation point, which is close to the 
edge, the field allowed by the Meixner condition behaves as $\sim \rho^{1/2}$, 
while the first prohibited term of the same symmetry type is $\sim \rho^{-1/2}$. 

Thus, in order for the edge condition to be satisfied, it is enough for the integral (\ref{eq:physicalfieldintegralxi}) to be convergent at the edges. 

Let us prove this for the edge $x_1 > 0, x_2 =0$, the other edge can be dealt with similarly.  The integral for the field 
at the edge $x_1 > 0, x_2 = x_3 = 0$ takes the form 
\begin{equation}
u(x_1 , 0 , 0) = - \frac{i}{4\pi^2} 
\int_{\Gamma_\xi} \int_{\Gamma_\xi}
\tilde K (\xi_1 , \xi_2)
\tilde W (\xi_1 , \xi_2) e^{ - i \xi_1 x_1}
 \, d \xi_1 \, d \xi_2,
\label{eq01}
\end{equation}
where we recall that $\Gamma_\xi$ is just the real axis. For the edge condition to be satisfied, we demand that this integral converges. 

To see this, fix $\xi_2\in\Gamma_\xi$ and consider the integral over $\xi_1$. Deform the contour of integration in the $\xi_1$ plane to the contour $\zeta$ as shown in figure~\ref{fig:Add01}. We shift the contour into the lower half-plane
since the term $e^{- i \xi_1 x_1}$ decays exponentially there if $x_1 >0$. The integration over $\xi_2$ is held over the real axis $\Gamma_\xi$. 

\begin{figure}[h]
  \centering\includegraphics[width=0.5\textwidth]{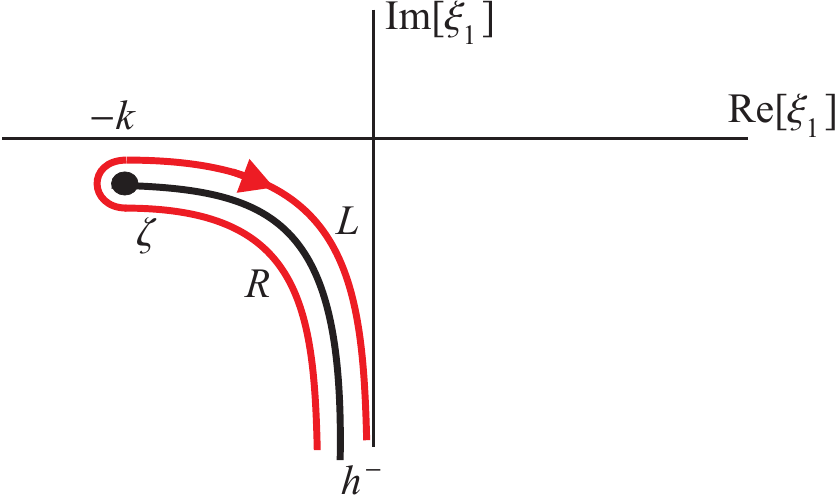}
  \caption{Deformation of the integration contour in the $\xi_1$-plane}
\label{fig:Add01}
\end{figure}

In figure ~\ref{fig:Add01}, the integration contour $\zeta$ is shown at some distance from $h^-$ for clarity. 
In fact, we are letting this distance tend to zero so that for each point of 
$h^-$ there is a small portion of $\zeta$ going from $-i\infty$ to $-k$, and a small 
portion of $\zeta$ going in the opposite direction.

So, the integral becomes rewritten in the form 
\begin{equation}
u(x_1 , 0 , 0) = - \frac{i}{4\pi^2} 
\int_{\Gamma_\xi}
\int_{\zeta}
\tilde K (\xi_1 , \xi_2)
\tilde W (\xi_1 , \xi_2) e^{ - i \xi_1 x_1}
 \, d \xi_1 \, d \xi_2.
\label{eq02}
\end{equation}

Because the function $\tilde W(\xi_1 , \xi_2)$ grows algebraically, the integral 
in the $\xi_1$-plane is always convergent. Our aim is to study the convergence of the 
external integral over~$\xi_2$.

Because, as $|\xi_2|\to\infty$, we have $\tilde{K}(\xi_1,\xi_2)=\mathcal{O}(1/|\xi_2|)$, one can see that the integral is convergent if 
\begin{equation}
\tilde W_L(\xi_1 , \xi_2) - \tilde W_R (\xi_1 , \xi_2) = o(1) 
\text{ as }
|\xi_2| \to \infty,
\label{eq03}
\end{equation}
where $\tilde W_L$ and $\tilde W_R$ are the values taken, respectively, 
on the left and and the right shore of the contour $\zeta$ for some~$\xi_1 \in h^-$. 

In the angular coordinates $(\alpha_1, \alpha_2)$ introduced in Section \ref{sec:formulationangularalpha}, the integration contours have the shape shown in 
figure~\ref{fig:Add01a}. Namely, the integration in the $\alpha_1$-plane
is held over the contour $m^- + \epsilon $, where $\epsilon$ is a small positive real number, 
and the limit $\epsilon \to 0$ corresponds to the contour $\zeta$ falling onto~$h^-$. 
The integration in the $\alpha_2$-plane is held over the contour 
$\Gamma_\alpha=m^- + \pi/2$, which is the image of 
the real axis under the mapping $\xi \to \alpha$, as discussed in Section \ref{sec:formulationangularalpha}.
The contour $m^- + \epsilon $ possesses a symmetry 
$\alpha \to -\pi + 2\epsilon - \alpha$. In the limit $\epsilon \to 0$, 
this symmetry corresponds to the formation of a pair of points belonging to the 
right and the left shore of the cut~$h^-$, as summarised in tables \ref{table:mapping-summary} and \ref{table:symalphaplane}.  
\begin{figure}[h]
  \centering
  \includegraphics[width=0.5\textwidth]{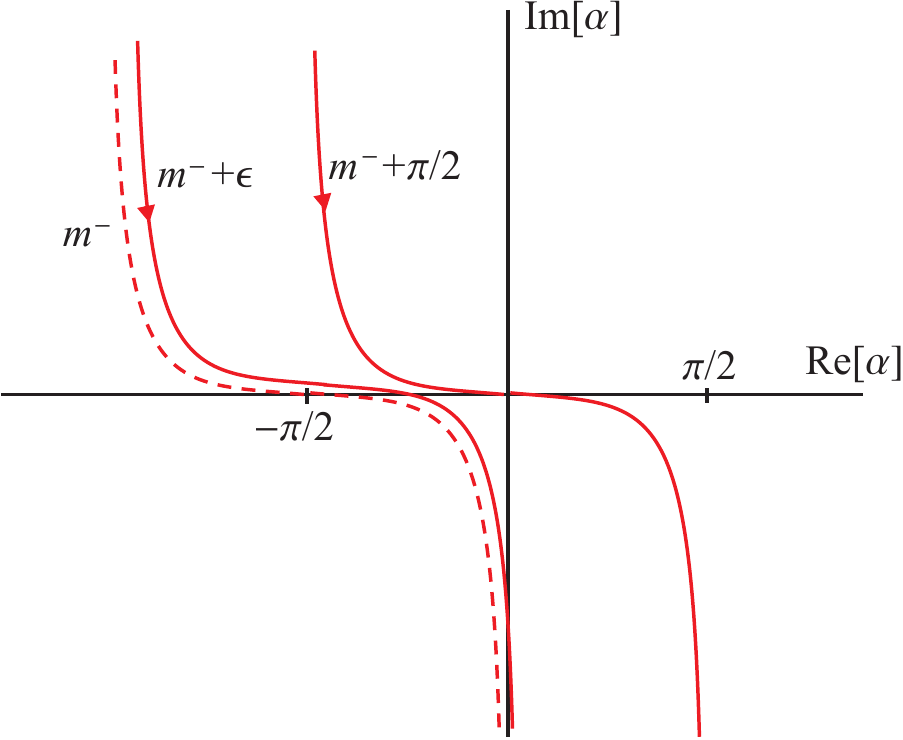}
  \caption{Integration contours in the $\alpha$-planes}
\label{fig:Add01a}
\end{figure}

In the variables $(\alpha_1, \alpha_2)$, the condition (\ref{eq03}) can be rewritten as:
\begin{equation}
\lim_{\epsilon \to 0} \left[
\hat W(\alpha_1 , \alpha_2) - \hat W(-\pi + 2\epsilon-\alpha_1 , \alpha_2)\right] = o(1),
\label{eq03a}
\end{equation}
for $\alpha_1 \in m^- + \epsilon$ and $\alpha_2 \in m^- + \pi / 2$, as $|{\rm Im}[\alpha_2]| \to \infty$.

%%%%%%%%%%%%%%%%%%%%%%%%%%%%%%%%%%

Now let us show that the condition (\ref{eq03a}) is satisfied for the solution $\hat W$ found above in Section \ref{sec:stencil2odetext}. Using the ODE (\ref{eq7033min}) and the restrictions 
on $a$ and $b$ formulated in Section \ref{sec:stencil2odetext}, find a value of $\lambda$ and a corresponding function $T(\beta) = T(\beta ; \lambda)$. Remembering that $T(\beta) = \Theta(\beta) = \Psi(\beta - \pi)$, construct $\hat W$ as follows:
\begin{equation}
\hat W(\alpha_1 , \alpha_2) = 
T(\alpha_1 + \alpha_2) T(\pi - \alpha_1 + \alpha_2).
\label{eq04}
\end{equation}
The function $T(\beta)$ has branch points on the real axis at $\beta=-\pi/2$ and $\beta=5\pi/2$. Cut the complex $\beta$-plane along the intervals $(- \infty, - \pi/2]$,
$[5\pi / 2, \infty)$ belonging to the real axis and note that $T(\beta)$ is single-valued over this cut plane, where it obeys the symmetry condition (\ref{eq7032b}). It is also worth noting that the arguments $\alpha_1 + \alpha_2$ and 
$\pi - \alpha_1 + \alpha_2$ of the function $T$ in (\ref{eq04})
belong to the cut plane if $\alpha_1 \in m^- + \epsilon$, 
$\alpha_2 \in m^- + \pi /2$.

Upon introducing $J(\alpha_1, \alpha_2; \epsilon)\equiv \hat W(\alpha_1 , \alpha_2) - \hat W(-\pi + 2\epsilon-\alpha_1 , \alpha_2)$, we obtain 
\begin{equation}
J(\alpha_1, \alpha_2; \epsilon)=
T(\alpha_1 + \alpha_2) T(\pi - \alpha_1 + \alpha_2)
-
T(\alpha_1 + \alpha_2 + 2\pi - 2\epsilon) T(-\pi - \alpha_1 + \alpha_2 + 2\epsilon),
\label{eq05}
\end{equation}
and we demand that 
\begin{equation}
\lim_{\epsilon \to 0} J(\alpha_1 , \alpha_2; \epsilon) = o(1) \text{ as } |{\rm Im}[\alpha_2]| \to \infty.
\label{eq05a}
\end{equation}
%which is a bit stronger than (\ref{eq05}).

Let us reconsider the ODE (\ref{eq7033min}) obeyed by $T(\beta)$, and focus on its 
behaviour as $\beta \to + i \infty$. As we established in Appendix~\ref{appB_3}, 
$T$ can be spanned over a basis composed of two functions $F_{1,2}$ such that
\[
F_j (\beta) \underset{\beta\to+i\infty}{\sim} e^{\kappa_j \beta},
\]
where $\kappa_{1,2}$ are given in (\ref{eq:proofs3}). It means that there exist two constants $q_{1,2}$ such that
\begin{equation}
T(\beta) = q_1 F_1 (\beta) + q_2 F_2 (\beta).
\label{eq06}
\end{equation}
%We remind (see \ref{eq:proofs6}) that in the upper half plane 
%\[
%F_j (\beta) \underset{\beta\to+i\infty}{\sim} e^{\kappa_j \beta}
%\]
Due to the symmetry (\ref{eq7032b}), in the lower half-plane of $\beta$ 
(here we mean the plane cut over the cuts introduced above) we have
\begin{equation}
T(\beta) = q_1 F_1 (2\pi- \beta) + q_2 F_2 (2\pi- \beta)
\label{eq06a}
\end{equation}
with the same $q_1$ and $q_2$. 

Let us now fix $\alpha_1 \in m^- + \epsilon$ and take $\alpha_2 \in m^- + \pi / 2$ such that ${\rm Im}[\alpha_2] > 0$. After noticing that due to the ansatz (\ref{eq:proofs6}) 
\begin{equation}
F_j(\alpha_1 + \alpha_2) F_j(\pi - \alpha_1 + \alpha_2)
-
F_j(\alpha_1 + \alpha_2 + 2\pi) F_j(-\pi - \alpha_1 + \alpha_2)
=0,
\qquad 
j = 1,2,
\label{eq09}
\end{equation}
substitute (\ref{eq06}) into (\ref{eq05a}) to conclude that 
\begin{align}
\lim_{\epsilon \to 0} J(\alpha_1, \alpha_2; \epsilon) &= 
q_1 q_2 [
F_1 (\alpha_1 + \alpha_2) F_2 (\pi - \alpha_1 + \alpha_2) 
+ 
F_2 (\alpha_1 + \alpha_2) F_1 (\pi - \alpha_1 + \alpha_2)]  \label{eq09a}\\
&- q_1 q_2 [F_1 (\alpha_1 + \alpha_2+ 2\pi) F_2 (-\pi - \alpha_1 + \alpha_2) 
-
F_2 (\alpha_1 + \alpha_2+ 2\pi) F_1 (-\pi - \alpha_1 + \alpha_2)]. \nonumber
\end{align}
Hence, according to (\ref{eq:proofs6}) and (\ref{eq:proofs3}), this expression grows as 
\[
\lim_{\epsilon \to 0} J(\alpha_1, \alpha_2; \epsilon) 
\sim \exp\{ (\kappa_1 + \kappa_2) \alpha_2\} = 
\exp\{ i \alpha_2 /2\}. 
\]
In the upper half-plane, i.e.\ $\text{Im}[\alpha_2]\to+\infty$, this is a decay. 

Due to the symmetry $\beta \to 2\pi - \beta$ of $T$ the function (\ref{eq09a}) decays as well in the lower half- plane and hence the integral (\ref{eq02}) is convergent. 
It is important to note that this consideration is based on the fact that $\kappa_1 + \kappa_2 = i /2$ (see (\ref{eq:proofs3})) and that this fact only holds if the equation (\ref{eq7033}) is taken to be minimal.

%%%%%%%%%%%%%%%%%%%%%%%%%%%%%%%%%%%%%%%%%%%%%%%%%%%%%%%%%%%%%%%%%%

\subsection{Wave field $u(\tmmathbf{x}, x_3)$ as the solution in the sphero-conal variables}
The following theorem describes the link between the solution built in the paper and 
the classical solution arising from separation of variables in the sphero-conal coordinates.  

\begin{theorem}
\label{prop:SphCon}

Let $T(\beta) = T(\beta; \lambda)$ be a solution 
of the ODE inverse monodromy problem formulated in Proposition~\ref{prop:monodromy} for some triplet $(\lambda,a,b)$, and let the field $u(\tmmathbf{x}, x_3)$, $x_3 > 0$, be defined by (\ref{eq:physicalfieldintegralxi}).
Then 
\begin{align}
u(x_1, 
x_2,
x_3
) &=\frac{-i k e^{i\tfrac{\nu \pi}{2}}}{\pi \sqrt{2\pi}}\times \frac{1}{\sqrt{kr}} H_{\nu+1/2}^{(1)}(k r)  
\, T^R(\tau)\, T^L(\chi), \label{eq:Lame01} 
\end{align}
where $\nu$ is related to $b$ by (\ref{eq:chi04}), $H^{(1)}_{\nu+1/2} (kr)$ is the Hankel function of the first kind of order $\nu+1/2$ and $T^R(\tau)$ and $T^L(\chi)$ are related to $T$ by
\begin{align*}
T^R(\tau) &= T(2 \arcsin(\tfrac{1}{\sqrt{2}} \cos \tau)) \text{ and } T^L(\chi)=T(2 \arccos(\tfrac{1}{\sqrt{2}} \cos\chi)).
\end{align*}
The coordinates $(r,\chi,\tau)$ are defined for $0<\chi<\pi$ and $0<\tau<\pi$ by
\begin{align}
x_1 = x_1(r,\chi,\tau)&=r \left[
\frac{\cos \chi}{\sqrt{2}} \sqrt{1 -\frac{\cos^2 \tau}{2}} -
\frac{\cos \tau}{\sqrt{2}} \sqrt{1 -\frac{\cos^2 \chi}{2}}
\right], \label{eq:Lame02} \\
x_2 = x_2(r,\chi,\tau)&=-r \left[
\frac{\cos \chi}{\sqrt{2}} \sqrt{1 -\frac{\cos^2 \tau}{2}} +
\frac{\cos \tau}{\sqrt{2}} \sqrt{1 -\frac{\cos^2 \chi}{2}}
\right], \label{eq:Lame03} \\
x_3 = x_3(r,\chi,\tau) &= r 
\sin \chi \, \sin \tau.
\label{eq:Lame04}
\end{align}
\end{theorem}

\begin{proof}
The proof of this theorem is based on the differential form notations 
introduced in Appendix~\ref{app:Integration} and on the statements proven within this appendix. In particular, we have already shown in Appendix~\ref{app:appA1diffform} that the field $u(\tmmathbf{x}, x_3)$ defined by (\ref{eq:physicalfieldintegralxi}) does satisfy the Helmholtz equation. Hence, by rewriting the Laplace operator in spherical coordinates, we obtain  
\begin{align}
\left[
\frac{\ptl^2 }{\ptl r^2} + \frac{2}{r} \frac{\ptl}{\ptl r}
+
\frac{1}{r^2} \tilde \Delta_{\tmmathbf{\nu}}
\right] u(\tmmathbf{\nu}, r) = 0, 
\label{eq:Lame05}
\end{align} 
where $\tilde{\Delta}_{\boldsymbol{\nu}}$ is the \textit{physical} Laplace-Beltrami operator defined in Appendix \ref{app:LBO} and $\boldsymbol{\nu}$ can be consider as a point on the unit sphere (see (\ref{eq:fr02})).
As discussed in Appendix~\ref{app:Integration} the field $u$ defined by (\ref{eq:physicalfieldintegralxi}) can be rewritten as per (\ref{eq:defunur}) as $u (\tmmathbf{\nu}, r)  =  \int_{\tilde{\Gamma}} w (\tmmathbf{\omega}) \,
p (kr, \tmmathbf{\nu}, \tmmathbf{\omega}) \, \psi_{\tmmathbf{\omega}}$, where $\tilde \Gamma$ is an integration surface that can be slightly deformed in order to ensure the exponential convergence of the integral, as discussed in Appendix \ref{app:SPM}. The point $\boldsymbol{\omega}$ defined in (\ref{eq:fr02}) belongs to the spectral complexified sphere $\mathfrak{S}$ defined in (\ref{eq:fr04}). The function $w:\mathfrak{S}\to\mathbb{C}$, the differential 2-form $\psi_{\boldsymbol{\omega}}$ and plane wave function $p$ are defined in (\ref{eq:fr05}), (\ref{eq:fr06}) and (\ref{eq:fr066}) respectively.

Applying the Laplace-Beltrami operator $\tilde{\Delta}_{\boldsymbol{\nu}}$ to the field $u$, and using Proposition \ref{prop:fr11}, we obtain
\begin{align}
\tilde \Delta_{\tmmathbf \nu} u &= 
\int_{\tilde \Gamma} 
\tilde \Delta_{\tmmathbf \omega}
[p (kr, \tmmathbf{\nu} , \tmmathbf{\omega})] \, 
w(\tmmathbf{\omega}) \, \psi_{\tmmathbf{\omega}},
\label{eq:Lame06}
\end{align}
which, using Proposition \ref{th:f1f2} and the exponential convergence of the integral, leads to
\begin{align}
\tilde \Delta_{\tmmathbf \nu} u &= 
\int_{\tilde \Gamma} 
p (kr, \tmmathbf{\nu} , \tmmathbf{\omega}) \, 
\tilde \Delta_{\tmmathbf \omega}
[w(\tmmathbf{\omega})] \, \psi_{\tmmathbf{\omega}}.
\label{eq:Lame07}
\end{align}

In Section \ref{sec:Lame}, we have demonstrated the link between the inverse monodromy problem 
for equation (\ref{eq7033min}) and the Sturm-Liouville problem for equations
(\ref{eq:chi03}) and (\ref{eq:chi07}). These two equations are the same as that derived in~\cite{kraus}, where they emerged as the result of a separation of variables method applied to the Laplace--Beltrami operator and its associated eigenvalue problem, and are obeyed by $T^R(\tau)$ and $T^L(\chi)$. Hence, since $w(\boldsymbol{\omega})=\tfrac{k}{4\pi^2i}T^R(\tau)\, T^L(\chi)$, it is clear that $w$ has to be an eigenfunction of the operator $\tilde \Delta_{\tmmathbf{\omega}}$. A comparison with \cite{kraus} implies that the associated eigenvalue is $-\nu(\nu+1)$, i.e.\
\begin{align}
\tilde \Delta_{\tmmathbf{\omega}} w (\tmmathbf{\omega})
&=
-\nu (\nu + 1) w (\tmmathbf{\omega}) ,
\label{eq:Lame08}
\end{align}
where  $\nu$ is linked to $b$ by (\ref{eq:chi04}) and the separation constant is $a$. Note that a detailed study of these eigenvalues can be found in \cite{Assier2016}.

Because of this and (\ref{eq:Lame07}), it transpires that $u(\boldsymbol{\nu},r)$ is itself an eigenfunction of the operator $\tilde \Delta_{\tmmathbf{\nu}}$ with the same eigenvalue, and hence, a separation of variables argument implies that
\begin{align}
\left[
r^2\frac{\ptl^2 }{\ptl r^2} + 2r \frac{\ptl}{\ptl r}
+k^2r^2-\nu (\nu + 1) 
\right] u(\tmmathbf{\nu}, r) &= 0. 
\label{eq:Lame09}
\end{align}
This equation (the spherical Bessel equation) can be solved explicitly to show that
\begin{align}
u(\tmmathbf{\nu}, r) &= A(\tmmathbf{\nu}) \frac{1}{\sqrt{kr}} H_{\nu+1/2}^{(1)} (k r) + 
B(\tmmathbf{\nu}) \frac{1}{\sqrt{kr}} H_{\nu+1/2}^{(2)} (k r),
\label{eq:Lame10}
\end{align}
for some unknown functions $A$ and $B$. To find $A$ and $B$, one should consider the  asymptotics of $u$ as $r \to \infty$. First of all, since the radiation condition should be satisfied, we must have $B\equiv 0$. Moreover, by applying the multi-dimensional saddle-point method (see e.g. \cite{Bleistein2012} or \cite{surprise}) to (\ref{eq:Lame07}), and considering solely the leading order, one obtains
\begin{align}
u(\tmmathbf{\nu},r) &= - 2i\pi \frac{e^{i k r}}{kr} w(\tmmathbf{\nu}) + o(\frac{1}{kr}).
\label{eq:Lame11}
\end{align}
Using the large argument asymptotic formula for the Hankel function, (\ref{eq:Lame10}) leads to
\begin{align}
u(\tmmathbf{\nu},r) &= -i \sqrt{\frac{2}{\pi}}e^{-i\frac{\nu \pi}{2}}A(\boldsymbol{\nu}) \frac{e^{ikr}}{kr}+o(\frac{1}{kr}),
\label{eq:Lame111}
\end{align}
and comparing (\ref{eq:Lame111}) and (\ref{eq:Lame11}), we obtain
\begin{align}
A(\tmmathbf{\nu}) = 2\sqrt{2\pi} e^{i\frac{\nu \pi}{2}} w(\tmmathbf{\nu}).
\label{eq:Lame12}
\end{align}
Remembering that because of (\ref{eq:fr05}), we have $w(\boldsymbol{\nu})=\tfrac{k}{4\pi^2i}T^R(\tau)T^L(\chi)$, we can input (\ref{eq:Lame12}) into (\ref{eq:Lame10}) to obtain the expected formula (\ref{eq:Lame01}).
%{\link{\qedsymbol}}
\end{proof}

\begin{remark}
As seen above, it is clear that our solution (\ref{eq:Lame01}) satisfies the Helmholtz equation. Moreover, one should note that with the definition of the variables $(r,\chi,\tau)$, the quarter-plane is described by $\tau=\pi$. Moreover, because of the second equation of (\ref{eq:chi09}), we have $T^R(\pi)=0$, which implies that our solution satisfies the Dirichlet boundary condition on the quarter-plane. Similarly, the other three quadrant of the $x_3=0$ plane are represented by $\chi=\pi$, $\tau=0$ and $\chi=0$ respectively. Moreover, because of Proposition \ref{prop:lame}, we clearly have $\tfrac{T^L}{d\chi}(\pi)=\tfrac{T^L}{d\chi}(0)=\tfrac{T^R}{d\tau}(0)=0$, which implies that our solution satisfies the Neumann boundary condition on the remaining three quadrants. By construction, this solution also satisfies the radiation condition and the edge condition. Hence, as expected, the solution of the type (\ref{eq:Lame01}) can be thought of as a tailored vertex Green's function, i.e.\ a function resulting from a point source placed on the vertex of the quarter-plane, that satisfies all the correct boundary, radiation and edge conditions. Note that this solution, as expected, is singular at the origin and behaves like $r^{-\nu-1}$ as $r\to 0$ and that $\nu$ is directly related to an eigenvalue of the Laplace--Beltrami operator. \RED{The solution (\ref{eq:Lame01}) therefore consists of a discrete set of possible solutions, indexed by the eigenvalues of the Laplace-Beltrami operator. Hence, as expected, the solution to our initial simplified problem \textbf{SFP} is not unique. Given restriction on the growth near the vertex would allow one to select the appropriate solutions from this set.} Such Green's functions, with a source located at a geometric singularity of an obstacle, have proved very useful in diffraction theory. They indeed play a critical role in the derivation of the so-called embedding formulae, which, amongst other achievements, have led to some substantial progress in the quarter-plane diffraction problem (see e.g. \cite{shanin1,Assier2012} \RED{with the modified Smyshlyaev formulae}). \RED{As mentioned in introduction, we note that (\ref{eq:Lame01}) can be recovered using separation of variables in sphero-conal coordinates directly in the physical space as in \cite{kraus}. However with the present article, we want to show how it can be obtained by means of advanced complex analysis from within the Fourier space.} 
%{\color{red}Indeed, representation (\ref{eq:Lame01}) is the solution in 
%the sphero-conal variables $(r, \chi, \tau)$. 

%According to Proposition~\ref{prop:lame}, the solution $u$
%obeys Dirichlet boundary conditions on the first quadrant
%of the $(x_1, x_2)$-plane, and 
%Neumann conditions on the three other quadrants.}
\end{remark}

%%%%%%%%%%%%%%%%%%%%%%%%%%%%%%%%%%%%%%%%%%%%%%%%%%%%%%%%%%%%%%%%%%

\section{Conclusion}

This paper can in some way be considered as a proof of concept. It shows how important the study of functions of several variables in general, and the additive crossing property in particular, can be to diffraction theory.

More precisely, we started with the physical problem of diffraction of a plane wave by a quarter-plane and its associated functional problem \textbf{FP} (\ref{eq:FPproblem})  arising from our previous work \cite{Assier2018a}. We then considered a simplified functional problem \textbf{SFP} (\ref{eq:SFPproblem}) that crucially retained the additive crossing property that existed in \textbf{FP}, but relaxed some of the other assumptions. Using the angular coordinates (\ref{eq:mappingccord}), we showed that it resulted in the so-called Stencil equation (\ref{eq:stencil}). Thanks to the change of variables (\ref{eq7008}) and a separation of variables argument, we successfully embedded our problem within the rich context of Fuchsian ODEs. The property of these ODEs where exploited to construct explicit wave-fields whose Fourier Transforms are solution to the simplified functional problem \textbf{SFP}. These resulting wave-fields are expressed in terms of some Lam\'e functions and some eigenvalues of the Laplace--Beltrami operator, and correspond to tailored vertex Green's functions.

In order to carry out our arguments, it was necessary to introduce and develop the bridge and arrow notations (see Appendix \ref{app:appA}) that allow us to precisely describe the concept of indentation of contour integrations in $\mathbb{C}^2$. It was also necessary to make use of the differential form theory (see Appendix \ref{app:Integration}), and we note that the latter can be extremely useful when trying to build the theory of Sommerfeld integrals in difficult situations (see e.g.\ \cite{Shanin2019SommerfeldTI}). 

\section*{Acknowledgements}
The work of A.V. Shanin has been supported by Russian Science Foundation grant RNF 14-22-00042.
R.C. Assier would like to acknowledge the support by UK EPSRC (EP/N013719/1). Both authors would like to thank the Isaac Newton Institute for
Mathematical Sciences, Cambridge, for support and hospitality during the
programme ``Bringing pure and applied analysis together via the
Wiener--Hopf technique, its generalisations and applications'' where some
work on this paper was undertaken.
A.V.~Shanin is also grateful to the Simons foundation for the support 
of his participation in this programme. 
 This work was supported by EPSRC (EP/R014604/1). We are also grateful to the Manchester Institute for Mathematical Sciences for its financial support.

\bibliography{biblio}
\bibliographystyle{unsrt}%doipubmed}%harvard}% plain}%uabbrvnat}%natbib}%unsrtnat}%

\appendix
\counterwithin{figure}{section}
%%%%%%%%%%%%%%%%%%%%%%%%%%%%%%%%%%%%%%%%%%%%%%%%%%%%%%%%%%%%%%%%%%%%%%%%%%%%%%%%%%%%%%
\section{On Fourier integrals in a 2D domain}\label{app:Integration}

\subsection{Differential form notation}\label{app:appA1diffform}
Upon denoting $r = \sqrt{x_1^2 + x_2^2 + x_3^2}$, let us define $\boldsymbol{\nu}\in\mathbb{R}^3$ and  $\boldsymbol{\omega}\in\mathbb{C}^3$ by
\begin{align}
\tmmathbf{\nu}& =\left(\nu_1,\nu_2,\nu_3\right)=\left( \frac{x_1}{r}, \frac{x_2}{r}, \frac{x_3}{r} \right) \text{ and }  \tmmathbf{\omega}=\left(\omega_1,\omega_2,\omega_3\right)= \left( \frac{\xi_1}{k}, \frac{\xi_2}{k},
\frac{\sqrt{k^2 - \xi_1^2  -\xi_2^2}}{k} \right).  
\label{eq:fr02} 
\end{align}
%{\bf Raphael, should we change notation for $\nu$ in (\ref{eq:chi03})? } No t for now
%\begin{equation}
%\tmmathbf{\omega}= \left( \frac{\xi_1}{k}, \frac{\xi_2}{k},
%\frac{\sqrt{k^2 - \xi_1^2  -\xi_2^2}}{k} \right).
%\label{eq:fr03}  
%\end{equation}
The points $\tmmathbf{\nu}$ are real, and they belong to the unit sphere. The points $\tmmathbf{\omega}$, however, are complex. We will say that they belong to the \textit{complexified unit sphere}, i.e.\ to the manifold $\mathfrak{S}$ defined by 
\begin{equation}
\mathfrak{S} = \{ (\omega_1 , \omega_2 , \omega_3) \in \mathbb{C}^3 \, | \,
\omega_1^2 + \omega_2^2 + \omega_3^2 = 1  \}. 
\label{eq:fr04}
\end{equation}
This manifold is analytic everywhere; it has complex dimension~2 
and real dimension~4. If $\omega_3 \ne 0$, one can take $(\omega_1, \omega_2)$ as its local coordinates. Otherwise, one should take $(\omega_1 , \omega_3)$
as local coordinates (if $\omega_2 \ne 0$) or take $(\omega_2 , \omega_3)$
(if $\omega_1 \ne 0$). 

One can consider the function $\tilde W (- \xi_1, - \xi)$ being defined on 
some subset of $\mathfrak{S}$. More precisely, we will work with the function $w:\mathfrak{S}\to\mathbb{C}$ defined by
\begin{align}
w(\tmmathbf{\omega}) &\equiv \frac{k \tilde{W} (- \omega_1 k, - \omega_2 k)}{4
  \pi^2 i}.
\label{eq:fr05}
\end{align}
The sign of the arguments of $\tilde W$ is chosen for convenience. Initially $w$ is defined only for $(-\omega_1 k , -\omega_2 k) \in \mathbb{R}^2$, but we should note that the solution $T$ of the ODE (\ref{eq7033min}) is defined on the complex plane almost everywhere
(continued along any contour not passing through the singular points). Thus, the same 
is valid for the combination (\ref{eq7032a}). Therefore, $w$ can be continued 
analytically onto 
$\mathfrak{S}$ with some branching, i.e.\ $\bar \partial w  =0$ almost everywhere on $\mathfrak{S}$. Here we use the differential notation from \cite{Shabat2}. 

Let us introduce the differential 2-form $\psi_{\boldsymbol{\omega}}$ on $\mathfrak{S}$ by
\begin{align}
\psi_{\tmmathbf{\omega}} &\equiv \frac{d \omega_1
\wedge d \omega_2}{\omega_3}. 
\label{eq:fr06}
\end{align}
Note that this form is analytic  everywhere on $\mathfrak{S}$. Indeed, it can be continued to the points of 
$\mathfrak{S}$ with $\omega_3 =0$ by the 
relations
\[
\frac{d \omega_1
\wedge d \omega_2}{\omega_3} = 
-
\frac{d \omega_1
\wedge d \omega_3}{\omega_2}
=
\frac{d \omega_2
\wedge d \omega_3}{\omega_1}
\]
valid on $\mathfrak{S}$.   

Upon introducing the following notation for a plane wave: 
\begin{equation}
p (kr, \tmmathbf{\nu}, \tmmathbf{\omega}) \equiv 
\exp \{ikr\,\tmmathbf{\nu}
  \cdot \tmmathbf{\omega} \}
  =
\exp \{ikr (\nu_1 \omega_1 + \nu_2 \omega_2 + \nu_3 \omega_3 ) \},
\label{eq:fr066}
\end{equation}
one can easily check that as a function of $\tmmathbf{\omega}$, 
$p$ is analytic everywhere on $\mathfrak{S}$. It is also possible to rewrite (\ref{eq:physicalfieldintegralxi})
as follows: 
\begin{equation}
u (\tmmathbf{\nu}, r)  =  \int_{\tilde{\Gamma}} w (\tmmathbf{\omega}) \,
p (kr, \tmmathbf{\nu}, \tmmathbf{\omega}) \, \psi_{\tmmathbf{\omega}}, 
  \label{eq:defunur}
\end{equation}
where 
$\tilde \Gamma$ is the integration surface on 
$\mathfrak{S}$ having real dimension~2 and corresponding to the 
real plane $\Gamma_\xi\times \Gamma_\xi$ in the initial $(\xi_1, \xi_2)$-coordinates. The orientation of the 
integration surface is chosen appropriately. 
Note that (\ref{eq:defunur}) is a common representation of a three-dimensional wave field, i.e.\ it is a general plane wave decomposition.

Since all factors in (\ref{eq:defunur}) are analytic on $\mathfrak{S}$,
one can use Stokes' theorem and deform $\tilde \Gamma$ if necessary. 
Such a deformation should be a homotopy and it should not cross the singularities 
of $w$. In this case, the value of the integral remains unchanged after the contour deformation (see e.g.\ \cite{Shabat2}). 
This fact is the main benefit of using the differential form notations. 
We get more possibilities of changing the integration contour comparatively to 
considering the representation  
(\ref{eq:physicalfieldintegralxi})
as two repeated 1D integrals in $\mathbb{C}^1$.

Note that $p$ is a plane wave obeying the Helmholtz equation 
\begin{equation}
\Delta p + k^2 p  = 0.
\label{eq:peq} 
\end{equation}
Since $p$ is the only part of the integrand of (\ref{eq:defunur}) depending on the physical variables $(x_1,x_2,x_3)$, we can conclude that the field $u$ defined by (\ref{eq:physicalfieldintegralxi}) also obeys the Helmholtz equation.

%%%%%%%%%%%%%%%%%%%%%%%%%%%%%%%%%%%
\subsection{The Laplace--Beltrami operator on $\mathfrak{S}$}\label{app:LBO}

Let us introduce the Laplace-Beltrami operator on $\mathfrak{S}$ by introducing the global complex coordinates
$(\theta_{\boldsymbol{\omega}} , \varphi_{\boldsymbol{\omega}})$ on $\mathfrak{S}$. One possible choice 
is to use the formulae
\begin{equation}
\theta_{\boldsymbol{\omega}} = \arcsin(\sqrt{\omega_1^2 + \omega_2^2}), 
\qquad
\varphi_{\boldsymbol{\omega}} = \arctan(\omega_2 / \omega_1).  
\label{eq:fr07}
\end{equation} 
Indeed, one can take a pair $(\omega_2, \omega_3)$ or a pair $(\omega_3, \omega_1)$
instead of the pair $(\omega_1, \omega_2)$. 
The change of variable (\ref{eq:fr07}) is not biholomorphic everywhere, however one can use the 
ambiguity shown above and take a neighbourhood small enough to make the change of variable locally biholomorphic.

Introduce the Laplace-Beltrami operator by the usual formula. For 
any function $\phi(\tmmathbf{\omega})$ in the neighbourhood 
\begin{equation}
\tilde \Delta_{\tmmathbf{\omega}} [\phi]  (\tmmathbf{\omega}) \equiv 
\frac{1}{\sin \theta_{\boldsymbol{\omega}}} \frac{\ptl}{\ptl \theta_{\boldsymbol{\omega}}} \left( 
\sin \theta_{\boldsymbol{\omega}} \frac{\ptl \phi}{\ptl \theta_{\boldsymbol{\omega}}}\right)
+ \frac{1}{\sin^2 \theta_{\boldsymbol{\omega}}}
\frac{\ptl^2 \phi}{\ptl \varphi_{\boldsymbol{\omega}}^2}.  
\label{eq:fr08}
\end{equation}
One can show that this definition is coordinate invariant, i.e.\ the 
result does not depend on the ambiguity of the coordinate change discussed above.

In the coordinates $(\theta_{\boldsymbol{\omega}}, \varphi_{\boldsymbol{\omega}})$, the form $\psi_{\tmmathbf{\omega}}$ has the following representation:
\begin{equation}
  \psi_{\tmmathbf{\omega}}  =  \sin \theta_{\boldsymbol{\omega}} \, d
  \theta_{\boldsymbol{\omega}} \wedge d \varphi_{\boldsymbol{\omega}} . 
  \label{eq:jacobiansphericalsurface}
\end{equation}
An important property of the Laplace-Beltrami operator
on $\mathfrak{S}$, which is indeed a complexification of a corresponding property 
on a real sphere, is the following: 
\begin{proposition}
\label{th:f1f2}
Let $\Gamma$ be a two-dimensional integration manifold on $\mathfrak{S}$ with a boundary $\ptl \Gamma$ at infinity, 
and let $\phi_1 (\tmmathbf{\omega})$ and $\phi_2 (\tmmathbf{\omega})$ be some functions holomorphic at the points of $\Gamma$.   
If $\phi_1$ and $\phi_2$ decay exponentially at infinity, then
\begin{equation}
    \int_{\Gamma} \phi_1 (\tmmathbf{\omega}) 
    \tilde{\Delta}_{\tmmathbf{\omega}} [\phi_2] (\tmmathbf{\omega}) \,
   \psi_{\tmmathbf{\omega}}  =  \int_{\Gamma} \phi_2
    (\tmmathbf{\omega})  \tilde{\Delta}_{\tmmathbf{\omega}} [\phi_1]
    (\tmmathbf{\omega}) \,  \psi_{\tmmathbf{\omega}} 
    \label{eq7079}
  \end{equation}
\end{proposition}

\begin{proof}
Consider the 1-form $\Omega$ in the coordinates $(\theta_{\boldsymbol{\omega}}, \varphi_{\boldsymbol{\omega}})$ defined by 
\begin{equation}
\Omega = \phi_1 \left[\sin \theta_{\boldsymbol{\omega}}  \frac{\ptl \phi_2}{\ptl \theta_{\boldsymbol{\omega}}} d \varphi_{\boldsymbol{\omega}}-
\frac{1}{\sin \theta_{\boldsymbol{\omega}}} \frac{\ptl \phi_2}{\ptl \varphi_{\boldsymbol{\omega}}} d\theta_{\boldsymbol{\omega}}\right]
-
\phi_2 \left[\sin \theta_{\boldsymbol{\omega}}  \frac{\ptl \phi_1}{\ptl \theta_{\boldsymbol{\omega}}} d \varphi_{\boldsymbol{\omega}}-
\frac{1}{\sin \theta_{\boldsymbol{\omega}}} \frac{\ptl \phi_1}{\ptl \varphi_{\boldsymbol{\omega}}} d\theta_{\boldsymbol{\omega}} \right].
\label{eq:fr09}
\end{equation}
A detailed consideration shows that
this form is coordinate-independent in the sense discussed above. 
Note in particular that 
\[
d \Omega = 
\phi_1 (\tmmathbf{\omega}) 
\tilde{\Delta}_{\tmmathbf{\omega}} [\phi_2] (\tmmathbf{\omega})
-
\phi_2 (\tmmathbf{\omega}) 
\tilde{\Delta}_{\tmmathbf{\omega}} [\phi_1] (\tmmathbf{\omega}) .
\]
Applying Stokes' theorem on manifolds (see e.g\ \cite{Shabat2}), we obtain  
\begin{equation}
\int_{\partial \Gamma} \Omega = \int_{\Gamma} d \Omega,
\label{eq:fr10}
\end{equation}
and using the exponential decay at infinity, we get $\int_{\partial \Gamma} \Omega=0$ and hence (\ref{eq7079}) is valid.
%{\link{\qedsymbol}}
\end{proof}

Besides the Laplace--Beltrami operator $\tilde \Delta_{\tmmathbf{\omega}}$,
define a Laplace--Beltrami operator $\tilde \Delta_{\tmmathbf{\nu}}$
on a real sphere. Such operator can be defined explicitly in the usual way using the usual physical spherical variables $(\theta,\varphi)$. The following statement can be checked explicitly in the angular coordinates:

\begin{proposition} 

\begin{equation}
  \tilde{\Delta}_{\tmmathbf{\nu}}  [p (kr, \tmmathbf{\nu}, \tmmathbf{\omega})]
   =  \tilde{\Delta}_{\tmmathbf{\omega}}  [p (kr, \tmmathbf{\nu},
  \tmmathbf{\omega})] .  
  \label{eq7078}
\end{equation}
\label{prop:fr11}
\end{proposition}
 
%%%%%%%%%%%%%%%%%%%%%%%%%%%%%%%%%%%%%%%%%%%%%%%%%%%%%%%%%%%%%%%%%%%%%%%%

\subsection{Integration contours to compute the integral (\ref{eq:physicalfieldintegralxi})}
\label{app:SPM}

This appendix aims at explaining the slight contour deformations required to ensure that the field integral representation (\ref{eq:physicalfieldintegralxi}) remains exponentially convergent. 
Note first that if $x_3 > 0$ then the integral converges exponentially and no contour deformation is required. Thus, it is only 
necessary to regularise the integral for $x_3 = 0$ such that the resulting function 
is continuous (maybe except at the edges and the vertex of the quarter-plane $x_3 =0$, 
$x_{1,2} > 0$). The regularisation is as follows. Let us assume that $x_{1,2} > 0$ and consider the integral (\ref{eq:physicalfieldintegralxi}) for $x_3 > 0$. Change the 
integration contour so that the integral can be rewritten
\begin{eqnarray}
  u (\tmmathbf{x}, x_3) & = & - \frac{i}{4 \pi^2} 
  \int_\gamma
  \int_\gamma
    \tilde{K} (\xi_1, \xi_2) \tilde{W} (\xi_1, \xi_2) 
    e^{i x_3 \tilde{K}^{-1}   (\xi_1, \xi_2)} 
    e^{- i (\xi_1 x_1 + \xi_2 x_2)} \mathd \xi_1  \mathd
  \xi_2,  \label{eq:physicalfieldintegralxi_1}
\end{eqnarray}
where the contour $\gamma$ is shown in figure~\ref{fig:deform}.
Note that the tails of $\gamma$ are slightly bent into the lower half-plane. 
Such a deformation of the contour can be performed first for $\xi_1$ and then for
$\xi_2$. Cauchy's theorem and the points \ref{item:SFP1}-\ref{item:SFP2} of the simplified functional problem \textbf{SFP} guarantee that the value of the integral is not changing during these deformations. 
The integral (\ref{eq:physicalfieldintegralxi_1}) converges exponentially 
for $x_3 \ge 0$, thus providing the continuity. 
For other signs of $x_{1,2}$ one should bend the tails of the contours (in the upper or lower half-plane) appropriately.  

\begin{figure}[h]
\centering{\includegraphics[width=0.35\textwidth]{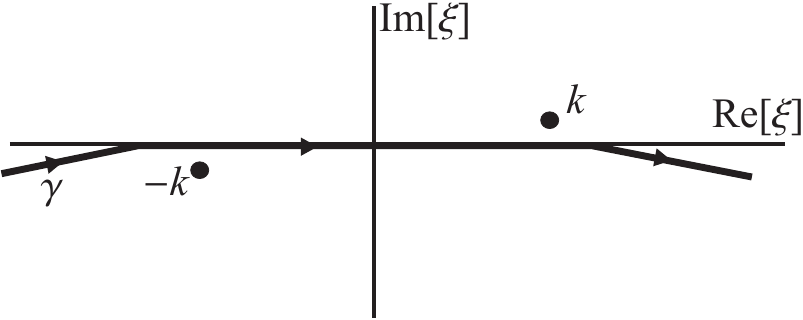}}
   \caption{Deformation of the integration contour for regularisation 
   of (\ref{eq:physicalfieldintegralxi})}
\label{fig:deform}
\end{figure}

%%%%%%%%%%%%%%%%%%%%%%%%%%%%%%%%%%%%%%%%%%%%%%%%%%%%%%%%%%%%%%%%%%%%%
%%%%%%%%%%%%%%%%%%%%%%%%%%%%%%%%%%%%%%%%%%%%%%%%%%%%%%%%%%%%%%%%%%%%%

\section{``Bridge and arrow'' bypass symbols}\label{app:appA}

\subsection{Motivation}

Consider an integral of the form 
%Generally speaking, the integral (\ref{eq:BA01}) takes
%the form 
\begin{equation}
  \int_{\boldsymbol{\Gamma}} \mathfrak{f} (\alpha_1, \alpha_2)  \hspace{0.17em} d \alpha_1 \wedge
  d \alpha_2 \label{eq0101}
\end{equation}
where $\boldsymbol{\Gamma}$ is a smooth (or a piece-wise smooth) oriented manifold of 
real dimension~2.
\RED{For the purpose of this appendix, we adopt here the Cartan formalism of differential forms, external product, and integration over manifolds.}
\RED{The manifold $\boldsymbol{\Gamma}$ 
is close to the plane of real $\alpha_1$ and $\alpha_2$
in the sense explained in the next paragraph.}
The integrand $\mathfrak{f} (\alpha_1, \alpha_2)$ is assumed to be an analytic
function of the arguments everywhere except at its \tmtextit{singularities}
$\sigma_j$,
i.e.\ the zero sets of some analytic functions $g_j$:
\begin{equation}
  \sigma_j = \left\{ (\alpha_1, \alpha_2) \in \mathbbm{C}^2 \text{ such that }
  g_j (\alpha_1, \alpha_2) = 0 \right\} . \label{eq0102}
\end{equation}
These sets are also manifolds of real dimension~2, and they can represent
either polar sets or branch sets of~$\mathfrak{f}$. 
\RED{
The functions $g_j(\alpha_1 , \alpha_2)$ are assumed to be ``real'' in the following 
sense: the partial derivatives $\ptl_{\alpha_1} g_j$ and $\ptl_{\alpha_2} g_j$
are real if $\alpha_1$ and $\alpha_2$ are real. 
The simplest (and quite common) case of a ``real'' function is when  $g_j (\alpha_1, \alpha_2)$ is real for real $(\alpha_1, \alpha_2)$.  
``Real'' functions 
possess an important topological property. The intersections of the $\sigma_j$ with
the real plane $(\alpha_1 , \alpha_2)$,
denoted by $\sigma_j'$ (and sometimes called the real trace of $\sigma_j$)
are sets of dimension~1 (i.e.\
a line or several lines), while in the general case such intersection 
is a set of points. We remind that, generally, an intersection of two 
manifolds of real dimension~2 in the space of real dimension~4 
is a manifold of real dimension~0.}

$\boldsymbol{\Gamma}$ can be
parametrized by the values \tmtextrm{Re}$[\alpha_1]$ and
\tmtextrm{Re}$[\alpha_2]$ by equations of the type
\begin{equation}
  \mathrm{Im} [\alpha_1] = \eta (\mathrm{Re} [\alpha_1], \mathrm{Re}
  [\alpha_2]), \qquad \mathrm{Im} [\alpha_2] = \zeta (\mathrm{Re} [\alpha_1],
  \mathrm{Re} [\alpha_2]) \label{eq0104}
\end{equation}
where $\eta$ and $\zeta$ are {\tmem{small}} real continuous functions equal to
zero everywhere except in the narrow vicinity of the branch sets $\sigma_j$.
\RED{These functions are chosen in such a way that 
$\boldsymbol{\Gamma}$ and 
$\sigma_j$ do not intersect: }
\[ 
\boldsymbol{\Gamma} \cap \sigma_j = \emptyset \text{ for all $j$}.
\] 
\RED{This is necessary for the integral (\ref{eq0101}) to be well-defined.} 

The detailed shape of the functions $\eta$ and $\zeta$ is not important due
to the generalisation of the Cauchy theorem  in
$\mathbbm{C}^2$:
the surface of integration can be deformed provided the singularities are not crossed
 (see e.g.{\cite{Shabat2}}). However, it is
important, whether $\boldsymbol{\Gamma}$ passes {\tmem{above}} or {\tmem{below}} the singularity
sets $\sigma_j$. Of course, what is meant by {\tmem{above}} and {\tmem{below}}
is not clear in $\mathbbm{C}^2$. The aim of this appendix is to introduce
these concepts precisely.

We are here trying to build the $\mathbbm{C}^2$ equivalent of a situation
that is very common in $\mathbbm{C}$ in diffraction theory. Often when for
example taking an inverse Fourier transform, one has to carefully make sure
that the singularities of the integrand, located on the real axis, are not hit
by the contour of integration. In order to do so, one should {\tmem{indent}}
the contour either above or below the singularities (poles or branch points).
The choice of indentation (above or below) is made based on physical
considerations (typically linked to the radiation condition).

%%%%%%%%%%%

Indeed, the whole technique is developed for 
the integral representation (\ref{eq:physicalfieldintegralxi}) of the wave field:
%Use the differential form notations.
%Thus, the resulting integral now reads as 
\begin{align}
u(\tmmathbf{x} , x_3) = 
-\frac{i k}{4\pi^2}
\int_{\Gamma_\alpha \times \Gamma_\alpha}
 & \exp\left\{ i k \left(- x_1 \sin \alpha_1 -x_2 \sin \alpha_2 + 
x_3 \sqrt{1 - \sin^2 \alpha_1 - \sin^2 \alpha^2}\right)  \right\}
\times \nonumber \\
& \frac{\hat W (\alpha_1 , \alpha_2) \cos \alpha_1 \, \cos \alpha_2}{
\sqrt{1 - \sin^2 \alpha_1 - \sin^2 \alpha_\RED{2}}
}
\, d\alpha_1 \, \wedge \, d \alpha_2
\label{eq:BA01}
\end{align}
where $\hat W$ is defined in Section~\ref{sec:formulationangularalpha} and the contour of integration is a product of two samples of $\Gamma_\alpha$, also defined in Section~\ref{sec:formulationangularalpha} and illustrated in figure~\ref{fig:fromxitoalphacontour}, right. 
The contour $\Gamma_{\alpha}$ is understood to be oriented from left to right.
Our aim is to explore the possibility of deformation of the 
surface of integration.  

In \RED{the case studied in the paper}, the branch sets are
complexified lines that can be generically defined by
\begin{equation}
  g_j (\alpha_1, \alpha_2) \equiv \mathsf{a}_j \alpha_1 + \mathsf{b}_j \alpha_2 + \mathsf{c}_j .
  \label{eq0103}
\end{equation}
\RED{The ``reality'' condition is provided by the 
requirements for the 
coefficients $\mathsf{a}_j$, $\mathsf{b}_j$ and $\mathsf{c_j}$
to be real.}

\RED{Indeed, the integral (\ref{eq:BA01}) fits the form (\ref{eq0101}).
Some 
other examples of integrals of the type (\ref{eq0101}), for which one needs to define 
the mutual location of the singularities and the integration surface,
are given, in \cite{ice}, for instance.  
}

%%%%%%%%%%%%%%%%%%%%%%%%%%%%%%%%%%%%%%%%%%%%%%%%%%%%%%%%%%%%

\subsection{The bridge and arrow notation}

We will now introduce a diagrammatic notation, the ``bridge and arrow''
notation, in order to precisely illustrate how such manifold $\boldsymbol{\Gamma}$ is
located with respect to a particular singular set~$\sigma_j$. 

\RED{Consider a neighbourhood $\mathcal{U}$ of some real point 
$\tmmathbf{\alpha}^{\star} = (\alpha_1^{\star},
\alpha_2^{\star})$ 
belonging to $\sigma_j$ (and hence belonging to $\sigma_j'$ since it is real), but not belonging to any other singularity. 
Let us define 
\[
\mathsf{a} = \ptl_{\alpha_1} g_j(\alpha_1^{\star}, \alpha_2^{\star}),
\qquad 
\mathsf{b} = \ptl_{\alpha_2} g_j(\alpha_1^{\star}, \alpha_2^{\star}),
%\qquad   
%\mathsf{c} = g_j(\alpha_1^{\star}, \alpha_2^{\star}).
\]
and assume that at least one of the two values $\mathsf{a}$ and $\mathsf{b}$ is non-zero, 
i.e.\ that $\tmmathbf{\alpha}^{\star}$ is a regular point of $\sigma_j$.
Because of the ``reality'' condition of $g_j$, the values $\mathsf{a}$ and $\mathsf{b}$ 
are real. 
In the linear approximation, the set $\sigma_j$ can be approximately described as a linear function 
of the type (\ref{eq0103}) near $\tmmathbf{\alpha}^{\star}$. 
}

Introduce a local coordinate system $(z, t)$, with the origin at
the point $\tmmathbf{\alpha}^{\star}$
defined for arbitrary nonzero {\em real\/} constants $\beta'$, $\beta''$, and a  possibly zero 
{\em real} constant $\beta'''$ by
\RED{
\begin{equation}
  z = \beta' g_j(\alpha_1 , \alpha_2),
  \qquad 
  t = \beta'' (
  - \mathsf{b} (\alpha_1 - \alpha_1^{\star} )
  +
  \mathsf{a}   (\alpha_2 - \alpha_2^{\star} )) + \beta''' z.  
  \label{eq0106} 
\end{equation}
The singularity $\sigma_j$ corresponds to $z = 0$, thus, $z$ is the 
``transversal'' coordinate, and $t$ is the ``tangential'' one.  
}

\RED{
The coordinate change $(\alpha_1 , \alpha_2) \to (z, t)$ is, indeed, 
a (local) transform between two pairs of {\em complex\/} coordinates.   
At the same time, at least in the linear approximation near $\tmmathbf{\alpha}^{\star}$, it can be considered as a {\em real\/} coordinate transform in the 
real plane $(\alpha_1, \alpha_2)$. 
Thus, one can draw the coordinate lines for real $(z , t)$ and the basis vectors 
${\tmmathbf{e}}_z$, ${\tmmathbf{e}}_t$ in the plane.       
The vector ${\tmmathbf{e}}_t$ is directed along $\sigma_j'$, and the vector     
${\tmmathbf{e}}_z$ is transversal (not necessarily normal) to $\sigma_j'$.
Their directions depend on the function $g_j$ and the coefficients
$\beta'$, $\beta''$, $\beta'''$.
Both vectors are real.  
%We are interested in the vector ${\tmmathbf{e}}_z$. 
These vectors are shown in figure~\ref{addit},~left.
}

\begin{figure}[h]
\centering  \includegraphics[width=.7\textwidth]{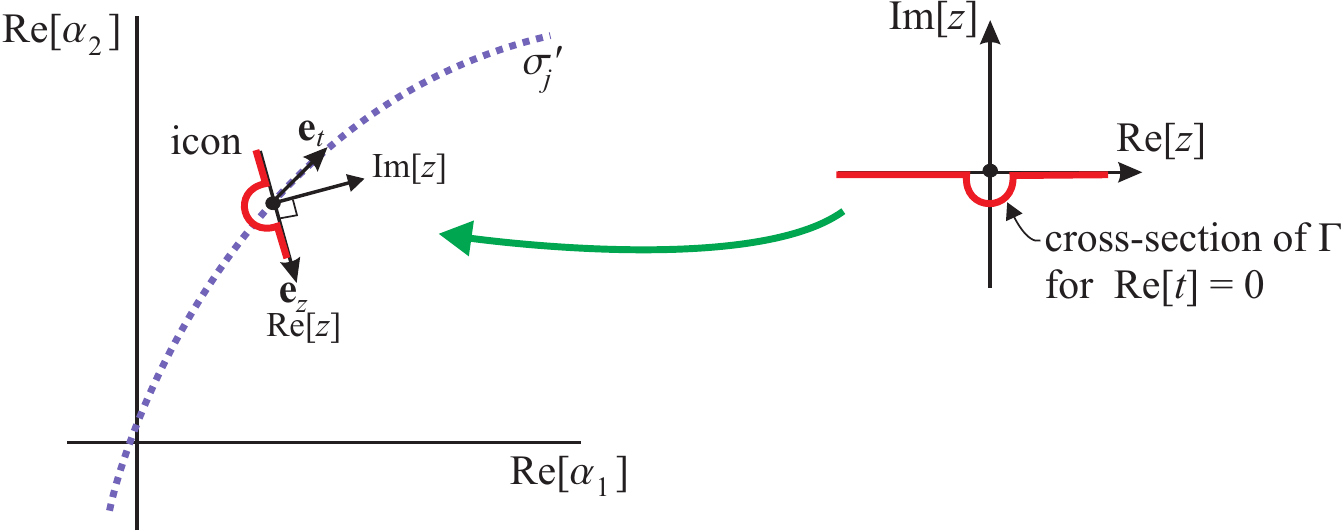}  
\caption{Bridge and arrow notation. }
\label{addit}
\end{figure}

\RED{
%Rewrite the integrand of (\ref{eq0101}) as 
%\[
%\mathfrak{f}'(z , t) \frac{\ptl(\alpha_1 , \alpha_2)}{\ptl (z, t)}
%dz \, \wedge dt.
%\]
%The Jacobian is analytic in a small neighborhood  $\mathcal{U}$ 
%of $\tmmathbf{\alpha}^{\star}$
%since $\tmmathbf{\alpha}^{\star}$ is a regular point of $g_j$. 
%The only singularity of the integrand near $\tmmathbf{\alpha}^{\star}$
%is the set $z =0$. 

Let us now fix the value ${\rm Re}[t] = 0$, and consider 
a set of points belonging to $\mathcal U \cap \boldsymbol{\Gamma}$
and ${\rm Re}[t] = 0$. This is a line that can be projected onto the 
$z$-coordinate complex plane.
To define the integral (\ref{eq0101}) correctly, 
this line should be close to the real axis of $z$, but 
pass above or below the point $z= 0$ to avoid an intersection between 
$\boldsymbol{\Gamma}$ and $\sigma_j$.    
This cross-section is shown in figure~\ref{addit},~right (the case of 
$\boldsymbol{\Gamma}$ passing below $z = 0$ is shown).

Now let us draw an {\em icon\/} of the right part of the figure on the left part of it. 
For this, use the vector ${\tmmathbf{e}}_z$ as the axis of ${\rm Re}[z]$.
Draw the axis of ${\rm Im}[z]$ such that 
the angle between the axes ${\rm Re}[z]$ and ${\rm Im}[z]$
is $\pi / 2$ anticlockwise. Then draw the same cross-section 
of $\boldsymbol{\Gamma}$ in the new coordinates in the left part of the figure. 
This is the ``bridge and arrow'' notation 
at the point $\tmmathbf{\alpha}^{\star}$. At the moment, this notation
depends on the function $g_j$ and on the coefficients 
$\beta'$, $\beta''$, $\beta'''$. Moreover, it is only currently defined at 
a single real point of $\sigma_j'$.
However, we will show below that this notation is very convenient in the sense that it is actually independent of the coefficients 
$\beta'$, $\beta''$, $\beta'''$, unchanged by a scaling of $g_j$, and can be rotated and translated freely along~$\sigma_j'$.  

}

%%%%%%%%%%%%%%%%%%%%%%%%%%%%%%%%%%%%%%%%%%%%%%%%

\RED{

\subsection{Usage rules for the bridge and arrow notations}

Here we list the properties of the bridge and arrow notations. 
The proofs are elementary and are based mainly on the continuity 
argument. 

\vskip 6pt

{\bf a. Consistency of the bridge and arrow notation. }
The properties {\bf a1}, {\bf a2} and {\bf a3} listed below are related to a single point 
$\tmmathbf{\alpha}^{\star}$. In all formulations
we assume that the function  $\mathfrak{f}$, its singularities 
$\sigma_j$ and the surface of integration $\boldsymbol{\Gamma}$
are known.  

\vskip 6 pt
\noindent
{\bf Property a1.} The bridge and arrow notation is defined uniquely if the 
curve $\sigma_1'$ and the 
direction ${\tmmathbf{e}}_z$ are known. This means that the bridge 
does not change if, say, the function $g_j$ is multiplied by a coefficient, 
or the coefficients $\beta'$, $\beta''$, $\beta'''$ are changed, while the 
direction of ${\tmmathbf{e}}_z$ remains the same. 

This statement means that the 
bridge relation is formally consistent. One can draw the axes 
${\rm Re}[z]$ and ${\rm Im}[z]$ (they are the ``arrows'') and the 
``bridge'' denoting the bypass without specifying the formula for 
$g_j$ and the coefficients $\beta'$, $\beta''$, $\beta'''$.    

\vskip 6 pt
\noindent
{\bf Property a2.} If the direction of ${\tmmathbf{e}}_z$
is changed continuously (rotated), the bridge symbol is rotated in the
same way, providing that ${\tmmathbf{e}}_z$ does not become tangential
to $\sigma_j$ during this rotation. 
Such a rotation is shown in figure~\ref{addit2}. 

\begin{figure}[h]
\centering  \includegraphics[width=0.45\textwidth]{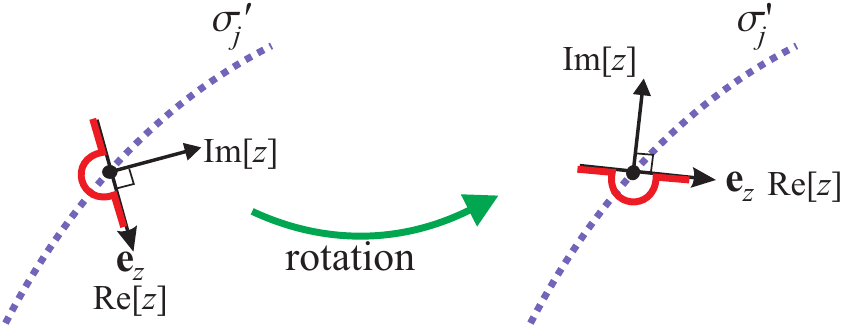}  
\caption{Rotation of the bridge and arrow symbol. }
\label{addit2}
\end{figure}

This property means that it is not necessary to establish the bridge notation for 
the whole continuum of directions of  ${\tmmathbf{e}}_z$. Instead, 
one can take two arbitrary directions (one pointing to the right of $\sigma_j$, and another pointing to the left, say) and define the bridge notation for them.
The bridges for all other directions of ${\tmmathbf{e}}_z$
can be obtained by rotations. 

\vskip 6 pt
\noindent
{\bf Property a3.} The bridge notation for two opposite directions 
${\tmmathbf{e}}_z$ is the same 
(see figure~\ref{addit3}). One can see that if, say, a contour passes below zero
in the variable $z$, then it passes above zero in the variable~$-z$.
Thus, the ``arrow'' changes its direction, while the image of the ``bridge''
remains the same. 
}

\begin{figure}[h]
\centering  \includegraphics[width=0.45\textwidth]{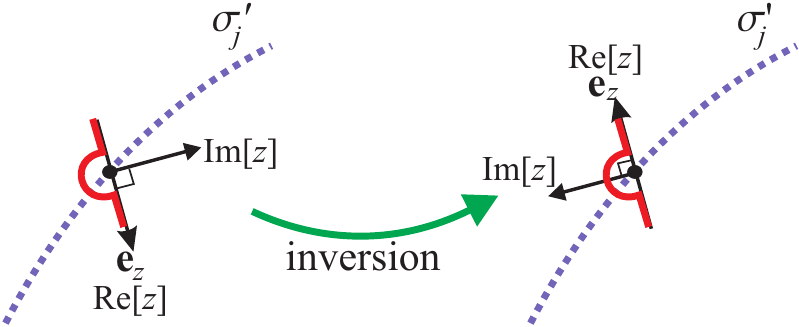}  
\caption{Inversion of the bridge and arrow symbol. }
\label{addit3}
\end{figure}

\RED{
This statement has two following corollaries. First, for each point  
$\tmmathbf{\alpha}^{\star}$ one can build the bridge for a single 
direction ${\tmmathbf{e}}_z$ (instead of two). The bridges for all other
transversal directions become known. Second, one can omit arrows in the 
bridge and arrow notation. Thus, it becomes a ``bridge'' notation,
and a single bridge symbol is necessary to describe the mutual 
position of the integration surface $\boldsymbol{\Gamma}$
and the singularity $\sigma_j$ at each particular point 
of $\sigma_j$.
 
Still, in the paper we continue to use the arrows to remind the reader that 
the bridge and arrow notation is just an icon of a cross-section of 
 $\boldsymbol{\Gamma}$ by a transversal complex variable.     
}

\vskip 6 pt
\RED{
{\bf b. Continuity property}
Consider a continuous fragment of 
$\sigma_j'$ not intersecting any other singularities.  
The bridge symbol changes continuously as $\tmmathbf{\alpha}^{\star}$
moves along this fragment and ${\tmmathbf{e}}_z$ cannot become 
tangential to $\sigma_j'$ during this change.
This situation is shown in figure~\ref{addit4}.
}   

\begin{figure}[h]
\centering  \includegraphics[width=0.2\textwidth]{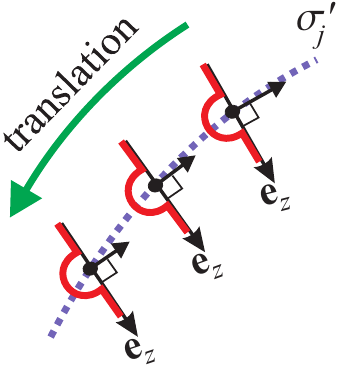}  
\caption{Carrying the bridge symbol along a fragment of $\sigma_j$.}
\label{addit4}
\end{figure}

\RED{
The corollary of this statement is that it is not necessary to find the 
bridge symbol for each point of $\sigma_j'$. Instead, one should 
find the bridge symbol for a single point of a fragment of  $\sigma_j'$ and extend 
it by continuity to all other points.  
}

\vskip 6pt
\RED{   
{\bf c. Intersection properties} 
 
\vskip 6pt
\noindent
{\bf Property c1. } Let two singularities, $\sigma_j$ and $\sigma_m$ have an intersection 
at a real point $\tmmathbf{\alpha}^{\star}$, and let the intersection be transversal. This means that 
the gradient vectors $(\ptl_{\alpha_1} g_j(\tmmathbf{\alpha}^{\star}) , \ptl_{\alpha_2} g_j(\tmmathbf{\alpha}^{\star}))$ 
and
$(\ptl_{\alpha_1} g_{m}(\tmmathbf{\alpha}^{\star}) , \ptl_{\alpha_2} g_{m}(\tmmathbf{\alpha}^{\star}))$ 
are not proportional. Then the bridge symbols for $\sigma_j'$ and 
$\sigma_{m}'$ can be chosen independently. The bridge symbol for $\sigma_j'$
can be carried continuously through the crossing point    
$\tmmathbf{\alpha}^{\star}$. The same is valid for the bridge symbol 
for the bridge symbol for $\sigma_{m}'$ (see figure~\ref{addit5}).
}

\begin{figure}[h]
\centering  \includegraphics[width=0.25\textwidth]{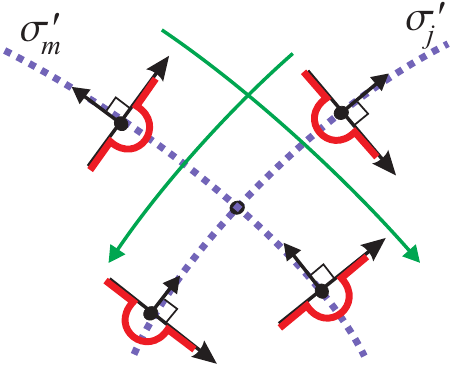}  
\caption{A transversal crossing of singularities.}
\label{addit5}
\end{figure}

\RED{
\vskip 6pt
\noindent
{\bf Property c2. } Let two singularities, $\sigma_j$ and $\sigma_m$ have an intersection 
at a real point $\tmmathbf{\alpha}^{\star}$, and let the intersection be quadratic. This means that the curves $\sigma'_j$ and $\sigma'_m$ are touching tangentially in the 
real plane, and the touch has order~2. Then the bridge notation for the 
the singularities $\sigma_j'$ and $\sigma_m'$ are 
not independent
and determined by the 
bridge notation at the touch point. The bridge notation is translated continuously 
from the touch point along each of the singularities. 
}   
   
\begin{figure}[h]
\centering  \includegraphics[width=0.35\textwidth]{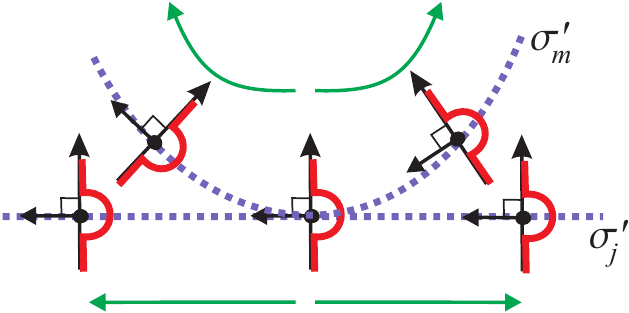}  
\caption{A touch of singularities.}
\label{addit6}
\end{figure}

\RED{
The proof of this fact is non-trivial and is based on the possibility to introduce the tangential / transversal coordinates for $\sigma_j'$ and $\sigma_m'$
in such a way that the bypass in the transversal plane 
at the touch point is common for both 
singularities. 

\vskip 6pt
\noindent
{\bf Remark~about~Property~c2.} 
Topologically, the situation of touching singularities is much more complicated
than that of transversally crossing singularities. It may happen that not all 
four lines stemming from the touch point correspond to a singular behaviour of the function. 
An example is the function 
\begin{equation}
\mathfrak{f} (\alpha_1 , \alpha_2)= \sqrt{\sqrt{\alpha_2} - \alpha_1}. 
\label{ead7}
\end{equation}
This function has two singularities:
\[
\sigma_1 : \ \alpha_2 = 0, 
\qquad \qquad
\sigma_2 : \ \alpha_2 - \alpha_1^2 = 0. 
\] 
For a certain sheet of $\mathfrak{f}$, the diagram of singularities looks as shown in figure~\ref{addit7}.       
}

\begin{figure}[h]
\centering  \includegraphics{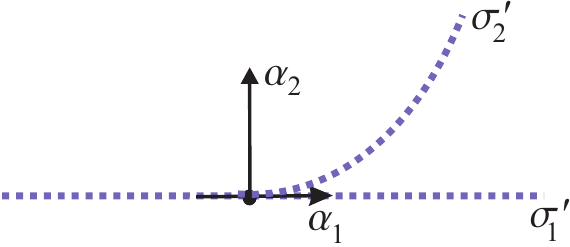}  
\caption{Singularities of the function (\ref{ead7})}
\label{addit7}
\end{figure}

\RED{
{\bf d. Topological properties.} 

\noindent
{\bf Property d1.} Let $\sigma_j'$ be a
closed smooth curve.
An example is a circle defined by the function 
\[
g_j (\alpha_1 , \alpha_2) = \alpha_1^2 + \alpha_2^2  - d^2, 
\]
where $d$ is a real parameter. 
Let the surface $\boldsymbol{\Gamma}$ be close to the real plane
in the sense introduced above. 
Let be $\boldsymbol{\Gamma} \cap \sigma_j = \emptyset$.
Then, 
the mutual position 
of $\boldsymbol{\Gamma}$ and $\sigma_j$ is defined by a single 
bridge symbol (see figure~\ref{addit8}, left or right, for two possible 
values of this symbol). 
Moreover, there exists a real point inside $\sigma_j'$ through which 
$\boldsymbol{\Gamma}$ passes. 
}

\begin{figure}[h]
\centering  \includegraphics[width=0.6\textwidth]{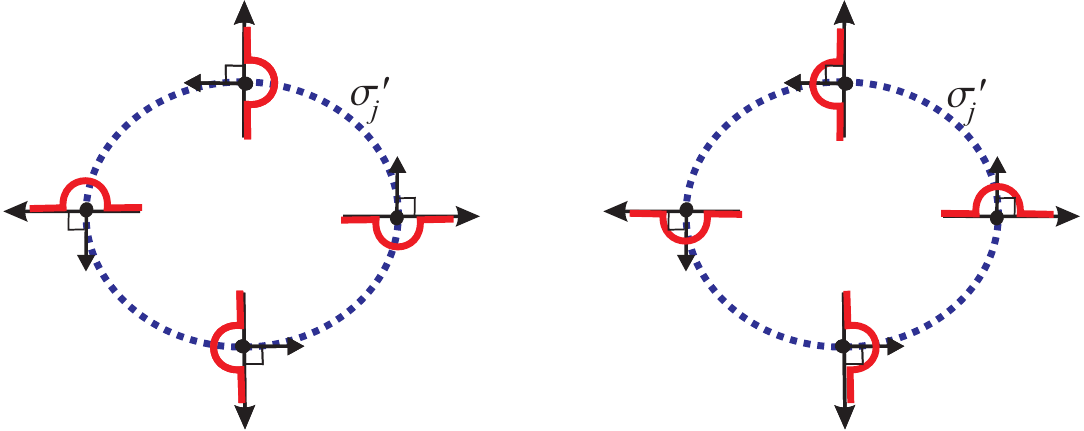}  
\caption{Two possible mutual positions of $\boldsymbol{\Gamma}$ and $\sigma'_j$}
\label{addit8}
\end{figure}

\RED{The proof is based on the study of the vector field $\boldsymbol{v}=(\text{Im}[\alpha_1],\text{Im}[\alpha_2])=(\eta,\zeta)$ over the circle $\sigma_j'$. By definition of $z$, we can show that $\text{Im}[z]\approx  \beta'(\mathsf{a}^2+\mathsf{b}^2) \boldsymbol{v}\cdot\boldsymbol{n}$, where $\boldsymbol{n}=\nabla_{\boldsymbol{\alpha}}g_j/|\nabla_{\boldsymbol{\alpha}}g_j|$ is the unit normal to the circle in the $(\text{Re}[\alpha_1],\text{Re}[\alpha_2])$ plane. Since $\text{Im}[z]$ is continuous and never zero on the circle, $\boldsymbol{v}\cdot\boldsymbol{n}$ should hence have a constant sign and $\boldsymbol{v}$ has to encircle the origin (i.e.\ its index about $\sigma_j'$ is not zero). By continuity, there should exist a point inside $\sigma_j'$ at which the vector field is equal to zero, see e.g.\ Th.~6.23 in \cite{DMeiss2017}, and hence $\boldsymbol{\Gamma}$ passes through this real point.}

\RED{
%The proof is based on the hodograph of the vector 
%$({\rm Im}[\alpha_1],{\rm Im}[\alpha_2])$ over the contour~$\sigma_j'$. 
%One can see that this hodograph encircles the origin. By continuity, there 
%should exist a point inside the contour at which the vector is equal to zero. 
The statement is equivalent to the fact that 
the linking number of $\boldsymbol{\Gamma}$ and $\sigma_j'$ is equal to~$\pm 1$.
The surface $\boldsymbol{\Gamma}$ cannot be 
deformed in such a way that it is
carried far enough from 
the real plane. Note that such a deformation can be necessary 
for an asymptotic evaluation of the integral (\ref{eq0101}) (see \cite{ice}).    
}

\RED{
\noindent
{\bf Property d2.} Let us consider a (possibly curvilinear) polygon formed by several 
traces of singularities~$\sigma_j'$ and assume that the surface 
$\boldsymbol{\Gamma}$ does not intersect any of the singularities $\sigma_j$. 
Introduce the concept of matching and unmatching bypasses across the 
boundaries of the polygon. For definiteness, direct all the vectors 
${\tmmathbf{e}}_z$ {\em outwards\/} the polygon. The bypass symbols are 
matching if all brigdes are of ``above'' type, or all bridges 
are of ``below'' type. Otherwise, the bypass symbols are unmatching. 
Indeed, there are two possible matching sets of bypasses. 
Examples of matching and unmatching bypasses are shown in 
figure~\ref{addit9}
for the case of the polygon being a triangle. 
}

\begin{figure}[h]
\centering  \includegraphics[width=0.6\textwidth]{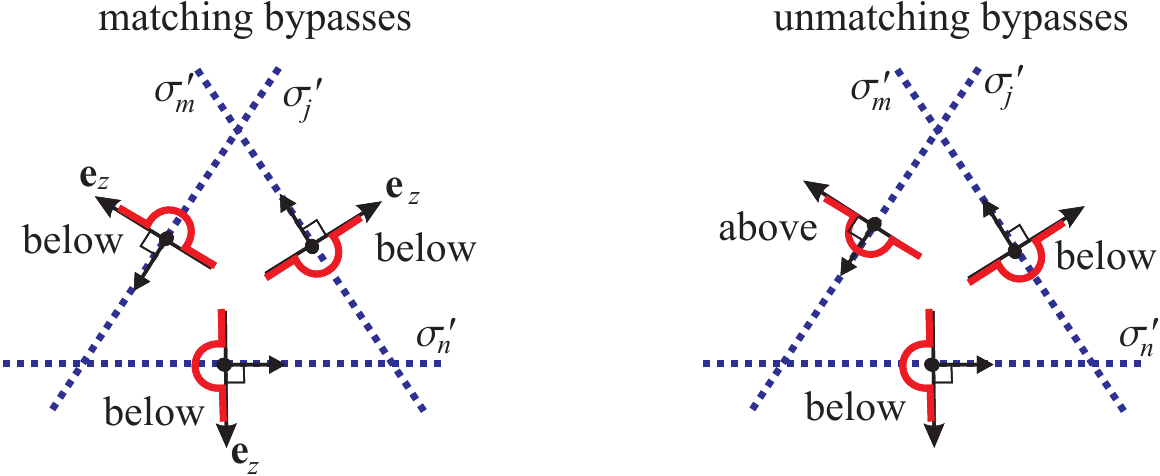}  
\caption{Matching and unmatching bypasses for a triangle}
\label{addit9}
\end{figure}

\RED{
The statement is as follows. If the bypasses are matching, there exists a 
(real) point inside the polygon, through which   
$\boldsymbol{\Gamma}$ passes. 

The proof is similar to that of Property~{\bf d1}.
}

%%%%%%%%%%%%%%%%%%%%%%%%%%%%%%%%%%%%%%%%%

\RED{
\subsection{Bridge and arrow notation for contours}

In the current paper we use the bridge and arrow notations 
not only for surfaces of integration, but for contours as well. Let 
$\gamma$ be a contour (a smooth oriented 1D manifold) located near the 
real plane. Namely, let the contour be parametrized by a real variable 
$\tau \in [0,1]$:
$(\alpha_1 (\tau) , \alpha_2 (\tau))$. Let the real gradient vector
$
\left( 
\ptl_\tau {\rm Re}[\alpha_1],
\ptl_\tau {\rm Re}[\alpha_2]
\right)
$
be non-zero for all $\tau$. We assume that  the vector 
$
({\rm Im}[\alpha_1 (\tau)],{\rm Im}[\alpha_2 (\tau)])
$ 
is small , and the vector   
$
\left( 
\ptl_\tau {\rm Im}[\alpha_1],
\ptl_\tau {\rm Im}[\alpha_2]
\right)
$
is finite.

Let $\sigma_j$ be defined as above by (\ref{eq0102}).
The intersection of $\gamma$ and $\sigma_j$ should be empty. Our aim is to describe 
graphically the mutual position of $\gamma$ and $\sigma_j$ 
near some real point $\tmmathbf{\alpha}^{\star} \in \sigma_j'$.

Introduce the variable $z$ by the first formula of (\ref{eq0106}). 
Draw the projection of the contour $\gamma$ on the $z$-plane near 
$\tmmathbf{\alpha}^{\star} \in \sigma_j'$, i.~e.\ the curve 
$
\beta' g_j (\alpha_1 (\tau) , \alpha_1 (\tau))
$
for $\tau$ belonging to some segment. This curve can be used as 
the red curve in the right side of figure~\ref{addit}.
Thus, one can define the bridge notation for $\gamma$ and $\sigma_j$, 
again, by drawing an icon of the $z$-plane near the point 
$\tmmathbf{\alpha}^{\star} \in \sigma_j'$.
Note that the vector ${\tmmathbf{e}}_z$ in this case is proportional to 
$
\left( 
\ptl_\tau {\rm Re}[\alpha_1],
\ptl_\tau {\rm Re}[\alpha_2]
\right) 
$.

Let the contour $\gamma$ be continuously deformed without crossing some singularities. 
One can see that analogs of the properties {\bf a1}, {\bf a2}, {\bf a3}, {\bf b}, and {\bf c1} remain valid in this case.
}

\subsection{Bypass symbols for the considered problem}

More specifically, in Section \ref{sec:formulationangularalpha}, because of
the integral (\ref{eq:BA01}), we are interested in the
behaviour of the function $\hat{W} (\alpha_1, \alpha_2)$ in the close vicinity
of the real square $\boldsymbol{\Gamma'} = \left[ - \frac{\pi}{2}, \frac{\pi}{2}
\right] \times \left[ - \frac{\pi}{2}, \frac{\pi}{2} \right]$. On this real
set we have shown in Proposition \ref{prop:sqaureroots} that $\hat{W}
(\alpha_1, \alpha_2)$:
\begin{enumerate}[label=\roman*)]
  \item is singular on the real trace of the complexified line $\sigma_{- 1}^+
  = \left\{ \tmmathbf{\alpha} \in \mathbbm{C}^2, \alpha_1 + \alpha_2 = -
  \frac{\pi}{2} \right\}$, where it changes its sign (it behaves like a square
  root);
  
  \item is regular on the real trace of the complexified lines $\sigma_0^- =
  \left\{ \tmmathbf{\alpha} \in \mathbbm{C}^2, \alpha_1 - \alpha_2 =
  \frac{\pi}{2} \right\}$, $\sigma_0^+ = \left\{ \tmmathbf{\alpha} \in
  \mathbbm{C}^2, \alpha_1 + \alpha_2 = \frac{\pi}{2} \right\}$ and $\sigma_{-
  1}^- = \left\{ \tmmathbf{\alpha} \in \mathbbm{C}^2, \alpha_1 - \alpha_2 = -
  \frac{\pi}{2} \right\}$;
  
  \item has no other singularities in the vicinity of $\boldsymbol{\Gamma'}$. 
\end{enumerate}
The problematic part of the integrand in (\ref{eq:BA01})
involves the function
\begin{eqnarray*}
  \hat{K} (\alpha_1, \alpha_2) \hat{W} (\alpha_1, \alpha_2) & = &
  \frac{\hat{W} (\alpha_1, \alpha_2)}{k \sqrt{1 - \sin^2 (\alpha_1) - \sin^2
  (\alpha_2)}} \cdot
\end{eqnarray*}
Since this function has some singularities on $\boldsymbol{\Gamma'}$, it is
important to analyse the bypass symbols associated to the manifold
$\boldsymbol{\Gamma}$ when it is very close to the square $\boldsymbol{\Gamma'}$. More
precisely, we should specify the bypass symbols using the bridge and arrow
notation for the complexified lines $\sigma_0^{\pm}$ and $\sigma_{- 1}^{\pm}$.
Using first a wavenumber $k$ with a small positive imaginary part, studying the position of resulting singularities close to $\boldsymbol{\Gamma'}$ and letting Im$[k]\rightarrow 0$, one can show that the symbols described in figure~\ref{fig04} should be used\RED{; this is a matching bypass}.

\begin{figure}[h]
\centering  \includegraphics[width=0.45\textwidth]{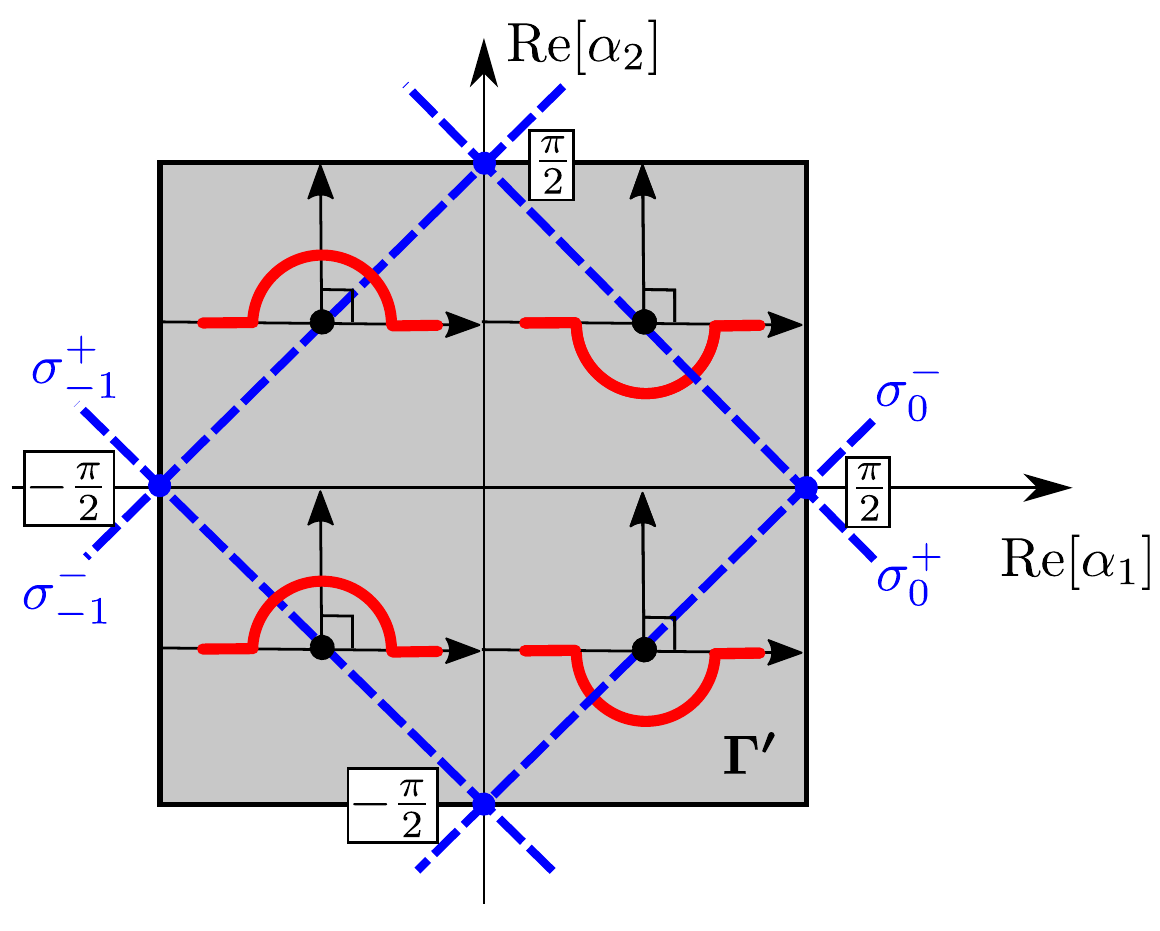}\quad\includegraphics[width=0.45\textwidth]{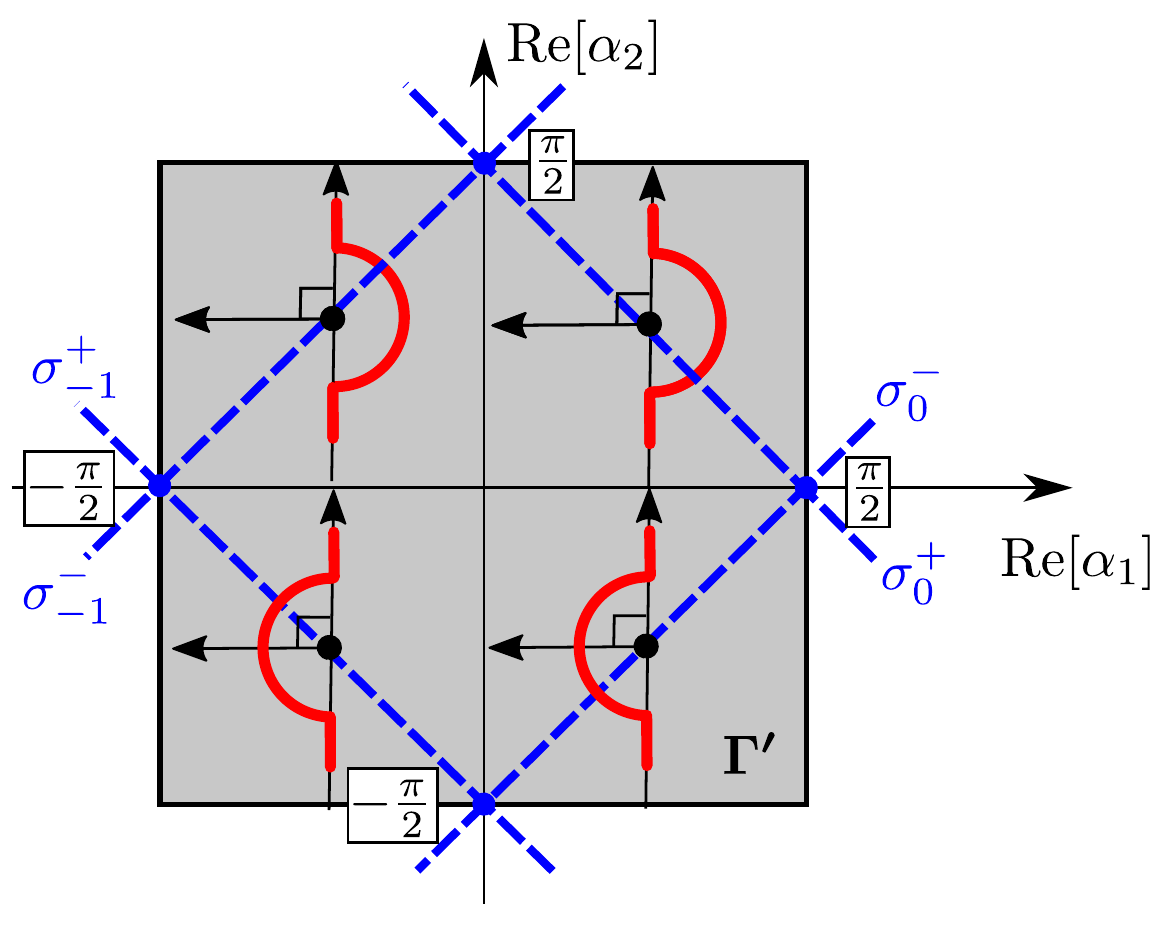}
  \caption{Bypass symbols associated with $\boldsymbol{\Gamma}$ in the vicinity of
  the real square $\boldsymbol{\Gamma'}$}
\label{fig04}
\end{figure}

%%%%%%%%%%%%%%%%%%%%%%%%%%%%%%%%%%%%%%%%%%%%%%%%%%%%%%%%
%%%%%%%%%%%%%%%%%%%%%%%%%%%%%%%%%%%%%%%%%%%%%%%%%%%%%%%%
\section{Proofs of theorems related to ODEs}
\label{app:stencil2ode}

\subsection{Proof of Theorem~\ref{th:stencil2ODE}}\label{app:C1}

Since $T$ has a square root behaviour at $\beta = 5 \pi / 2$, it is clear that this point is a branch point of $T$. 
Let us start by changing the path connecting the reference points in figure~\ref{fig:pathbetaplane} as it is shown in figure~\ref{fig0203}.  Since the branch point $\beta = 5 \pi/2$ is bypassed in a different way now, one should change the sign of the corresponding term of the stencil equation. The new version of the stencil equation reads:   
  \begin{equation}
    T (\beta + 4 \pi)   = T (\beta) - \lambda T (\beta + 2 \pi). 
    \label{eq7034}
  \end{equation}
Initially, we assume that this equation is valid for $\tmop{Re} [\beta] = - \pi$, and then we continue this equation onto the whole complex plane. 

 \begin{figure}[h]
\centering    \includegraphics[width=0.4\textwidth]{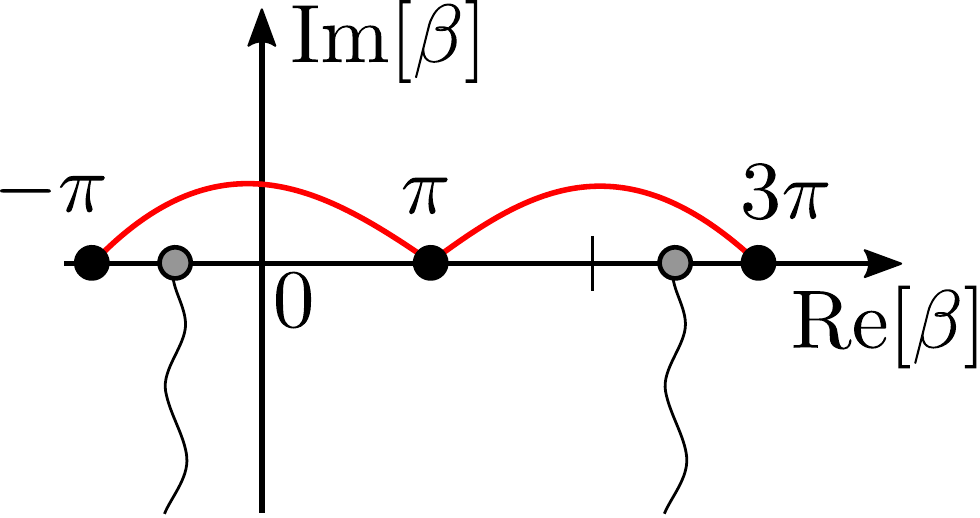}
    \caption{New connections of the points for the stencil equation
    (\ref{eq7034})}
\label{fig0203}
  \end{figure}

The aim of the change of path and, respectively, of the form of the 
stencil equation, is to connect the points $(\beta, \beta + 2\pi, 
\beta+4\pi)$ in the simplest way at least in the upper half-plane.  Namely, 
the function $T(\beta)$ is holomorphic in the half strip $-\pi<{\rm Re}[\beta]<3\pi$,
${\rm Im}[\beta] > 0$, and the points $(\beta, \beta + 2\pi, 
\beta+4\pi)$ for ${\rm Re}[\beta] = -\pi$ are just points obtained from one another by simple translations.    
  
Let us now define the function $R$ by   
\begin{equation}
  R (\beta) = T (\beta + 2 \pi).
  \label{eq7034aa}
\end{equation}  
Here we will attempt to find some functions $g(\beta)$, $f(\beta)$ such that the two equations    
\begin{equation}
\left[
\frac{d^2 }{d \beta^2} + f(\beta) \frac{d }{d \beta} + g(\beta)
\right]
 T(\beta) = 0,
 \qquad
\left[
\frac{d^2 }{d \beta^2} + f(\beta) \frac{d }{d \beta} + g(\beta)
\right]
 R(\beta) = 0
\label{eq7034aa1}
\end{equation}
are fulfilled. Note that the first equation is 
(\ref{eq7033}), thus, if the coefficients $f$ and $g$ posses the necessary properties, 
the statement of the theorem is proven. 
  
The coefficients $f$ and $g$ can be found formally by solving 
(\ref{eq7034aa1}) as a system of two linear algebraic equations: 
\begin{eqnarray}
    f (\beta) = - \frac{D_{0, 2}}{D_{0, 1}} & \tmop{and} & g (\beta) =
    \frac{D_{1, 2}}{D_{0, 1}} ,  \label{eq7034a}
 \end{eqnarray} 
where $D_{m,n}(\beta)$ are the generalised Wronsky determinants defined by
 \begin{eqnarray}
    D_{m, n} (\beta) & = & \left| \begin{array}{cc}
      T^{(m)} (\beta) & T^{(n)} (\beta)\\
      R^{(m)} (\beta) & R^{(n)} (\beta)
    \end{array} \right|,  \label{eq7035}
  \end{eqnarray}
  where $m$ and $n$ are non-negative integers and the subscript $^{(m)}$ say,
  corresponds to the $m^{\text{th}}$ derivative of a given function with respect to~$\beta$. 
  The determinants are defined in the strip $-\pi<{\rm Re}[\beta]<\pi$. 
  They are continuous on the edges of this strip, and are branching at the 
  points $\beta = \pm \pi / 2$. According to the properties of 
  $T$ and $R$ at $\beta = \pm \pi / 2$, each Wronsky determinant changes 
  its sign when a point $\beta = \pm \pi / 2$ is encircled.

Using the definition of $R$, is is possible to rewrite the stencil equation (\ref{eq7034}) in the vectorial form
  \begin{eqnarray}
    \left( \begin{array}{c}
      T (\beta + 2 \pi)\\
      R (\beta + 2 \pi)
    \end{array} \right) & = & \left( \begin{array}{cc}
      0 & 1\\
      1 & - \lambda
    \end{array} \right) \left( \begin{array}{c}
      T (\beta)\\
      R (\beta)
    \end{array} \right),  \label{eq7036}
  \end{eqnarray}
  for $\tmop{Re} [\beta] = - \pi$. By direct differentiation, the same relation is valid for any
  derivative of the vector~$\begin{psmallmatrix}T\\R\end{psmallmatrix}$, and hence, for any non-negative integers $m$ and
  $n$, we have
  \begin{eqnarray}
    D_{m, n} (\beta + 2 \pi) & = & \left| \begin{array}{cc}
      0 & 1\\
      1 & - \lambda
    \end{array} \right| D_{m, n} (\beta) = - D_{m, n} (\beta) . 
    \label{eq:interraphdetDmn}
  \end{eqnarray}
  This relation enables one to perform the analytic continuation of 
  $D_{m,n} (\beta)$ outside the strip~$-\pi < \beta < \pi$.

Because $(\beta+\pi/2)^{-1/2}T(\beta)$ and $(\beta-\pi/2)^{-1/2}R(\beta)$ are analytic on the strip $-\pi\leq\text{Re}[\beta]\leq\pi$, and because $\cos(\beta)$ has simple zeros at $\beta=\pm\pi/2$, one can show, by direct differentiation or otherwise, that the function
\begin{align*}
P_{m,n}(\beta)&=(\cos(\beta))^{\max(m,n)} D_{m,n}(\beta)/\sqrt{\cos(\beta)}
\end{align*}
is analytic on the strip and $P_{m,n}(\pm\pi/2)\neq 0$. Since $\cos(\beta)$ is $2\pi$-periodic, (\ref{eq:interraphdetDmn}) implies that $P_{m,n}(2\pi+\beta)=-P_{m,n}(\beta)$. Hence the function $\tilde{P}_{m,n}(\beta)=\sin(\beta/2)P_{m,n}(\beta)$ is analytic on the strip and satisfies $\tilde{P}_{m,n}(\beta+2\pi)=\tilde{P}_{m,n}(\beta)$.

The mapping $\mathfrak{r} = e^{i \beta}$ maps the strip $-\pi\leq\beta\leq\pi$ to the entire $\mathfrak{r}$ plane, the left (resp. right) boundary of the $\beta$ strip is sent to the bottom (resp. top) part of the negative real axis in the $\mathfrak{r}$ plane. We can hence consider the function $\tilde{p}_{m,n}(\mathfrak{r})=\tilde{P}_{m,n}(\beta(\mathfrak{r}))$. It is clearly analytic everywhere in the $\mathfrak{r}$ plane, with the possible exception of the negative real axis that can possibly be a branch cut. However, since $\tilde{P}_{m,n}(\beta)$ is $2\pi$-periodic, its value on the right and left boundary is the same, and hence $\tilde{p}_{m,n}(\mathfrak{r})$ is continuous across the negative real axis, so there is no cut there and $\mathfrak{r}=0$ is hence not a branch point (but can still be a pole).

Moreover, the exponential growth\footnote{To be rigorous here, one should impose the growth restriction of the type chosen not only on the function $\hat W$ but also on all its derivatives.} of $R$ and $T$ ensures that $D_{m,n}$ also has an exponential growth at infinity, and hence, it is clear that $\tilde{P}_{m,n}(\beta(\mathfrak{r}))$ also has exponential growth. Noting that under the same $\mathfrak{r}$ mapping, $\beta=+i\infty$ (resp. $\beta=-i\infty$) is sent to $0$ (resp. $\infty$). We conclude that $\tilde{p}_{m,n}(\mathfrak{r})$ grows no faster than a power of $\tau$ at $\infty$ and $0$. Since $\mathfrak{r}=0$ is not a branch point, $\tilde{p}_{m,n}(\mathfrak{r})$ should grow like $1/|\mathfrak{r}|^{\mathfrak{m}}$ for an integer power $\mathfrak{m}$ say. 

Hence the function $\mathfrak{r}^\mathfrak{m}\tilde{p}_{m,n}(\mathfrak{r})$ is entire and has power growth at infinity. According to the extended Liouville theorem, this has to be a polynomial. And hence $\tilde{p}_{m,n}(\mathfrak{r})$ is a rational function. Finally, we can conclude that the ratio of two generalised Wronsky determinants is a rational function. More precisely, we can show that
\begin{align*}
f(\beta)&=-\frac{D_{0,2}(\beta)}{D_{0,1}(\beta)}=\frac{-1}{\cos(\beta)} \frac{\tilde{P}_{0,2}(\beta)}{\tilde{P}_{0,1}(\beta)}=\frac{-2}{\mathfrak{r}+1/\mathfrak{r}} \frac{\tilde{p}_{0,2}(\mathfrak{r})}{\tilde{p}_{0,1}(\mathfrak{r})}, \\
g(\beta)&=\frac{D_{1,2}(\beta)}{D_{0,1}(\beta)}=\frac{1}{\cos(\beta)} \frac{\tilde{P}_{1,2}(\beta)}{\tilde{P}_{0,1}(\beta)}=\frac{2}{\mathfrak{r}+1/\mathfrak{r}} \frac{\tilde{p}_{1,2}(\mathfrak{r})}{\tilde{p}_{0,1}(\mathfrak{r})},
\end{align*}
are rational functions of $\mathfrak{r}$. The boundedness  of the coefficients $f$ and $g$ follows from the 
  exponential behaviour of the solutions.

The reasoning above has been based on the fact that the Wronsky determinant $D_{0,1}$ is not identically equal to zero. One can see that this happens if and only if 
$T$ obeys a linear homogeneous ODE of the first order with a $2\pi$-periodic coefficient. Indeed, if $T$ obeys such ODE, then so does $R$, and then $D_{0,1}=0$. If $D_{0,1}=0$, then $T$ obeys the ODE 
\begin{align}
T'-\frac{R'}{R} T=0, \label{eq:1storderODE}
\end{align}
which is a homogeneous 1st order linear ODE with coefficient $F(\beta)=-\tfrac{R'(\beta)}{R(\beta))}$. Using the 1D stencil equation, this coefficient can be shown to be $2\pi$-periodic:
\begin{align}
F(\beta+2\pi)=-\frac{R'(\beta+2\pi)}{R(\beta+2\pi)}\underset{\left( \ref{eq7036} \right)}{=}-\frac{T'(\beta)-\lambda R'(\beta)}{T(\beta)-\lambda R(\beta)}\underset{\left( \ref{eq:1storderODE} \right)}{=}-\frac{\frac{R'(\beta)}{R(\beta)}T(\beta)-\lambda R'(\beta)}{T(\beta)-\lambda R(\beta)}=F(\beta).
\end{align}
This cannot happens since in this case, we would get $T (2 \pi + \beta)  =  \Upsilon \, T (\beta)$ for some constant $\Upsilon$, and this cannot be true because of the singularity structure of~$T$.

Let us finish by proving the symmetry relations (\ref{eq7033a}). The condition 
(\ref{eq7032b}) implies that $R(\beta) = T(-\beta)$, and hence (\ref{eq7033a}) follows from 
(\ref{eq7034a}) and (\ref{eq7035}). {\link{\qedsymbol}}

\subsection{Proof of Proposition~\ref{prop:minimal}} \label{appB_2}

According to the periodicity of $f$ and $g$, their symmetry properties 
(\ref{eq7033a}), the position of (regular)
singular points, and the fact that no other singularities can occur due to the equation being minimal, the coefficients $f$ and $g$ can be rewritten as 
\begin{align}
%f(\beta) = \frac{p_1 (\sin \beta)}{\cos \beta},
f(\beta) &= \frac{1}{\cos \beta} \sum_{n=1}^{N_f} a^{(f)}_n \sin(n \beta)=\frac{\sin \beta}{\cos \beta}\sum_{n=1}^{N_f} a^{(f)}_n U_{n-1}(\cos \beta),\label{eq:Un}\\
\qquad 
%g(\beta) = \frac{p_2 (\cos \beta)}{\cos^2 \beta},
g(\beta) &= \frac{1}{\cos \beta} \sum_{n=0}^{N_g} a^{(g)}_n \cos(n \beta)=\frac{1}{\cos \beta} \sum_{n=0}^{N_g} a^{(g)}_n T_n( \cos \beta),
\label{eq7054bb}
\end{align}
for some positive integers $N_{f,g}$ and constants $a_n^{(f,g)}$, where $T_n$ and $U_n$ are the Chebyshev polynomials of first and second kind respectively. Since $f$ and $g$ must remain bounded away from $\beta=\pm \pi/2$, (\ref{eq:Un}) and (\ref{eq7054bb}) have to take the form

%where $p_1$ and $p_2$ are some polynomials; $p_1$ contains only odd powers 
%of the argument. 
%Since the coefficients should not grow faster than a constant, 
\begin{align}
f(\beta) &= \frac{c \sin \beta}{\cos \beta},
&
g(\beta) &= \frac{a + b \cos \beta}{\cos \beta},
\label{eq7054bb2}
\end{align}
for some constants $a$, $b$ and $c$.
For the resulting equation (\ref{eq7033}) to have exponents $(0, 1/2)$ at the point $\beta = \pi/2$ (i.e.\ it will have a regular solution and a solution which is a product of a regular function and 
$\sqrt{\cos \beta}$ near this point), it is necessary that $c = -1/2$ (see e.g.\ \cite{Olver}). {\link{\qedsymbol}}

\subsection{Proof of Proposition~\ref{prop:monodromy}}
\label{appB_3}
Let us pick $T(\beta) = \psi_{1,2} (\beta-\pi/2)$ and prove that it obeys all conditions of Proposition~\ref{prop:stencilsepvar}. Let us start by proving that $T$ obeys the symmetry condition~\ref{hyp:4th9}. 
For this, consider the function $T^{\S} (\beta) \equiv T(2\pi - \beta)$. 
The function $T^{\S} (\beta)$ obeys equation (\ref{eq7033min}) due to the 
symmetry of the coefficients of (\ref{eq7033min}). It is regular 
at the point $\beta = \pi/2$ because $\mathsf{n}_{1,2}=0$, thus, its expansion in the basis $B_{\pi/2}$
does not contain the branching term. Thus, $T^{\S}(\beta) = \mathcal{Q} T(\beta)$
for some constant $\mathcal{Q}$. Finally, taking $\beta = \pi$ we 
prove that $\mathcal{Q} = 1$. 

Because of this symmetry, the condition~\ref{hyp:1th9} is hence obeyed by construction. The condition~\ref{hyp:3th9} follows from the fact that $0$ and 
$\infty$ are regular singular points of (\ref{eq:fuchsian1}).

Let us finally prove that condition~\ref{hyp:2th9} of Proposition~\ref{prop:stencilsepvar}
is satisfied.
For this, we consider $R(\beta) = T (\beta + 2\pi)$ and analyse the behaviour of the 
functions $T(\beta)$ and $R(\beta)$ as $\beta \to i \infty$ (i.e.\ in the 
upper half-plane far from the real axis). On the one hand, 
due to the $2\pi$-periodicity of the coefficients of (\ref{eq7033}) and 
by the construction of $R$, $R$ and $T$ are two independent solutions to our ODE, and $R(\beta+2\pi)$ is also a solution. Hence $R(\beta+2\pi)$ is a linear combination of $R$ and $T$, and we can write
\begin{equation}
\left( \begin{array}{c}
T(\beta + 2\pi) \\
R(\beta + 2\pi)
\end{array}\right)
=
\left( \begin{array}{cc}
0 & 1 \\
s & -\lambda 
\end{array}\right) 
\left( \begin{array}{c}
T(\beta) \\
R(\beta)
\end{array}\right)
\label{eq:proofs1}
\end{equation} 
for some constants $s$ and $\lambda$. The condition~\ref{hyp:2th9} 
will be proven if we establish that  $s = 1$, i.e\ we want (\ref{eq7034}) to be satisfied. 

On the other hand, according to the theory of Fuchsian ODEs, 
in the upper half-plane the space of the solutions of (\ref{eq7033})
has a basis, whose terms  
can be expressed as formal series
\begin{equation}
F(\beta) = e^{\kappa \beta} \left(1 + \sum_{l = 1}^{\infty}
p_l e^{i \beta l} \right)
\label{eq:proofs6}
\end{equation}
for some $p_l$. 
The parameter $\kappa$ can be defined from a 
{\em characteristic equation\/} obtained from (\ref{eq7033})
by taking the leading terms of the coefficients in the upper half-plane as $\beta\to+i\infty$: 
\begin{equation}
\kappa^2 - \frac{i}{2} \kappa + b = 0.
\label{eq:proofs2}
\end{equation}
The roots of this equation are as follows:
\begin{align}
\kappa_{1,2} &= \frac{i}{4} \pm i \sqrt{b+1/16}, & \kappa_1+\kappa_2=i/2  . 
\label{eq:proofs3}
\end{align}
The two roots of this equation correspond to two 
functions $F_{1,2} (\beta)$ forming the solution basis. 

Note that, for (\ref{eq7034}) to be satisfied, 
the values $\kappa$ should also obey the equation
\begin{equation}
e^{4 \pi \kappa} + \lambda e^{2 \pi \kappa} - 1 = 0.
\label{eq:proofs5}
\end{equation}
Note that this equation defines the values $\kappa$ only up to 
a term equal to $2\pi i n$. 
Comparing (\ref{eq:proofs3}) and (\ref{eq:proofs5})
obtain
\begin{equation}
\lambda = - 2 i \cos(2\pi \sqrt{b + 1/16}). 
\label{eq:proofs5a}
\end{equation}
%{\bf Raphael, please, could you  check this?} OK it's correct
Due to (\ref{eq:proofs6}), the transformation matrix for the basis $F_{1,2}$ is obviously diagonal:
\begin{equation}
\left( \begin{array}{c}
F_1(\beta + 2\pi) \\
F_2(\beta + 2\pi)
\end{array}\right)
=
\left( \begin{array}{cc}
e^{2\pi \kappa_1} & 0 \\
0  & e^{2\pi \kappa_2}
\end{array}\right) 
\left( \begin{array}{c}
F_1(\beta) \\
F_2(\beta)
\end{array}\right). 
\label{eq:proofs4}
\end{equation} 
Since the vector of unknowns in (\ref{eq:proofs4}) is a linear transformation of 
the vector in (\ref{eq:proofs1}), the determinants of the 
transformation matrices should be equal:
\[
{\rm det}\left[ \left( \begin{array}{cc}
0 & 1 \\
s & -\lambda 
\end{array}\right)\right] 
=
{\rm det}\left[ \left( \begin{array}{cc}
e^{2\pi \kappa_1} & 0 \\
0  & e^{2\pi \kappa_2}
\end{array}\right)\right],
\]
leading to 
\[
-s = \exp \{ 2\pi (\kappa_1 + \kappa_2) \} = \exp\{\pi i \} = -1 .
\]
Thus, $s = 1$, equation (\ref{eq7036}) is valid for some $\lambda$ (given by (\ref{eq:proofs5a})), and
(\ref{eq7034}) is valid by construction of $R$. Finally, for the connection contours shown in figure~\ref{fig:pathbetaplane} the functional equation of condition~\ref{hyp:2th9} is valid. {\link{\qedsymbol}}

\end{document}